\documentclass[a4paper, 11pt, reqno]{amsart}

\usepackage{amsmath, amssymb, amsthm, mathrsfs, hyperref, color, enumerate, tikz, tikz-cd, mathtools, mathdots}
\usepackage[capitalize]{cleveref}
\usetikzlibrary{calc, matrix, shapes, decorations.pathreplacing}
\usetikzlibrary{intersections}

\usepackage{verbatim}

\newif\ifshowtikz

\let\oldtikzcd\tikzcd
\let\oldendtikzcd\endtikzcd

\renewenvironment{tikzcd}{%
    \ifshowtikz\expandafter\oldtikzcd%
    \else\comment%
    \fi
}{%
    \ifshowtikz\oldendtikzcd%
    \else\endcomment%
    \fi
}

\let\oldtikzpicture\tikzpicture
\let\oldendtikzpicture\endtikzpicture

\renewenvironment{tikzpicture}{%
    \ifshowtikz\expandafter\oldtikzpicture%
    \else\comment%
    \fi
}{%
    \ifshowtikz\oldendtikzpicture%
    \else\endcomment%
    \fi
}

\showtikztrue

\usepackage{booktabs, array}
\newcolumntype{C}{>{$}c<{$}}
\renewcommand{\P}{\mathbb{P}}
\newcommand{\R}{\mathbb{R}}
\newcommand{\C}{\mathbb{C}}
\newcommand{\Z}{\mathbb{Z}}
\newcommand{\Cone}{\operatorname{Cone}}
\newcommand{\gr}{\operatorname{gr}}
\newcommand{\Tw}{\operatorname{Tw}}
\newcommand{\Perf}{\operatorname{Perf}}
\newcommand{\Hom}{\operatorname{Hom}}
\newcommand{\Ext}{\operatorname{Ext}}
\newcommand{\End}{\operatorname{End}}
\renewcommand{\hom}{\operatorname{hom}}

\newcommand{\Proj}{\operatorname{Proj}}
\newcommand{\Split}{\operatorname{\Pi}}
\newcommand{\diff}{\mathrm{d}}
\newcommand{\fac}{w}
\newcommand{\fact}{\check{w}}
\newcommand{\obj}[2]{{{}^{#2}K_{#1}}}
\newcommand{\vc}[2]{{{}^{#2}V_{#1}}}
\newcommand{\vcpr}[2]{{{}^{#2}V^\mathrm{pr}_{#1}}}
\newcommand{\w}{\mathbf{w}}
\newcommand{\wt}{\check{\mathbf{w}}}
\renewcommand{\phi}{\varphi}
\newcommand{\eps}{\varepsilon}
\newcommand{\xt}{\check{x}}
\newcommand{\yt}{\check{y}}

\theoremstyle{plain}
\newtheorem{mthm}{Theorem}
\newtheorem{mconj}{Conjecture}
\newtheorem{thm}{Theorem}[section]
\newtheorem{lem}[thm]{Lemma}
\newtheorem{prop}[thm]{Proposition}

\theoremstyle{remark}
\newtheorem{ex}[thm]{Example}
\newtheorem{rmk}[thm]{Remark}

\title{Homological Berglund--H\"ubsch mirror symmetry for curve singularities}
\address{Department of Mathematics\\ University College London\\ Gower Street\\ London\\ WC1E  6BT}
\email{m.habermann.17@ucl.ac.uk}
\address{St John's College\\ Cambridge\\ CB2 1TP}
\email{j.smith@dpmms.cam.ac.uk}
\author{Matthew Habermann and Jack Smith}

\begin{document}
\begin{abstract}
Given a two-variable invertible polynomial, we show that its category of maximally-graded matrix factorisations is quasi-equivalent to the Fukaya--Seidel category of its Berglund--H\"ubsch transpose.  This was previously shown for Brieskorn--Pham and $D$-type singularities by Futaki--Ueda.  The proof involves explicit construction of a tilting object on the B-side, and comparison with a specific basis of Lefschetz thimbles on the A-side.
\end{abstract}
\maketitle

\section{Introduction}

\subsection{Berglund--H\"ubsch mirror symmetry}
\label{BHMirrorSymmetry}

Suppose $f : \C^n \rightarrow \C$ is a polynomial with an isolated singularity at the origin.  This paper is concerned with two $A_\infty$-categories one can naturally associate to such an object: the Fukaya--Seidel category $\mathcal{F}(f)$ as defined in \cite{SeidelBook} (the `A-model'), which categorifies the intersections of vanishing cycles in the Milnor fibre of a Morsification of $f$; and the (dg-)category $\mathrm{mf}(\C^n, f)$ of matrix factorisations of $f$ (the `B-model').  Mirror symmetry predicts that for certain pairs of singularities the A-model of one is equivalent to the B-model of the other (after taking some symmetries into account), and vice versa, and our main result confirms this conjecture for curve singularities ($n=2$).

More precisely, given an $n\times n$ matrix $A$ with non-negative integer entries $a_{ij}$, one can define a polynomial $\w$ in $\C[x_1, \dots, x_n]$ by
\[
\w = \sum_{i=1}^n \prod_{j=1}^n x_j^{a_{ij}}.
\]
The \emph{Berglund--H\"ubsch transpose} of $\w$, denoted $\wt$, is then defined by
\[
\wt = \sum_{i=1}^n \prod_{j=1}^n \xt_j^{a_{ji}}.
\]
A polynomial is called \emph{invertible} if it is quasi-homogeneous and of the form $\w$ for some matrix $A$ with non-zero determinant, such that both $\w$ and $\wt$ have isolated singularities at the origin.

Quasi-homogeneity means that there exist positive integral weights $d_1, \dots, d_n$ and $h$ such that
\[
\w(t^{d_1}x_1, \dots, t^{d_n}x_n) = t^h\w(x_1, \dots, x_n)
\]
for all $t$ in $\C^*$.  The \emph{maximal symmetry group} $\Gamma_\w$ of $\w$ is defined by
\[
\Gamma_\w = \{(t_1, \dots, t_n, t_{n+1}) \in (\C^*)^{n+1} : \w(t_1x_1, \dots, t_nx_n) = t_{n+1}\w(x_1, \dots, x_n)\}.
\]
This group acts on $\C^n$ in the obvious way, and we consider the category $\mathrm{mf}(\C^n, \Gamma_\w, \w)$ of matrix factorisations which are equivariant with respect to this group action.  This is equivalent to considering \emph{graded} matrix factorisations with respect to the maximal grading group for which $\w$ is homogeneous, namely the abelian group $L$ freely generated by elements $\vec{x}_1, \dots, \vec{x}_n$ (the degrees of $x_1, \dots, x_n$ respectively) and $\vec{c}$ (the degree of $\w$) modulo the relations
\[
\sum_{j=1}^n a_{ij}\vec{x}_j = \vec{c} \quad \text{for all } i.
\]
The prediction of mirror symmetry is then:

\begin{mconj}
\label{MainConjecture}
For any invertible polynomial $\w$ there is a quasi-equivalence of pretriangulated $A_\infty$-categories
\[
\mathrm{mf}(\C^n, \Gamma_\w, \w) \simeq \mathcal{F}(\wt).
\]
\end{mconj}

\begin{rmk}
Our Fukaya categories are all implicitly completed with respect to cones.
\end{rmk}

This conjecture appears in \cite{FutakiUedaBrieskornPhamProceedings}, which explains some of the background on mirror symmetry for Landau--Ginzburg models.  See also \cite[Conjecture 1.2]{LekiliUeda}, and references therein.  The underlying construction of mirror pairs via the transpose operation originated with Berglund--H\"ubsch \cite{BerglundHubsch}, and was later extended by Krawitz \cite{Krawitz}, who replaced the $\Gamma_\w$ on the left-hand side of \cref{MainConjecture} with a subgroup.  This requires the introduction of a `transpose' group on the right-hand side, but to make this precise one would need a rigorous definition of an orbifold Fukaya--Seidel category which is currently out of reach \cite[Problem 3]{FutakiUedaBrieskornPhamProceedings}.

Recall that the \emph{derived category of singularities} of a stack $X_0$ is defined to be the quotient
\[
D^b_\mathrm{sing}(X_0) \coloneqq D^b(X_0) / \Perf(X_0)
\]
of the derived category of coherent sheaves on $X_0$ by the category of perfect complexes (those complexes quasi-isomorphic to complexes of vector bundles).  Orlov \cite[Theorem 39]{Orlov09} showed that when $X_0$ is a hypersurface in a regular scheme, its singularity category can be expressed in terms of matrix factorisations of the defining equation.  This can be extended to stacks \cite[Proposition 3.19]{PolishchukVaintrobSingularityCategoriesForStacks} and in our setting we obtain an equivalence of triangulated categories
\begin{equation}
\label{HMFStacksingEquivalence}
\mathrm{HMF}(\C^2, \Gamma_\w, \w) \rightarrow D^b_\mathrm{sing}([\w^{-1}(0) / \Gamma_\w]),
\end{equation}
where $\mathrm{HMF}(\C^2, \Gamma_\w, \w)$ denotes the cohomology category of $\mathrm{mf}(\C^2, \Gamma_\w, \w)$.  \Cref{MainConjecture} therefore relates the algebraic geometry of the singularity $\w$ to the symplectic topology of the singularity $\wt$.

Our main result is:

\begin{mthm}
\label{Thm1}
\Cref{MainConjecture} holds when $n=2$, i.e.~for curve singularities.
\end{mthm}

As a by-product of our proof we also show:

\begin{mthm}[{\cite[Conjecture 1.4, $n=2$]{LekiliUeda}}]
\label{Thm2}
For every two-variable invertible polynomial $\w$ the category $\mathrm{mf}(\C^2, \Gamma_\w, \w)$ has a tilting object, meaning an object $\mathcal{E}$ satisfying $\End^i(\mathcal{E}) = 0$ for all $i\neq 0$ and such that $\hom^\bullet(\mathcal{E}, X) \simeq 0$ implies $X \cong 0$.
\end{mthm}


\subsection{Proof outline}
\label{sscOutline}

Invertible polynomials have been classified Kreuzer--Skarke \cite{KreuzerSkarke} and are known to be Thom--Sebastiani sums of \emph{atomic} polynomials of the following three types:
\begin{itemize}
\item Fermat, or type $A_{p-1}$: $\w = x^p$
\item chain: $\w = x_1^{p_1}x_2 + \dots + x_{n-1}^{p_{n-1}}x_n + x_n^{p_n}$
\item loop: $\w = x_1^{p_1}x_2 + \dots + x_{n-1}^{p_{n-1}}x_n + x_n^{p_n}x_1$.
\end{itemize}

\begin{ex}
A sum of Fermat polynomials is called \emph{Brieskorn--Pham}, and \cref{MainConjecture} was established for these polynomials, for all values of $n$, by Futaki--Ueda \cite{FutakiUedaBrieskornPhamProceedings,FutakiUedaBrieskornPhamJournal}.
\end{ex}

\begin{ex}
The $D_k$ singularity corresponds to the polynomial $x_1^2x_2 + x_2^{k-1}$ of chain type.  Futaki and Ueda also proved the conjecture for these singularities \cite{FutakiUedaD} (where the $D_k$ polynomial is on the A-side), as well as for Thom--Sebastiani sums of Brieskorn--Pham and type $D$ polynomials.
\end{ex}

We restrict attention to $n=2$, and use variables $x$ and $y$ rather than $x_i$, and $p$ and $q$ in place of $p_i$.  By the above classification we need to deal with the following cases:
\begin{itemize}
\item Brieskorn--Pham: $\w = x^p+y^q$, $\wt = \xt^p+\yt^q$
\item chain: $\w = x^py + y^q$, $\wt = \xt^p + \xt\yt^q$
\item loop: $\w = x^py + xy^q$, $\wt = \xt^p\yt + \xt\yt^q$.
\end{itemize}
We treat all three families in a uniform way, and obtain new proofs of the results of Futaki--Ueda for the two-variable Brieskorn--Pham and type $D$ (chain, $q=2$) singularities.  We shall always assume that $p$ and $q$ are at least $2$.  In the Brieskorn--Pham and chain cases these inequalities are necessary in order for the origin to be a critical point of both $\w$ and $\wt$, whilst if $p$ or $q$ is $1$ in the loop case then $\w$ and $\wt$ can be reduced to $x^2+y^2$ and $\xt^2+\yt^2$ by a change of variables.

The general strategy of proof is familiar: we match up explicit collections of generators on the two sides.  Concretely, on the A-side we compute the directed $A_\infty$-category $\mathcal{A}$  associated to a basis of vanishing cycles in the Milnor fibre of $\wt$.  Seidel \cite[Theorem 18.24]{SeidelBook} famously showed that after taking twisted complexes we obtain a quasi-equivalence
\[
\Tw \mathcal{A} \rightarrow \mathcal{F}(\wt),
\]
and readers unfamiliar with Fukaya--Seidel categories can take this as a definition of $\mathcal{F}(\wt)$.  The number of vanishing cycles in the basis, i.e.~the Milnor number of the singularity, is given by
\[
(p-1)(q-1)
\]
in the Brieskorn--Pham case,
\[
pq-p+1 = (p-1)(q-1) + (q-1) + 1
\]
in the chain case, and
\[
pq = (p-1)(q-1) + (p-1) + (q-1) + 1
\]
in the loop case.  These quantities, and the reasons for expressing them in this way, will fall out of our computations.

Meanwhile, on the $B$-side we identify a collection of objects in $\mathrm{mf}(\C^2, \Gamma_\w, \w)$ whose corresponding full subcategory $\mathcal{B}$ is quasi-equivalent to $\mathcal{A}$.  Since the matrix factorisation category is already pretriangulated we obtain a functor
\[
\Tw \mathcal{B} \rightarrow \mathrm{mf}(\C^2, \Gamma_\w, \w),
\]
and by a generation result (see \cref{PVGeneration,GenerationRefs}) this becomes a quasi-equivalence after taking the idempotent completion.  Our calcuations will actually show that the objects in $\mathcal{B}$ form a full exceptional collection so by \cite[Remark 5.14]{SeidelBook} the categories are in fact already idempotent complete.  Putting everything together we obtain a chain of quasi-equivalences
\[
\mathcal{F}(\wt) \simeq \Tw \mathcal{A} \simeq \Tw \mathcal{B} \simeq \mathrm{mf}(\C^2, \Gamma_\w, \w),
\]
proving \cref{Thm1}.  The sum of the objects in $\mathcal{B}$ gives the tilting object of \cref{Thm2}.

The choice of generators on the B-side is fairly natural; the main difficulty in proving \cref{Thm1} is to construct a Morsification and basis of vanishing paths for $\wt$ such that the category $\mathcal{A}$ built from the corresponding vanishing cycles matches up with $\mathcal{B}$.  In order to do this systematically we make a preliminary perturbation of $\wt$ by subtracting $\eps\xt\yt$ for small positive real $\eps$.  This has Morse critical points but not, in general, distinct critical values---following a suggestion of Yank\i~Lekili, we call this a \emph{resonant} Morsification.  The central fibre is nodal and upon passing to a nearby regular fibre the nodes are smoothed to thin necks, each supporting a vanishing cycle as the waist curve.  These cycles naturally pair up with the B-side generators supported along components of $\w^{-1}(0)$.

Understanding the remaining vanishing cycles, which are mirror to sheaves supported at the origin in $\w^{-1}(0)$, requires most of the work.  There is an obvious `real' vanishing cycle, and by acting by roots of unity on the $\xt$- and $\yt$-coordinates we obtain curves which are almost the other vanishing cycles.  The problem is that they live in different regular fibres, and carrying them to the same fibre requires explicit analysis of the parallel transport equation on the thin neck regions.  The resulting vanishing paths overlap each other, so we carefully perturb them to reduce to a small set of transverse intersections, and then eliminate these intersections by large deformations of the paths which do not affect the vanishing cycles.  Finally we modify the vanishing cycles by Hamiltonian perturbations to resolve the remaining ambiguities in their intersection pattern.

We end this discussion by pointing out recent work of Hirano and Ouchi \cite{HiranoOuchi}, which constructs semi-orthogonal decompositions of matrix factorisation categories for sums of polynomials which are only partially decoupled (non-Thom--Sebastiani).  In particular, this gives an approach to understanding the B-model for chain polynomials, and Hirano--Ouchi show that in this case the category has a full exceptional collection whose size matches the Milnor number of the Berglund--H\"ubsch transpose, providing further evidence for \cref{MainConjecture} in higher dimensions.  Shortly after the present paper appeared on arXiv, Aramaki and Takahashi gave an explicit full exceptional collection for chain polynomials, and showed that the Euler characteristics of the morphism complexes match the intersection form for a specific choice of vanishing cycles on the mirror, proving the chain case of \cref{MainConjecture} at the level of Grothendieck groups \cite[Corollary 3.8]{AramakiTakahashi}.

\subsection{Structure of the paper}

We first consider the case of loop polynomials in detail, describing the B-model in \cref{BModel} and the A-model in \cref{AModel}, culminating in proofs of \cref{Thm2} and \cref{Thm1} (in the loop case) respectively.  In \cref{BModelChain,AModelChain} we describe the minor modifications needed to deal with chain polynomials, and finally in \cref{BrieskornPham} we summarise the further modifications needed for Brieskorn--Pham polynomials.  We emphasise that these modifications are essentially just simplifications of the argument---the general approach is identical and all of the ingredients are contained in the loop case.

\subsection{Acknowledgements}

The authors are indebted to Yank\i~Lekili for suggesting this project, for many useful discussions, and for valuable feedback.  We are also grateful to the anonymous referee for helpful comments and suggestions.  JS would like to thank Jonny Evans and Michael Wong for their interest in this work.  MH is supported by the Engineering and Physical Sciences Research Council [EP/L015234/1], The EPSRC Centre for Doctoral Training in Geometry and Number Theory (The London School of Geometry and Number Theory), University College London.  JS is supported by EPSRC grant [EP/P02095X/1].


\section{B-model for loop polynomials}
\label{BModel}

\subsection{Graded matrix factorisations}
\label{sscMatrixFacs}

Our goal in this section is to understand the category $\mathrm{mf}(\C^2, \Gamma_\w, \w = x^py + xy^q)$ of equivariant matrix factorisations for the loop polynomial.  Recall that here $p$ and $q$ are assumed to be at least $2$.  We begin by briefly reviewing the definition, following \cite{FutakiUedaD}.  As mentioned in \cref{BHMirrorSymmetry}, we shall encode equivariance as respect for the grading by the abelian group $L$ freely generated by elements $\vec{x}$, $\vec{y}$ and $\vec{c}$ modulo the relations
\[
p\vec{x} + \vec{y} = \vec{x} + q\vec{y} = \vec{c}.
\]
Equivalently, $L$ is the quotient of $\Z^2$ by the subgroup generated by $(p-1, 1-q)$: the elements $\vec{x}$, $\vec{y}$ and $\vec{c}$ correspond to $(1, 0)$, $(0, 1)$ and $(p, 1) = (1, q)$ respectively.  Note that the quotient $L / \Z\vec{c}$ is isomorphic to $\Z/(pq-1)$, generated by $\vec{x}$ or equivalently by $\vec{y} = -p\vec{x}$.

Let $S$ denote the $L$-graded algebra $\C[x, y]$ in which $x$ has degree $\vec{x}$ and $y$ has degree $\vec{y}$.  The polynomial $\w = x^py+xy^q$ is a homogeneous element of degree $\vec{c}$, and we write $R$ for the quotient $S/(\w)$.  Given an $L$-graded $R$- or $S$-module $M$, and an element $l$ of $L$, we write $M(l)$ for the module obtained from $M$ by shifting the degree of each element downwards by $l$.  We shall use subscripts to denote $L$-graded pieces, so that $M(l)_i = M_{i+l}$ and $S_{\vec{x}} = k\cdot x$ for example.  Note that our notation for $R$ and $S$ is consistent with Futaki--Ueda \cite{FutakiUedaD}, but opposite to that of Dyckerhoff \cite{DyckerhoffCompactGenerators}.

By an \emph{$L$-graded matrix factorisation} of $W$ we mean a sequence
\[
K^\bullet = ( \cdots \rightarrow K^i \xrightarrow{k^i} K^{i+1} \xrightarrow{k^{i+1}} K^{i+2} \rightarrow \cdots )
\]
of $L$-graded free $S$-modules of finite rank such that $K^\bullet[2]$ is identified with $K^\bullet(\vec{c})$---i.e.~$K^{i+2}$ with $K^{i}(\vec{c})$ and $k^{i+2}$ with $k^i(\vec{c})$ for all $i$---and such that under these identifications the composition of any two consecutive maps in the sequence is multiplication by $\w$.  A finitely generated $L$-graded $R$-module $M$ gives rise to a matrix factorisation by taking a free resolution, which eventually stabilises (becomes $2$-periodic to the left, up to shifting the $L$-grading by $\vec{c}$ every two terms), then extending this $2$-periodic part indefinitely to the right, and replacing the free $R$-modules by the corresponding free $S$-modules; see \cite[Sections 2.1 and 2.2]{DyckerhoffCompactGenerators}.  This is the \emph{stabilisation} of $M$.

The set of $L$-graded matrix factorisations forms a $\Z$-graded dg-category $\mathrm{mf}(\C^2, \Gamma_\w, \w)$ as follows: $\hom^i(K^\bullet, H^\bullet)$ comprises sequences $(f^\bullet : K^\bullet \rightarrow H^\bullet[i])$ satisfying $f^\bullet[2] = f^\bullet(\vec{c})$, the differential
\[
\diff : \hom^i(K^\bullet, H^\bullet) \rightarrow \hom^{i+1}(K^\bullet, H^\bullet)
\]
is given by \cite[Definition 2.1]{DyckerhoffCompactGenerators}, namely
\[
\diff f = h \circ f - (-1)^i f \circ k,
\]
and composition is component-wise.  We shall write $\Hom^i$ for the degree $i$ cohomology of $\hom^\bullet$.

Finitely generated $L$-graded $R$-modules correspond to coherent sheaves on the stack $[\w^{-1}(0) / \Gamma_\w]$ and this gives a natural equivalence between $D^b_\mathrm{sing}([\w^{-1}(0) / \Gamma_\w])$ and the derived category of singularities of graded $R$-modules
\[
D^b_\mathrm{sing}(\gr R) \coloneqq D^b(\gr R) / \Perf(\gr R),
\]
where $\mathrm{Perf}$ now refers to complexes of projective modules ($D^b(\gr R)$ is the usual derived category of finitely-generated $L$-graded $R$-modules).  The equivalence \eqref{HMFStacksingEquivalence} then becomes an equivalence
\begin{equation}
\label{HMFsingEquivalence}
\mathrm{HMF}(\C^2, \Gamma_\w, \w) \rightarrow D^b_\mathrm{sing}(\gr R).
\end{equation}
Stabilisation of a module gives an inverse to this equivalence, and we will frequently switch between talking about matrix factorisations, modules, and sheaves on $[\w^{-1}(0) / \Gamma_\w]$.

\subsection{The basic objects}

The stack $[\w^{-1}(0) / \Gamma_\w]$ has three components: the lines $x=0$ and $y=0$ and the curve $x^{p-1}+y^{q-1} = 0$.  For brevity we will denote $x^{p-1}+y^{q-1}$ by $\fac$, so that $\w = xy\fac$.  The matrix factorisations corresponding to the structure sheaves of these components are
\begin{gather*}
\obj{x}{}^\bullet = ( \cdots \rightarrow S(-\vec{c}) \xrightarrow{y\fac} S(-\vec{x}) \xrightarrow{x} S \rightarrow \cdots ),
\\ \obj{y}{}^\bullet = ( \cdots \rightarrow S(-\vec{c}) \xrightarrow{x\fac} S(-\vec{y}) \xrightarrow{y} S \rightarrow \cdots ),
\end{gather*}
and
\[
\obj{\fac}{}^\bullet = ( \cdots \rightarrow S(-\vec{c}) \xrightarrow{xy} S(-\vec{c}+\vec{x}+\vec{y}) \xrightarrow{\fac} S \rightarrow \cdots )
\]
respectively, obtained by applying the stabilisation procedure of \cref{sscMatrixFacs} to the $L$-graded $R$-modules $R/(x)$, $R/(y)$ and $R/(\fac)$.  In each case, the third of the three terms written lies in degree $0$ within the sequence.  We will be particularly interested in the shifts
\[
\obj{x}{i} = \obj{x}{}((i+1-p)\vec{x}) \quad \text{for } i=1, \dots, p-1
\]
and
\[
\obj{y}{j} = \obj{x}{}((j+1-q)\vec{y}) \quad \text{for } j=1, \dots, q-1
\]
of the $\obj{x}{}$ and $\obj{y}{}$ objects.

The unique singular point of the stack is the origin, and the other main objects we will be interested in are $L$-grading shifts of the structure sheaf of its fattenings.  Specifically, for $i=1, \dots, p-1$ and $j=1, \dots, q-1$ let $\obj{0}{i,j}^\bullet$ be the matrix factorisation
\begin{center}
\begin{tikzcd}[row sep=7ex, column sep=10ex]
S(\vec{x}+\vec{y}) \arrow{r}{y^j} \arrow{dr}[near start, outer sep=-2pt]{-x^i} \ar[d, phantom, description, "\cdots\hskip7ex\bigoplus\phantom{\hskip7ex\cdots}"]
&  S(\vec{x}+(j+1)\vec{y}) \arrow{r}{xy^{q-j}} \arrow{dr}[near start]{x^i} \ar[d, phantom, description, "\bigoplus"]
& S(\vec{c}+\vec{x}+\vec{y}) \ar[d, phantom, description, "\phantom{\cdots\hskip7ex}\bigoplus\hskip7ex\cdots"]
\\ S(-\vec{c}+(i+1)\vec{x}+(j+1)\vec{y}) \arrow{ur}[near start, outer sep=-1pt]{x^{p-i}y} \arrow{r}{xy^{q-j}}
& S((i+1)\vec{x}+\vec{y}) \arrow{ur}[near start, outer sep=-1pt]{-x^{p-i}y} \arrow{r}{y^j}
& S((i+1)\vec{x}+(j+1)\vec{y})
\end{tikzcd}
\end{center}
corresponding to the $R$-module $R((i+1)\vec{x}+(j+1)\vec{y})/(x^i, y^j)$.  This stabilisation can be computed by starting with the obvious first steps of an $R$-free resolution
\[
R(\vec{x}+(j+1)\vec{y}) \oplus R((i+1)\vec{x}+\vec{y}) \xrightarrow{(\begin{smallmatrix} x^i & y^j \end{smallmatrix})} R((i+1)\vec{x}+(j+1)\vec{y})
\]
and extending by hand.  Shifts of the object $R/(x,y)$ appear in the work of Dyckerhoff \cite[Section 4.1]{DyckerhoffCompactGenerators}, who calls it $k^\mathrm{stab}$ ($k$ is the ground field), and Seidel \cite[Section 11]{SeidelGenusTwo}; here the resolution is described abstractly as a Koszul complex.  A concrete example close to our setting is given by Futaki--Ueda \cite[Section 4]{FutakiUedaD}.

\begin{rmk}
\label{rmkOrlov}
The motivation for considering these objects is Orlov's result \cite[Theorem 40(ii)]{Orlov09}, extended to the present setting in \cite[Theorem B.2]{HiranoOuchi}, which gives a semi-orthogonal decomposition
\[
D^b_\mathrm{sing}(\gr R) = \langle \mathcal{C},  D^b(Y) \rangle,
\]
where $Y$ is the projectivised stack $[(\w^{-1}(0) \setminus \{0\}) / \Gamma_\w]$ and $\mathcal{C}$ is the full subcategory on a certain collection of grading shifts of the structure sheaf of the origin.  In our case $Y$ is the zero locus of $\w$ inside the weighted projective line $\Proj S$, and it consists of three points: one is smooth and its structure sheaf corresponds to $\obj{\fac}{}$; the other two are stacky and their structure sheaves, twisted by characters of their isotropy groups, are given by the $\obj{x}{i}$ and $\obj{y}{j}$.  We replace $\mathcal{C}$ by the related category $\langle \obj{0}{i,j} \rangle$ to give the right pattern of morphisms.
\end{rmk}

Let $\mathcal{B}$ be the full $A_\infty$-subcategory of $\mathrm{mf}(\C^2, \Gamma_\w, \w)$ on the $pq-1$ objects
\[
\{\obj{0}{i,j}, \obj{x}{i}[3], \obj{y}{j}[3], \obj{\fac}{}[3]\}_{i=1, \dots, p-1;\ j=1, \dots, q-1}.
\]
The reason for the shifts is so that all morphisms turn out to have degree $0$.  In \cref{BMorphisms1,BMorphisms2,BMorphisms4,BMorphisms5,BMorphisms6} we compute the morphisms between these objects in the cohomology category $\mathrm{HMF}(\C^2, \Gamma_\w, \w)$.  The reader willing to take these calculations on trust may skip immediately to \cref{BMorphismsTotal}, where we assemble the results and deduce that $\mathcal{B}$ is quasi-equivalent to a specific quiver algebra with relations, with formal $A_\infty$-structure.  Then in \cref{BGeneration} we address the issue of generation, and show that $\Tw\mathcal{B} \rightarrow \mathrm{mf}(\C^2, \Gamma_\w, \w)$ is a quasi-equivalence.  We conclude that the sum of the objects in $\mathcal{B}$ is a tilting object for $\mathrm{mf}(\C^2, \Gamma_\w, \w)$, proving \cref{Thm2} for loop polynomials.


\subsection{Morphisms between $\obj{x}{}$'s, between $\obj{y}{}$'s, and from $\obj{\fac}{}$ to itself}
\label{BMorphisms1}

We wish to compute morphisms in the cohomology category $\mathrm{HMF}(\C^2, \Gamma_\w, \w)$, and a priori this involves taking the cohomology of the morphism complexes described in \cref{sscMatrixFacs}.  Thinking of matrix factorisations as stabilisations of $R$-modules, this corresponds to computing module $\Ext$'s by (projectively) resolving both the domain and codomain.  One might expect the latter to be unnecessary, and Buchweitz \cite[Section 1.3, Remark (a)]{Buchweitz} showed that this is indeed the case: given $L$-graded $R$-modules $M$ and $M'$ with stabilisations $K$ and $K'$, we have
\[
\Hom^\bullet_{\mathrm{HMF}(\C^2, \Gamma_\w, \w)} (K, K') \cong \mathrm{H}^\bullet \left(\Hom_{\gr R} (K \otimes_S R, M')\right).
\]
The $\Hom$ on the right-hand side is taken component-wise on the complex $K \otimes_S R$.

For any $l$ in $L$ we therefore have
\[
\Hom^\bullet(\obj{x}{}, \obj{x}{}(l)) \cong \mathrm{H}^\bullet \big( \cdots \rightarrow (R/(x))_l \xrightarrow{x} (R/(x))_{l+\vec{x}} \xrightarrow{y\fac} (R/(x))_{l+\vec{c}} \rightarrow \cdots \big),
\]
where the first of the three written terms now lives in degree $0$ (we have taken $L$-graded module homomorphisms from $\obj{x}{}^\bullet \otimes_S R$ into $R(l)/(x)$).  This gives
\[
\Hom^{2m}(\obj{x}{}, \obj{x}{}(l)) \cong (R/(x,yw))_{m\vec{c}+l}
\]
for any integer $m$, whilst $\Hom^{2m+1}(\obj{x}{}, \obj{x}{}(l)) = 0$.

One can easily compute a basis of $\Hom^{2m}$ by hand in this situation, but since we will repeatedly make similar arguments we record the following general facts relating gradings and divisibility:

\begin{lem}
\label{GradingIdeal}
Suppose that $a$ and $b$ are integers satisfying $a \leq p-1$ and $b \leq q-1$, and that $s$ is an element of $S$ (or $R$) which is homogeneous modulo $\vec{c}$, of degree $a\vec{x}+b\vec{y} \mod \vec{c}$.  Then:
\begin{enumerate}[(i)]
\item\label{gr1} $s$ lies in the ideal $(x^a, y^{q-1+b}) \cap (x^{p-1+a}, y^b)$.
\item\label{gr2} If $a \leq p-2$ then $s$ also lies in $(x^a, y^{q+b})$.
\item\label{gr3} If $b \leq q-2$ then $s$ also lies in $(x^{p+a}, y^b)$.
\end{enumerate}
\end{lem}
\begin{proof}
Assume $a \leq p-1$ and $b \leq q-1$, and let $x^uy^v$ be a monomial in $s$, so that
\begin{equation}
\label{eqMonomialDegree}
(u-a)\vec{x} + (v-b)\vec{y} \equiv 0 \mod \vec{c}.
\end{equation}
We claim first that $u \geq a$ or $v \geq q-1+b$, so suppose for contradiction that neither holds.  Then
\[
-(p-1) \leq u-a \leq -1 \quad \text{and} \quad -(q-1) \leq v-b \leq q-2,
\]
so $(u-a) - p(v-b)$ is non-zero (by reducing modulo $p$) and lies strictly between $\pm (pq-1)$.  Substituting $\vec{y} = -p\vec{x} \mod \vec{c}$ into \eqref{eqMonomialDegree} tells us that $(u-a)-p(v-b) \equiv 0 \mod{(pq-1)}$, which gives the desired contradiction, and we deduce that $u \geq a$ or $v \geq q-1+b$, and hence that $s$ lies in $(x^a, y^{q-1+b})$.  The other arguments are analogous.
\end{proof}

\begin{lem}
\label{GradingZero}
Suppose $s$ is an element of degree $0 \mod \vec{c}$.  Then the non-constant terms in $s$ lie in the ideal $(x^{pq-1}, x^py, xy^q, y^{pq-1})$.
\end{lem}
\begin{proof}
Let $x^uy^v$ be a non-constant monomial in $s$.  If $u=0$ (or $v=0$) then one easily obtains $v \geq pq-1$ (respectively $u \geq pq-1$), so suppose now that $u$ and $v$ are both positive.  We have $u-pv \equiv 0 \mod{(pq-1)}$, so if $u < p$ then we must have $u-pv \leq -(pq-1)$ and hence $v \geq q$.
\end{proof}

From these we conclude:

\begin{lem}\label{KIorthonormal}
In $\mathrm{HMF}(\C^2, \Gamma_\w, \w)$ the objects $\obj{x}{1}, \dots, \obj{x}{p-1}$ are exceptional (the endomorphisms of each are just the scalar multiples of the identity) and pairwise orthogonal.
\end{lem}
\begin{proof}
By the above computation the morphisms from $\obj{x}{i}$ to $\obj{x}{I}$ are given by the elements of $R/(x, y\fac)$ of degree $(I-i)\vec{x} \mod \vec{c}$.  If $I-i > 0$ then \cref{GradingIdeal}(\ref{gr2}) tells us that all such elements lie in $(x, y^q) = (x, y\fac)$, and hence vanish in the quotient.  If $I-i < 0$ then the same argument applies but using \cref{GradingIdeal}(\ref{gr1}) instead, after rewriting the degree as $(p+I-i)\vec{x} + \vec{y} \mod \vec{c}$.  Finally, if $I=i$ then \cref{GradingZero} tells us that only constants survive in the quotient.
\end{proof}

Likewise we have:

\begin{lem}\label{KIIorthonormal}
The objects $\obj{y}{1}, \dots \obj{y}{q-1}$ are exceptional and pairwise orthogonal.\hfill$\qed$
\end{lem}

Similar calculations give
\[
\Hom^{2m}(\obj{\fac}{}, \obj{\fac}{}) \cong (R/(xy, \fac))_{m\vec{c}}
\]
and $\Hom^{2m+1}(\obj{\fac}{}, \obj{\fac}{}) = 0$, so by \cref{GradingZero} we deduce:

\begin{lem}\label{KIIIends}
The object $\obj{\fac}{}$ is exceptional.\hfill$\qed$
\end{lem}


\subsection{Morphisms between $\obj{x}{}$'s, $\obj{y}{}$'s, and $\obj{\fac}{}$}
\label{BMorphisms2}

For all $l$ and $m$ we have
\[
\Hom^{2m+1}(\obj{x}{}, \obj{y}{}(l)) \cong (R/(x,y))_{m\vec{c}+l+\vec{x}}
\]
whilst $\Hom^{2m}(\obj{x}{}, \obj{y}{}(l)) = 0$.  This gives:

\begin{lem}\label{KIKIIorthonormal}
Each $\obj{x}{i}$ is orthogonal to each $\obj{y}{j}$.
\end{lem}
\begin{proof}
For morphisms $\obj{x}{i}$ to $\obj{y}{j}$ we need to show that there are no (non-zero) elements in $R/(x,y)$ of degree $(1-i)\vec{x}+j\vec{y} \mod \vec{c}$, and this follows from \cref{GradingIdeal}(\ref{gr1}).  The argument is similar for morphisms in the opposite direction.
\end{proof}

Analogous computations yield
\[
\Hom^{2m+1}(\obj{x}{}(l), \obj{\fac}{}) \cong (R/(x, w))_{m\vec{c}-l+\vec{x}}
\]
and $\Hom^{2m}(\obj{x}{}(l), \obj{\fac}{}) = 0$ for all $l$ and $m$.  Similarly
\[
\Hom^{2m+1}(\obj{\fac}{}, \obj{x}{}(l)) \cong (R/(x,\fac))_{(m+1)\vec{c}+l-\vec{y}}
\]
whilst $\Hom^{2m}(\obj{\fac}{}, \obj{x}{}(l)) = 0$.

Likewise
\begin{gather*}
\Hom^{2m+1}(\obj{y}{}(l), \obj{\fac}{}) \cong (R/(y, \fac))_{m\vec{c}-l+\vec{y}},
\\ \Hom^{2m+1}(\obj{\fac}{}, \obj{y}{}(l)) \cong (R/(y, \fac))_{(m+1)\vec{c}+l-\vec{x}},
\\ \Hom^{2m}(\obj{y}{}(l), \obj{\fac}{}) = \Hom^{2m}(\obj{\fac}{}, \obj{y}{}(l)) = 0.
\end{gather*}

In particular:

\begin{lem}
For each $i$ and $j$ the objects $\obj{x}{i}$ and $\obj{y}{j}$ are orthogonal to $\obj{\fac}{}$.
\end{lem}
\begin{proof}
For orthogonality of $\obj{x}{i}$ and $\obj{\fac}{}$ we need to check that elements of degree $(p-i)\vec{x}$ or $(i+1)\vec{x}$ modulo $\vec{c}$ lie in the ideal $(x, \fac) = (x, y^{q-1})$.  This follows immediately from \cref{GradingIdeal}(\ref{gr1}), except that for $(i+1)\vec{x}$ with $i=p-1$ we must first rewrite the degree as $\vec{x}+(q-1)\vec{y} \mod \vec{c}$.  The argument for $\obj{y}{j}$ is analogous.
\end{proof}

\begin{rmk}
These results match our expectation from \cref{rmkOrlov} that the objects $\obj{x}{i}$, $\obj{y}{j}$ and $\obj{\fac}{}$ correspond to structure sheaves of disjoint points in the projective stack $Y$, twisted by characters of their isotropy groups, and hence should be exceptional and orthogonal.
\end{rmk}


\subsection{Morphisms between $\obj{\fac}{}$ and $\obj{0}{}$'s}
\label{BMorphisms4}

We now fix $(i,j)$ with $1 \leq i \leq p-1$ and $1 \leq j \leq q-1$, and see that
\[
\Hom^\bullet(\obj{\fac}{}, \obj{0}{i,j}) \cong \mathrm{H}^\bullet \big( \cdots \rightarrow (R/(x^i, y^j))_l \xrightarrow{\fac} (R/(x^i, y^j))_{l+\vec{c}-\vec{x}-\vec{y}} \xrightarrow{xy} (R/(x^i, y^j))_{l+\vec{c}} \rightarrow \cdots \big),
\]
where $l=(i+1)\vec{x}+(j+1)\vec{y}$.  The terms in odd positions in the complex have degree $i\vec{x}+j\vec{y} \mod \vec{c}$ so by \cref{GradingIdeal}(\ref{gr1}) they lie in $(x^i, y^j)$ and therefore vanish.  The same holds in even positions after rewriting the degree $
(i+1)\vec{x}+(j+1)\vec{y} \mod \vec{c}$ as $(i+1-p)\vec{x}+j\vec{y} \mod \vec{c}$.

In the other direction, $\Hom^\bullet(\obj{0}{i,j}, \obj{\fac}{})$ is the cohomology of the complex
\begin{center}
\begin{tikzcd}[row sep=7ex, column sep=10ex]
(R/(\fac))_{-\vec{c}-\vec{x}-\vec{y}} \arrow{r}{xy^{q-j}} \arrow{dr}[near start, outer sep=-2pt]{-x^{p-i}y} \ar[d, phantom, description, "\cdots\hskip7ex\bigoplus\phantom{\hskip7ex\cdots}"]
& (R/(\fac))_{-\vec{x}-(j+1)\vec{y}} \arrow{r}{y^j} \arrow{dr}[near start]{x^{p-i}y} \ar[d, phantom, description, "\bigoplus"]
& (R/(\fac))_{-\vec{x}-\vec{y}} \ar[d, phantom, description, "\phantom{\cdots\hskip7ex}\bigoplus\hskip7ex\cdots"]
\\ (R/(\fac))_{-(i+1)\vec{x}-(j+1)\vec{y}} \arrow{ur}[near start, outer sep=-1pt]{x^i} \arrow{r}{y^j}
& (R/(\fac))_{-(i+1)\vec{x}-\vec{y}} \arrow{ur}[near start, outer sep=-1pt]{-x^i} \arrow{r}{xy^{q-j}}
& (R/(\fac))_{\vec{c}-(i+1)\vec{x}-(j+1)\vec{y}}
\end{tikzcd}
\end{center}
For each $m$, $\Hom^{2m}(\obj{0}{i,j}, \obj{\fac}{})$ is therefore given by
\[
\operatorname{Ker} \begin{pmatrix} xy^{q-j} & x^i \\ -x^{p-i}y & y^j \end{pmatrix} \quad \text{modulo} \quad \operatorname{Im} \begin{pmatrix} y^j & -x^i \\ x^{p-i}y & xy^{q-j} \end{pmatrix}
\]
in $(R/(\fac))_{(m-1)\vec{c}-\vec{x}-\vec{y}} \oplus (R/(\fac))_{m\vec{c}-l}$.  Ignoring gradings for a second, this kernel is spanned by those $f$, $g$ in $R$ such that there exist $h$, $k$ in $R$ with
\[
xy^{q-j}f+x^ig = (x^{p-1}+y^{q-1})h \quad \text{and} \quad -x^{p-i}yf+y^jg = (x^{p-1}+y^{q-1})k.
\]
Subtracting $x^i$ times the latter from $y^j$ times the former we see that $h=xh'$ and $k=yk'$ for some polynomials $h'$ and $k'$, and that $f=y^{j-1}h'-x^{i-1}k'$.  Plugging this back in gives $g=x^{p-i}h'+y^{q-j}k'$, so $\Hom^{2m}(\obj{0}{i,j}, \obj{\fac}{})$ is parametrised by
\[
\begin{pmatrix} y^{j-1} & -x^{i-1} \\ x^{p-i} & y^{q-j} \end{pmatrix} \begin{pmatrix} h' \\ k' \end{pmatrix} \quad \text{modulo} \quad \operatorname{Im} \begin{pmatrix} y^j & -x^i \\ x^{p-i}y & xy^{q-j} \end{pmatrix} \quad \text{(and modulo $w$)},
\]
with $h' \in R_{(m-2)\vec{c}+(q-j)\vec{y}}$ and $k' \in R_{(m-2)\vec{c}+(p-i)\vec{x}}$.  It is clear from this description that $h'$ and $k'$ only matter modulo $(y, w) = (y, x^{p-1})$ and $(x, w) = (x, y^{q-1})$, but $h'$ and $k'$ must lie in these ideals by \cref{GradingIdeal}(\ref{gr1}), so we conclude that $\Hom^{2m}(\obj{0}{i,j}, \obj{\fac}{})$ vanishes.

Similarly, $\Hom^{2m+1}(\obj{0}{i,j}, \obj{\fac}{})$ is parametrised by
\[
\begin{pmatrix} y^{q-j-1} & x^i \\ -x^{p-i-1} & y^j \end{pmatrix} \begin{pmatrix} h' \\ k' \end{pmatrix} \quad \text{modulo} \quad \operatorname{Im} \begin{pmatrix} xy^{q-j} & x^i \\ -x^{p-i}y & y^j \end{pmatrix} \quad \text{(and $w$)},
\]
with $h' \in R_{m\vec{c}-\vec{x}-(j+1)\vec{y}}$ and $k' \in R_{m\vec{c}-(i+1)\vec{x}-\vec{y}}$.  Obviously $k'$ can be eliminated and we're left with
\[
\Hom^{2m+1}(\obj{0}{i,j}, \obj{\fac}{}) \cong (R/(xy,\fac))_{(m-1)\vec{c}} \begin{pmatrix} y^{q-j-1} \\ -x^{p-i-1} \end{pmatrix},
\]
and by \cref{GradingZero} $(R/(xy,\fac))_{(m-1)\vec{c}}$ has only constants.  The upshot is:

\begin{lem}
In $\mathrm{HMF}(\C^2, \Gamma_\w, \w)$ the only morphisms between $\obj{\fac}{}$ and $\obj{0}{i,j}$ are from the latter to the former, spanned by $(y^{q-j-1}, -x^{p-i-1})$ in degree $3$ in the above complex.\hfill$\qed$
\end{lem}


\subsection{Morphisms between $\obj{x}{}$'s and $\obj{y}{}$'s and $\obj{0}{}$'s}
\label{BMorphisms5}

For each $I$ we have that $\Hom^\bullet (\obj{x}{I}, \obj{0}{i,j})$ vanishes since again the whole complex is zero by \cref{GradingIdeal}(\ref{gr1}).  Morphisms the other way are computed by the complex
\begin{center}
\begin{tikzcd}[row sep=7ex, column sep=10ex]
(R/(x))_{-2\vec{c}+I\vec{x}} \arrow{r}{xy^{q-j}} \arrow{dr}[near start, outer sep=-2pt]{-x^{p-i}y} \ar[d, phantom, description, "\cdots\hskip7ex\bigoplus\phantom{\hskip7ex\cdots}"]
& (R/(x))_{-\vec{c}+I\vec{x}-j\vec{y}} \arrow{r}{y^j} \arrow{dr}[near start]{x^{p-i}y} \ar[d, phantom, description, "\bigoplus"]
& (R/(x))_{-\vec{c}+I\vec{x}} \ar[d, phantom, description, "\phantom{\cdots\hskip7ex}\bigoplus\hskip7ex\cdots"]
\\ (R/(x))_{-\vec{c}+(I-i)\vec{x}-j\vec{y}} \arrow{ur}[near start, outer sep=-1pt]{x^i} \arrow{r}{y^j}
& (R/(x))_{-\vec{c}+(I-i)\vec{x}} \arrow{ur}[near start, outer sep=-1pt]{-x^i} \arrow{r}{xy^{q-j}}
& (R/(x))_{(I-i)\vec{x}-j\vec{y}}
\end{tikzcd}
\end{center}
All differentials vanish except $y^j$, which is injective, so we get
\begin{gather*}
\Hom^{2m}(\obj{0}{i,j}, \obj{x}{I}) \cong (R/(x, y^j))_{(m-2)\vec{c}+I\vec{x}},
\\ \Hom^{2m+1}(\obj{0}{i,j}, \obj{x}{I}) \cong (R/(x, y^j))_{(m-1)\vec{c}+(I-i)\vec{x}}.
\end{gather*}
The former is zero by \cref{GradingIdeal}(\ref{gr1}), whilst the latter is zero unless $I=i$, when it contains only constants, by the argument used in the proof of \cref{KIorthonormal}.  From this we get:

\begin{lem}
In $\mathrm{HMF}(\C^2, \Gamma_\w, \w)$ there are no morphisms from $\obj{x}{I}$ to $\obj{0}{i,j}$.  There are no morphisms in the other direction unless $I=i$, in which case the morphism space is spanned by $(0, 1)$ in degree $3$ in the above complex.  Similarly for morphisms between $\obj{y}{J}$ and $\obj{0}{i,j}$.\hfill$\qed$
\end{lem}


\subsection{Morphisms between $\obj{0}{}$'s}
\label{BMorphisms6}

The complex computing $\Hom^\bullet(\obj{0}{i,j}, \obj{0}{I,J})$ is
\begin{center}
\begin{tikzcd}[row sep=7ex, column sep=10ex]
(R/(x^I, y^J))_{-\vec{c}+I\vec{x}+J\vec{y}} \arrow{r}{xy^{q-j}} \arrow{dr}[near start, outer sep=-2pt]{-x^{p-i}y} \ar[d, phantom, description, "\cdots\hskip7ex\bigoplus\phantom{\hskip7ex\cdots}"]
& (R/(x^I, y^J))_{I\vec{x}+(J-j)\vec{y}} \arrow{r}{y^j} \arrow{dr}[near start]{x^{p-i}y} \ar[d, phantom, description, "\bigoplus"]
& (R/(x^I, y^J))_{I\vec{x}+J\vec{y}} \ar[d, phantom, description, "\phantom{\cdots\hskip7ex}\bigoplus\hskip7ex\cdots"]
\\ (R/(x^I, y^J))_{(I-i)\vec{x}+(J-j)\vec{y}} \arrow{ur}[near start, outer sep=-1pt]{x^i} \arrow{r}{y^j}
& (R/(x^I, y^J))_{(I-i)\vec{x}+J\vec{y}} \arrow{ur}[near start, outer sep=-1pt]{-x^i} \arrow{r}{xy^{q-j}}
& (R/(x^I, y^J))_{\vec{c}+(I-i)\vec{x}+(J-j)\vec{y}}
\end{tikzcd}
\end{center}
By \cref{GradingIdeal}(\ref{gr1}) all of the terms vanish except the bottom term in the even positions, giving
\begin{gather*}
\Hom^{2m}(\obj{0}{i,j}, \obj{0}{I,J}) \cong (R/(x^I, y^J))_{(I-i)\vec{x}+(J-j)\vec{y}},
\\ \Hom^{2m+1}(\obj{0}{i,j}, \obj{0}{I,J}) = 0.
\end{gather*}
If $I<i$ then we can rewrite $(I-i)\vec{x}+(J-j)\vec{y}$ as $(p+I-i)\vec{x}+(J-j+1)\vec{y}$ modulo $\vec{c}$ and apply \cref{GradingIdeal}(\ref{gr1}) to see that $\Hom^{2m}$ vanishes.  Likewise if $J<j$.

Now assume that $I\geq i$ and $J\geq j$.  By \cref{GradingIdeal}(\ref{gr1}), any element of degree $(I-i)\vec{x}+(J-j)\vec{y} \mod \vec{c}$ is divisible by $x^{I-i}y^{J-j}$ modulo $(x^I, y^J)$.  So we can rewrite $\Hom^{2m}$ as
\[
(R/(x^i, y^j))_0 \cdot x^{I-i}y^{J-j},
\]
and by \cref{GradingZero} the only surviving term is $\C \cdot x^{I-i}y^{J-j}$.  We deduce:

\begin{lem}
For all $(i,j)$ and $(I,J)$ we have that
\begin{flalign*}
&& \Hom^\bullet (\obj{0}{i,j}, \obj{0}{I,J}) &\cong \begin{cases} \C \cdot x^{I-i}y^{J-j} & \text{if } I\geq i \text{, } J \geq j \text{ and } \bullet=0 \\ 0 & \text{otherwise.} \end{cases} & \qed
\end{flalign*}
\end{lem}


\subsection{The total endomorphism algebra of the basic objects}
\label{BMorphismsTotal}

Combining the results of \cref{BMorphisms1,BMorphisms2,BMorphisms4,BMorphisms5,BMorphisms6} we see that in $\mathrm{HMF}(\C^2, \Gamma_\w, \w)$ the basic objects $\obj{x}{i}$, $\obj{y}{j}$, $\obj{\fac}{}$ and $\obj{0}{i,j}$ are all exceptional, and that the morphisms between distinct objects are spanned by:
\begin{itemize}
\item $(0, 1)$ in degree $3$ from each $\obj{0}{i,j}$ to $\obj{x}{i}$
\item $(0, 1)$ in degree $3$ from each $\obj{0}{i,j}$ to $\obj{y}{j}$
\item $(y^{q-j-1}, -x^{p-i-1})$ in degree $3$ from each $\obj{0}{i,j}$ to $\obj{\fac}{}$
\item $x^{I-i}y^{J-j}$ in degree $0$ from $\obj{0}{i,j}$ to $\obj{0}{I,J}$ whenever $I \geq i$ and $J \geq j$.
\end{itemize}
We immediately see that morphisms between the $\obj{0}{i,j}$ compose in the obvious way so that their total endomorphism algebra is the tensor product $A_{p-1} \otimes A_{q-1}$ of the path algebras of the $A_{p-1}$- and $A_{q-1}$-quivers (this is the path algebra of the obvious product quiver subject to the relations which say that the squares commute).  In fact, we have:

\begin{thm}
\label{LoopEndAlgebra}
The cohomology-level total endomorphism algebra of the objects $\obj{x}{i}[3]$, $\obj{y}{j}[3]$, $\obj{\fac}{}[3]$ and $\obj{0}{i,j}$ in $\mathcal{B}$ is the path algebra of the quiver-with-relations described in \cref{figLoopQuiver}, with all arrows living in degree zero.  In particular, $\mathcal{B}$ is a $\Z$-graded $A_\infty$-category concentrated in degree $0$, so is intrinsically formal.

\begin{figure}[ht]
\centering
\begin{tikzpicture}[blob/.style={circle, draw=black, fill=black, inner sep=0, minimum size=\blobsize}, arrow/.style={->, shorten >=6pt, shorten <=6pt}]
\def\blobsize{1.2mm}

\draw (0, 0) node[blob]{};
\draw (1, 0) node[blob]{};
\draw (2, 0) node[blob]{};
\draw (3.25, 0) node{\ $\cdots$};
\draw (4.5, 0) node[blob]{};
\draw (6, 0) node[blob]{};

\draw[arrow] (0, 0) -- (1, 0);
\draw[arrow] (1, 0) -- (2, 0);
\draw[arrow] (2, 0) -- (3, 0);
\draw[arrow] (3.5, 0) -- (4.5, 0);
\draw[arrow] (4.5, 0) -- (6, 0);

\draw[arrow] (0, 0) -- (0, 1);
\draw[arrow] (1, 0) -- (1, 1);
\draw[arrow] (2, 0) -- (2, 1);
\draw[arrow] (4.5, 0) -- (4.5, 1);

\begin{scope}[yshift=1cm]
\draw (0, 0) node[blob]{};
\draw (1, 0) node[blob]{};
\draw (2, 0) node[blob]{};
\draw (3.25, 0) node{\ $\cdots$};
\draw (4.5, 0) node[blob]{};
\draw (6, 0) node[blob]{};

\draw[arrow] (0, 0) -- (1, 0);
\draw[arrow] (1, 0) -- (2, 0);
\draw[arrow] (2, 0) -- (3, 0);
\draw[arrow] (3.5, 0) -- (4.5, 0);
\draw[arrow] (4.5, 0) -- (6, 0);

\draw[arrow] (0, 0) -- (0, 1);
\draw[arrow] (1, 0) -- (1, 1);
\draw[arrow] (2, 0) -- (2, 1);
\draw[arrow] (4.5, 0) -- (4.5, 1);
\end{scope}

\begin{scope}[yshift=3.5cm]
\draw (0, 0) node[blob]{};
\draw (1, 0) node[blob]{};
\draw (2, 0) node[blob]{};
\draw (3.25, 0) node{\ $\cdots$};
\draw (4.5, 0) node[blob]{};
\draw (6, 0) node[blob]{};

\draw[arrow] (0, 0) -- (1, 0);
\draw[arrow] (1, 0) -- (2, 0);
\draw[arrow] (2, 0) -- (3, 0);
\draw[arrow] (3.5, 0) -- (4.5, 0);
\draw[arrow] (4.5, 0) -- (6, 0);

\draw[arrow] (0, 0) -- (0, 1.5);
\draw[arrow] (1, 0) -- (1, 1.5);
\draw[arrow] (2, 0) -- (2, 1.5);
\draw[arrow] (4.5, 0) -- (4.5, 1.5);
\end{scope}

\begin{scope}[yshift=5cm]
\draw (0, 0) node[blob]{};
\draw (1, 0) node[blob]{};
\draw (2, 0) node[blob]{};
\draw (3.25, 0) node{\ $\cdots$};
\draw (4.5, 0) node[blob]{};
\draw (6, 0) node[blob]{};
\end{scope}

\draw (0, 2.37) node{$\vdots$};
\draw (1, 2.37) node{$\vdots$};
\draw (2, 2.37) node{$\vdots$};
\draw (4.5, 2.37) node{$\vdots$};
\draw (6, 2.37) node{$\vdots$};

\begin{scope}[yshift=2.5cm]
\draw[arrow] (0, 0) -- (0, 1);
\draw[arrow] (1, 0) -- (1, 1);
\draw[arrow] (2, 0) -- (2, 1);
\draw[arrow] (4.5, 0) -- (4.5, 1);
\end{scope}

\draw (3.25, 2.37) node{\ $\iddots$};
\draw[arrow] (4.5, 3.5) -- (6, 5);

\draw[line width=1mm, opacity=0.2] (-0.5, -0.5) rectangle (5, 4);
\draw[line width=1mm, opacity=0.2] (-0.5, 4.5) rectangle (5, 5.5);
\draw[line width=1mm, opacity=0.2] (5.5, -0.5) rectangle (6.5, 4);

\draw (-1.2, 2) node{$\obj{0}{i,j}$};
\draw (-1.25, 5) node{$\obj{x}{i}[3]$};
\draw (7.25, 2) node{$\obj{y}{j}[3]$};
\draw (6.6, 5) node{$\obj{\fac}{}[3]$};

\draw[->, dashed, rounded corners] (0.15, 3.72) -- (0.78, 3.72) -- (0.78, 4.85);
\draw[->, dashed, rounded corners] (1.15, 3.72) -- (1.78, 3.72) -- (1.78, 4.85);
\draw[->, dashed, rounded corners] (3.65, 3.72) -- (4.28, 3.72) -- (4.28, 4.85);
\draw[->, dashed, rounded corners] (4.72, 0.15) -- (4.72, 0.78) -- (5.85, 0.78);
\draw[->, dashed, rounded corners] (4.72, 2.65) -- (4.72, 3.28) -- (5.85, 3.28);

\draw (10.6, 2.25) node{\parbox{115pt}{\small \ \textbf{Relations:} \begin{enumerate}[(i)]\item Squares commute\item Dashed compositions vanish\end{enumerate}}};

\end{tikzpicture}
\caption{The quiver describing the category $\mathcal{B}$ for loop polynomials.\label{figLoopQuiver}}
\end{figure}
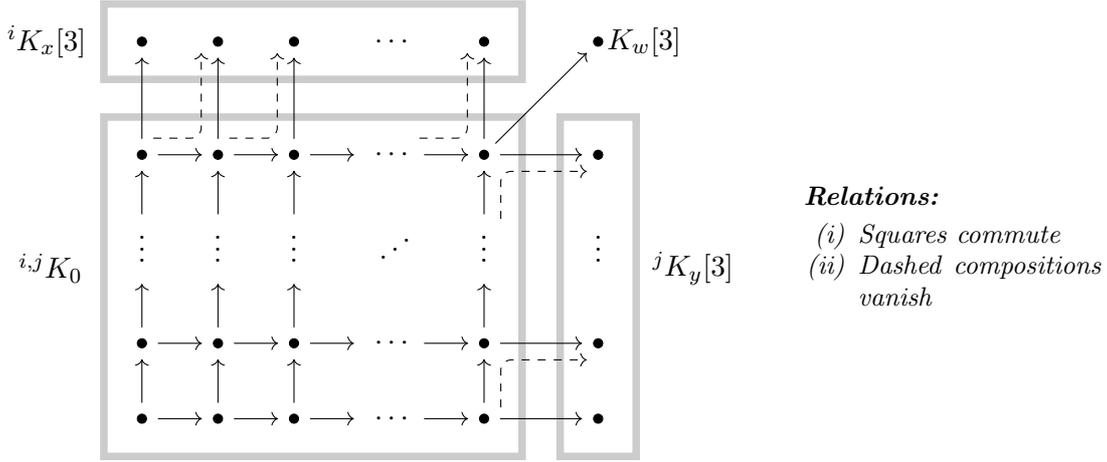

\end{thm}
\begin{proof}
To prove the cohomology statement we just need to check that the morphisms compose correctly, namely that for $I\geq i$ and $J\geq j$ the compositions
\begin{gather*}
\Hom^3(\obj{0}{I,J}, \obj{\fac}{}) \otimes \Hom^0(\obj{0}{i,j}, \obj{0}{I,J}) \rightarrow \Hom^3(\obj{0}{i,j}, \obj{\fac}{}),
\\ \Hom^3(\obj{0}{i,J}, \obj{x}{i}) \otimes \Hom^0(\obj{0}{i,j}, \obj{0}{i,J}) \rightarrow \Hom^3(\obj{0}{i,j}, \obj{x}{i}),
\\ \Hom^3(\obj{0}{I,j}, \obj{y}{j}) \otimes \Hom^0(\obj{0}{i,j}, \obj{0}{I,j}) \rightarrow \Hom^3(\obj{0}{i,j}, \obj{y}{j})
\end{gather*}
send generators to generators.  This is immediate from the explicit descriptions of the morphisms above after noting that the generator
\[
R((i+1)\vec{x}+(j+1)\vec{y})/(x^i, y^j) \xrightarrow{x^{I-i}y^{J-j}} R((I+1)\vec{x}+(J+1)\vec{y})/(x^I, y^J)
\]
of $\Hom^0(\obj{0}{i,j}, \obj{0}{I,J})$ induces the maps
\[
\begin{pmatrix} 1 & 0 \\ 0 & x^{I-i}y^{J-j} \end{pmatrix} \text{ in even degree and } \begin{pmatrix} y^{J-j} & 0 \\ 0 & x^{I-i} \end{pmatrix} \text{ in odd degree}
\]
between the matrix factorisations (the degree $3$ matrix is the only one we actually need).

The final claim, about the $A_\infty$-structure, follows from the fact that a directed algebra concentrated in degree zero is formal---there is no room for non-trivial higher $A_\infty$-operations.
\end{proof}

\subsection{Generation}
\label{BGeneration}

We have now computed the quasi-isomorphism type of the full $A_\infty$-subcategory $\mathcal{B} \subset \mathrm{mf}(\C^2, \Gamma_\w, \w)$ on the basic objects $\obj{x}{i}, \obj{y}{j}, \obj{\fac}{}, \obj{0}{i,j}$.  The goal of this subsection is to prove:

\begin{prop}
\label{BGenerates}
The functor
\[
\Split (\Tw \mathcal{B}) \rightarrow \Split(\mathrm{mf}(\C^2, \Gamma_\w, \w))
\]
is a quasi-equivalence, where $\Split$ denotes $A_\infty$- (or dg-) idempotent completion.
\end{prop}

\begin{rmk}
As mentioned in \cref{sscOutline}, the $\Split$'s can be removed from this statement (and this is what we need to prove \cref{Thm1}) using the fact that the objects in $\mathcal{B}$ form a full exceptional collection in $\Tw \mathcal{B}$, so that the category is already idempotent complete by \cite[Remark 5.14]{SeidelBook}.
\end{rmk}

For a triangulated category $\mathcal{C}$ and a collection $V$ of objects in $\mathcal{C}$, let $\langle V \rangle$ denote the smallest full triangulated subcategory of $\mathcal{C}$ which contains the objects in $V$ and is closed under isomorphism, and let superscript $\pi$ denote idempotent completion.  We'll say that $V$ \emph{split-generates} $\mathcal{C}$ if the functor $\langle V \rangle^\pi \rightarrow \mathcal{C}^\pi$ induced by the obvious inclusion of $\langle V \rangle$ in $\mathcal{C}$ is an equivalence.

The content of \cref{BGenerates} is that the set
\[
V = \{\obj{x}{i}, \obj{y}{j}, \obj{\fac}{}, \obj{0}{i,j}\}
\]
split-generates $\mathcal{C} = \mathrm{HMF}(\C^2, \Gamma_\w, \w)$.  The key to establishing this is the following application of a result of Polishchuk--Vaintrob:

\begin{lem}[{\cite[Proposition 2.3.1]{PolishchukVaintrobCFT}}]
\label{PVGeneration}
The category $\mathrm{HMF}(\C^2, \Gamma_\w, \w)$ is split-generated by the $L$-grading shifts of the stabilisation of the module $R/(x, y)$.\hfill$\qed$
\end{lem}

\begin{rmk}
\label{GenerationRefs}
The cited result is a simple modification of the non-equivariant case, previously obtained by several authors including Schoutens \cite{Schoutens}, Dyckerhoff \cite[Corollary 5.3]{DyckerhoffCompactGenerators}, Seidel \cite[Lemma 12.1]{SeidelGenusTwo} (building on work of Orlov \cite{Orlov11}), and Murfet \cite[Proposition A.2]{KellerMurfetVandenBergh}.
\end{rmk}
\begin{proof}[Proof of \cref{BGenerates}]
By \cref{PVGeneration} it suffices to show that under the equivalence \eqref{HMFsingEquivalence} the category $\langle V \rangle$ contains all of the $L$-grading shifts of $R/(x, y)$.  In other words, it is enough to prove that for all $l$ in $L$ the $L$-graded $R$-module $R(l)/(x, y)$ can be built from the objects
\[
R((i+1-p)\vec{x})/(x) \text{, } R((j+1-q)\vec{y})/(y) \text{, } R/(\fac) \text{, and } R((i+1)\vec{x}+(j+1)\vec{y})/(x^i, y^j)
\]
with $1 \leq i \leq p-1$ and $1 \leq j \leq q-1$, by taking cones and shifts (in the triangulated category sense, rather than in the $L$-grading).  Since $[2]$ is equivalent to $(\vec{c})$, we actually only need consider $l$ in a set of representatives of $L/\Z\vec{c}$.

For any $1 \leq i \leq p-1$ and $1 \leq j \leq q-1$ we have a morphism (of $L$-graded $R$-modules)
\begin{equation}
\label{Cone1}
R(i\vec{x}+j\vec{y})/(x^{i-1}, y^{j-1}) \xrightarrow{x} R((i+1)\vec{x}+j\vec{y})/(x^i, y^{j-1})
\end{equation}
whose cone---which is just the cokernel in this case---is the module $R((i+1)\vec{x}+j\vec{y})/(x, y^{j-1})$.  Both objects in \eqref{Cone1} lie in $V$ unless $i$ or $j$ is $1$, in which case the offending objects are zero, so we conclude that this cone lies in $\langle V \rangle$.  Similarly $R((i+1)\vec{x}+(j+1)\vec{y})/(x, y^j)$ is in $\langle V \rangle$, and hence
\[
R((i+1)\vec{x}+(j+1)\vec{y})/(x, y) \cong \Cone \big( R((i+1)\vec{x}+j\vec{y})/(x, y^{j-1}) \xrightarrow{y} R((i+1)\vec{x}+(j+1)\vec{y})/(x, y^j) \big)
\]
is also in $\langle V \rangle$.  This gives $(p-1)(q-1)$ of the $pq-1$ objects we need.

Next consider the extension
\[
0 \rightarrow R((i+1)\vec{x}+\vec{y})/(x) \xrightarrow{y^j} R((i+1)\vec{x}+(j+1)\vec{y})/(x) \rightarrow R((i+1)\vec{x}+(j+1)\vec{y})/(x, y^j) \rightarrow 0.
\]
The outer terms are in $\langle V \rangle$ (the first is $\obj{x}{i}[2]$ and the last is built from $R(a\vec{x}+b\vec{y})/(x,y)$ for $a=i+1$ and $b=2, 3, \dots, j+1$ by taking cones), so the middle term is in $\langle V \rangle$.  In particular, taking $j=q-1$ we see that
\[
R(i\vec{x})/(x) = R((i+1)\vec{x}+q\vec{y})[-2]/(x)
\]
lies in $\langle V \rangle$.  If $i$ is at least $2$ then $R(i\vec{x}+\vec{y})/(x) = \obj{x}{i-1}[2]$ is also in $\langle V \rangle$, and hence so is
\[
R(i\vec{x}+\vec{y})/(x, y) \cong \Cone \big( R(i\vec{x})/(x) \xrightarrow{y} R(i\vec{x}+\vec{y})/(x) \big).
\]
One can make a similar argument with the roles of $x$ and $y$ interchanged to construct $R(\vec{x}+j\vec{y})/(x, y)$ when $2 \leq j \leq p-1$.

So far we have thus seen that $R(a\vec{x}+b\vec{y})/(x, y)$ lies in $\langle V \rangle$ for $1 \leq a \leq p$ and $1 \leq b \leq q$, except for the cases $(a, b) = (1, 1)$, $(1, q)$ and $(p, 1)$.  If we can fill in these missing three cases (the latter two are in fact equivalent---both correspond to $R(\vec{c})/(x, y)$) then we will have constructed shifts of $R/(x,y)$ by representatives of each class in $L/\Z\vec{c}$, and will therefore be done.

To build $R(\vec{x}+\vec{y})/(x, y)$ note that it is the cokernel of
\[
R(\vec{x})/(x) \oplus R(\vec{y})/(y) \xrightarrow{(\begin{smallmatrix}y & x\end{smallmatrix})} R(\vec{x}+\vec{y})/(xy).
\]
The two summands in the domain were constructed above, whilst the codomain is $\obj{\fac}{}[1]$.  Finally, to get $R(\vec{c})/(x, y)$ observe that $R/(x,y)$ is the cokernel of
\begin{equation}
\label{Cone2}
R(-\vec{y})/(x, y^{pq-2}) \xrightarrow{y} R/(x, y^{pq-1}).
\end{equation}
The domain can be built from $R(-b\vec{y})/(x, y)$ for $b=1, \dots, pq-2$ by taking cones, and these objects are all (up to repeated applications of $[\pm2]$) ones that we have already constructed.  The codomain, meanwhile, is given by
\[
\Cone \big( R(-(p-1)\vec{c})/(x) \xrightarrow{y^{pq-1}} R/(x) \big),
\]
and the two terms inside the cone are $\obj{x}{p-1}[-2(p-1)]$ and $\obj{x}{p-1}$.  This means that both objects in \eqref{Cone2} lie in $\langle V \rangle$, and hence so does the cokernel $R/(x,y)$.  Shifting by $[2]$ gives the object $R(\vec{c})/(x, y)$ that we need.
\end{proof}

\begin{rmk}
\label{rmkGenVsSplit}
We proved that $\mathcal{B}$ generates $\mathrm{mf}(\C^2, \Gamma_\w, \w)$ by showing that it generates the objects $R(l)/(x, y)$, which \emph{split-}generate the category, and then invoking the fact that $\Tw \mathcal{B}$ is idempotent complete.  The $R(l)/(x,y)$ themselves cannot possibly generate (as opposed to split-generate), for the following reason: $\mathrm{mf}(\C^2, \Gamma_\w, \w)$ has a full exceptional collection of size $pq$, so its Grothendieck group is free of rank $pq$, whereas the span of the $R(l)/(x,y)$ has rank at most $|L/\Z\vec{c}| = pq-1$.
\end{rmk}

As a corollary of \cref{BGenerates}, we obtain:

\begin{thm}[\cref{Thm2}, loop polynomial case]
The object
\[
\mathcal{E} \coloneqq \bigg( \bigoplus_{\substack{i=1, \dots, p-1\\j=1, \dots, q-1}} \obj{0}{i,j} \bigg) \oplus \bigg( \bigoplus_{i=1}^{p-1} \obj{x}{i}[3] \bigg) \oplus \bigg( \bigoplus_{j=1}^{q-1} \obj{y}{j}[3] \bigg) \oplus \obj{w}{}[3]
\]
is a tilting object for $\mathrm{mf}(\C^2, \Gamma_\w, \w)$.
\end{thm}
\begin{proof}
We need to show that $\End^i(\mathcal{E}) = 0$ for all $i\neq 0$ and that $\hom^\bullet(\mathcal{E}, X) \simeq 0$ implies $X \cong 0$.
The first statement follows immediately from \cref{LoopEndAlgebra}, whilst the second is a consequence of \cref{BGenerates}: if $\hom^\bullet(\mathcal{E}, X) \simeq 0$ then there are no non-zero morphisms from $\langle V \rangle^\pi$ to $X$ in $\mathrm{HMF}(\C^2, \Gamma_\w, \w)$, which forces $X$ to be quasi-isomorphic to $0$.
\end{proof}


\section{A-model for loop polynomials}
\label{AModel}

\subsection{A resonant Morsification}

We are now interested in the polynomial $\wt = \xt^p\yt + \xt\yt^q$ as a map $\C^2 \rightarrow \C$.  To construct the category $\mathcal{A}$ we should Morsify $\wt$ by adding a small perturbation, fix a regular value $*$, then pick a \emph{distinguished basis of vanishing paths} $(\gamma_1, \dots, \gamma_N)$ in the base $\C$, where $\gamma_i$ is a smooth embedded path from $*$ the $i$th critical value.  We require that the $\gamma_i$ are pairwise disjoint except for their common initial point $\gamma_i(0) = *$, that the vectors $\dot{\gamma}_i(0)$ in $T_*\C$ are non-zero and distinct, and that the corresponding directions are in clockwise order as $i$ increases from $1$ to $N$ (we are free to choose the starting direction for this clockwise ordering).  We then consider the corresponding vanishing cycles in the fibre $\Sigma$ over $*$ (strictly we should take $\Sigma$ to be the Liouville completion of the Milnor fibre, but this is equivalent in our case), and define $\mathcal{A}$ to be the directed $A_\infty$-category on these cycles whose morphisms and compositions in the allowed direction are given by those in the compact Fukaya category $\mathcal{F}(\Sigma)$.  Note that we are free to modify the vanishing cycles by Hamiltonian isotopy in order to compute $\mathcal{A}$ up to quasi-equivalence.

In order to implement this, we first consider the perturbation
\[
\wt_\eps \coloneqq \wt - \eps \xt\yt = \xt\yt(\xt^{p-1}+\yt^{q-1}-\eps)
\]
of $\wt$, where $\eps$ is a small positive real number; in analogy with \cref{BModel} we shall denote $\xt^{p-1}+\yt^{q-1}$ by $\fact$.  We call this a resonant Morsification, since its critical points are Morse but the critical values are not all distinct.  In fact, the critical points fall into four types:
\begin{enumerate}[(i)]
\item\label{crit1} $\xt^{p-1}=\eps$, $\yt=0$
\item\label{crit2} $\xt=0$, $\yt^{q-1}=\eps$
\item\label{crit3} $\xt=\yt=0$
\item\label{crit4} $(\xt^{p-1}, \yt^{q-1}) = \frac{\eps}{pq-1}(q-1, p-1)$.
\end{enumerate}
The critical points of the types (\ref{crit1})--(\ref{crit3}) all lie over the critical value zero, whilst for type (\ref{crit4}) the critical value is $-\xt\yt \eps(p-1)(q-1)/(pq-1)$ so is non-zero and lies on the ray through $-\xt\yt$.

We denote the unique positive real critical point of type (\ref{crit4}) by $(\xt^+_\mathrm{crit}, \yt^+_\mathrm{crit})$, with corresponding critical value $c_\mathrm{crit}$ (this is \emph{negative} real).  Letting $\zeta$ and $\eta$ denote the roots of unity
\[
\label{RootsOfUnity}
\zeta = e^{2\pi i/(p-1)} \quad \text{and} \quad \eta = e^{2\pi i/(q-1)},
\]
the full set of type (\ref{crit4}) critical points is then given by
\[
\{(\zeta^l\xt^+_\mathrm{crit}, \eta^m\yt^+_\mathrm{crit}) : 0 \leq l \leq p-2 \text{, } 0 \leq m \leq q-2\}.
\]
The critical value corresponding to $(\zeta^l\xt^+_\mathrm{crit}, \eta^m\yt^+_\mathrm{crit})$ is $\zeta^l\eta^mc_\mathrm{crit}$, so there are $\gcd(p-1, q-1)$ critical points in each of these critical fibres.

We now fix our regular fibre $\Sigma$ to be $\wt_\eps^{-1}(-\delta)$ where $\delta$ is a positive real number much less than $\eps$ (in other words, we take $*=-\delta$).  The condition $\delta \ll \eps$ is to allow us to understand $\Sigma$ as a smoothing of $\wt_\eps^{-1}(0)$.  For the critical points of types (\ref{crit1})--(\ref{crit3}) we choose the vanishing path given by the straight line segment from $-\delta$ to $0$.  For the critical point $(\zeta^l\xt^+_\mathrm{crit}, \eta^m\yt^+_\mathrm{crit})$, meanwhile, we define the \emph{preliminary vanishing path} $\gamma_{l,m}$ by following the circular arc $-\delta e^{i\theta}$ as $\theta$ increases from $0$ to
\[
\theta_{l,m} \coloneqq 2\pi \left(\frac{l}{p-1}+\frac{m}{q-1}\right)
\]
and then following the radial straight line segment from $-\zeta^l\eta^m\delta$ to $\zeta^l\eta^mc_\mathrm{crit}$.  As the name suggests, we will later modify these preliminary vanishing paths (they currently do not form a distinguished basis since they intersect and overlap each other), but they serve an important intermediate role.

\Cref{figCriticalPoints} shows the critical values of $\wt_\eps$, the vanishing path for the type (\ref{crit1})--(\ref{crit3}) critical points, and the preliminary vanishing paths for $(l,m)=(0,0)$, $(1, 0)$ and $(1, 2)$, all in the case $(p,q)=(4,6)$.  We have slightly separated the arcs for clarity---really they both have radius $\delta$.
\begin{figure}[ht]
\centering
\begin{tikzpicture}[blob/.style={circle, draw=black, fill=black, inner sep=0, minimum size=\blobsize, line width=0.5mm}, critval/.style={cross out, draw=black, fill=black, inner sep=0, minimum size=\crosssize, line width=0.5mm}]
\def\blobsize{1mm}
\def\crosssize{1mm}
\def\masterDELTA{0.7}
\def\DELTA{\masterDELTA}
\def\r{3}

\draw (-\DELTA, 0) node[blob]{};
\draw (0, 0) node[critval]{};

\foreach \a in {1, ..., 15}
{
\draw ($-cos(360*\a/15)*(\r, 0)+sin(360*\a/15)*(0, \r)$) node[critval]{};
}

\def\l{1}
\def\m{2}
\def\theta{360*(\l/3+\m/5)}
\def\DELTA{\masterDELTA*0.97}
\draw ($-cos(\theta)*(\r, 0)-sin(\theta)*(0, \r)$) -- ($-cos(\theta)*(\DELTA, 0)-sin(\theta)*(0, \DELTA)$);
\draw ($-cos(\theta)*(\DELTA, 0)-sin(\theta)*(0, \DELTA)$) arc(\theta+180:180:\DELTA);
\draw ($-cos(\theta)*(\DELTA, 0)-sin(\theta)*(0, \DELTA)-(0, 0.1)$) node[anchor=south west]{\small$\gamma_{1, 2}$};

\def\l{1}
\def\m{0}
\def\theta{360*(\l/3+\m/5)}
\def\DELTA{\masterDELTA*1.05}
\draw ($-cos(\theta)*(\r, 0)-sin(\theta)*(0, \r)$) -- ($-cos(\theta)*(\DELTA, 0)-sin(\theta)*(0, \DELTA)$);
\draw ($-cos(\theta)*(\DELTA, 0)-sin(\theta)*(0, \DELTA)$) arc(\theta+180:180:\DELTA);
\draw ($-cos(\theta)*2*(\DELTA, 0)-sin(\theta)*2*(0, \DELTA)-(0.2, 0.1)$) node[anchor=north]{\small$\gamma_{1, 0}$};

\draw (-\r, 0) -- (-\DELTA, 0);
\draw (0, 0) -- (-\DELTA, 0);

\draw [decorate,decoration={brace,amplitude=5pt}] (-\r, 0.1) -- (-\DELTA, 0.1);
\draw ($(-\r/2, 0.15)+(-\DELTA/2, 0.15)$) node[anchor=south]{\small$\gamma_{0,0}$};

\draw (-\DELTA, 0) node[anchor=north east]{\small$-\delta$};
\draw (0, -0.02) node[anchor=north]{\small$0$};

\end{tikzpicture}
\caption{The critical values of $\wt_\eps$ (crosses), the vanishing path for critical value $0$, and three of the preliminary vanishing paths, when $(p,q)=(4,6)$.\label{figCriticalPoints}}
\end{figure}
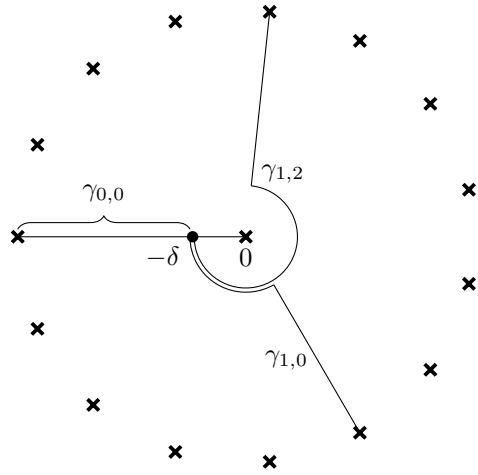
Note that $\theta_{l,m}$ may be greater than $2\pi$, in which case $\gamma_{l,m}$ covers more than a full circle, but these paths are difficult to indicate on a diagram.  Note also that different values of $(l, m)$ may give rise to different preliminary vanishing paths, even if the critical values are the same.

\subsection{The zero-fibre and its smoothing $\Sigma$}

The fibre of $\wt_\eps$ over zero has three components: the lines $\{\xt=0\}$ and $\{\yt=0\}$, and the smooth curve $\{\fact = \eps\}$.  Schematically the picture is as in \cref{figZeroFibre}.  The crosses denote transverse intersections between the components, and the dotted line where the planes appear to meet is to indicate that they are actually disjoint in $\C^2$ except for the intersection at the origin.
\begin{figure}[ht]
\centering
\begin{tikzpicture}[blob/.style={cross out, draw=black, fill=black, inner sep=0, minimum size=\blobsize, line width=0.5mm}, noblob/.style={inner sep=0, minimum size=0}]
\def\blobsize{1.5mm}

\begin{scope}[xshift=-0.2cm]
\draw (-2.5, -0.75) -- (1, 0.75) -- (8.5, 0.75) -- (5, -0.75) -- cycle;
\draw[name path global=plane1] (-1.5, -2) -- (2, -0.5) -- (2, 5.5) -- (-1.5, 4) -- cycle;
\draw[fill=white, opacity=0.8] (-1.5, -0.75) -- (5, -0.75) -- (8.5, 0.75) -- (2, 0.75) -- (2, 5.5) -- (-1.5, 4) -- cycle;
\draw[dotted] (-1.5, -0.75) -- (2, 0.75);
\end{scope}

\draw[fill=white, fill opacity=0.8] plot [smooth, tension=1] coordinates {(2.5, 4.5) (1, 3.9) (0, 3.5) (1, 3) (0.2, 2.5) (1, 2) (0, 1.5) (1, 1) (1.5, 0) (2, 1) (2.5, 0.2) (3, 1) (3.5, 0.2) (4, 1) (4.5, 0) (5.2, 1) (6.5, 2)};

\draw plot [smooth, tension=1] coordinates {(3.5, 4.5) (3.2, 3.9) (4.5, 4.5)};
\draw plot [smooth, tension=1] coordinates {(5.2, 4.5) (4.2, 3.1) (6.5, 4)};
\draw plot [smooth, tension=1] coordinates {(6.5, 3.5) (5, 2.2) (6.5, 2.5)};

\begin{scope}[xshift=3.5cm, yshift=2cm, rotate=-20]
\draw (0, 0) to [bend left=30] (1, 0);
\draw (-0.1, 0.075) to [bend right=10] (0, 0) to [bend right=20] (1, 0) to [bend right=10] (1.1, 0.075);
\end{scope}

\begin{scope}[xshift=1.5cm, yshift=2.2cm, rotate=-20]
\draw[name path global=curve1] (0, 0) to [bend left=30] (1, 0);
\draw[name path global=curve2] (-0.1, 0.075) to [bend right=10] (0, 0) to [bend right=20] (1, 0) to [bend right=10] (1.1, 0.075);

\draw[name intersections={of=plane1 and curve1}] (intersection-1) node[noblob](pt1){};
\draw[name intersections={of=plane1 and curve2}] (intersection-1) node[noblob](pt2){};
\draw (pt1) -- (pt2);
\end{scope}

\begin{scope}[xshift=2.2cm, yshift=3.3cm, rotate=-20]
\draw (0, 0) to [bend left=30] (1, 0);
\draw (-0.1, 0.075) to [bend right=10] (0, 0) to [bend right=20] (1, 0) to [bend right=10] (1.1, 0.075);
\end{scope}

\draw (-0.3, 0.8) node[anchor=east]{$\xt=0$};
\draw (5.8, 0.6) node[anchor=north]{$\yt=0$};
\draw (4.2, 2.55) node{$\fact=\eps$};

\draw (0, 1.5) node[blob]{};
\draw (0.2, 2.5) node[blob]{};
\draw (0, 3.5) node[blob]{};

\draw (0, 0) node[blob]{};

\draw (1.5, 0) node[blob]{};
\draw (2.5, 0.2) node[blob]{};
\draw (3.5, 0.2) node[blob]{};
\draw (4.5, 0) node[blob]{};

\draw [decorate,decoration={brace,amplitude=7pt,mirror}] (1.3, -0.05) -- (4.7, -0.05);
\draw (3, -0.2) node[anchor=north]{\small$p-1$};
\draw [decorate,decoration={brace,amplitude=7pt}] (-0.05, 1.3) -- (-0.05, 3.7);
\draw (-0.28, 2.45) node[anchor=east]{\small$q-1$};

\end{tikzpicture}
\caption{The fibre $\wt_\eps^{-1}(0)$ for loop polynomials.\label{figZeroFibre}}
\end{figure}
In $\Sigma$, each of the nodes is smoothed to a thin neck whose waist curve is the corresponding vanishing cycle.  We denote these vanishing cycles by $\vc{\yt\fact}{l}$, $\vc{\xt\fact}{m}$ and $\vc{\xt\yt}{}$ for $l=0, \dots, p-2$ and $m=0, \dots, q-2$, corresponding to critical points $(\zeta^l\eps^{1/(p-1)}, 0)$, $(0, \eta^m\eps^{1/(q-1)})$ and $(0, 0)$ respectively.

\begin{rmk}
We can compute the genus and number of punctures of $\Sigma$ as follows.  The punctures correspond to boundary components at infinity, where the defining equation looks like $\xt^p\yt+\xt\yt^q=0$.  The lines $\{\xt=0\}$ and $\{\yt=0\}$ each give rise to a boundary component, whilst $\{\xt^{p-1}+\yt^{q-1}=0\}$ gives $\gcd(p-1, q-1)$ components.  We deduce
\[
\# \text{ punctures of } \Sigma = \gcd(p-1,q-1)+2.
\]
The $pq$ vanishing cycles form a basis for $\mathrm{H}_1(\Sigma; \Z)$, whose rank is
\[
2g(\Sigma) + \# \text{ punctures} - 1,
\]
so we obtain
\[
g(\Sigma) = \frac{1}{2}\left(pq-\gcd(p-1,q-1)-1\right).
\]
\end{rmk}

If $\delta$ is chosen sufficiently small then the monodromy of parallel transport around the circle of radius $\delta$ is supported in small neighbourhoods of these $p+q-1$ curves, and is simply the product of the Dehn twists in them.  It is not \emph{strictly} true that the monodromy is supported in these neighbourhoods, but as explained in \cite[Section 19]{SeidelThesis} it can be made so by a small deformation of the fibration, which does not affect the categories and which we will not explicitly notate.  After deleting these neighbourhoods (and corresponding neighbourhoods in the other fibres) we may therefore trivialise the fibration $\wt_\eps$ over the disc of radius $\delta$, and identify each fibre with the curve $\Sigma'$ obtained from $\wt_\eps^{-1}(0)$ by removing neighbourhoods of the critical points marked in \cref{figZeroFibre}.  Equivalently, we may think of $\Sigma'$ as being obtained from $\Sigma$ by removing the neck regions.  Concretely, it consists of: a complex line (the $\xt$-axis) with small balls around the origin and the $(p-1)$th roots of $\eps$ removed; a complex line (the $\yt$-axis) with small balls around the origin and the $(q-1)$th roots of $\eps$ removed; and a $(p-1)(q-1)$-fold cover of the line $\{u+v=\eps\}$ with small balls about $(\eps, 0)$ and $(0, \eps)$ removed, with the covering map given by $(u, v) = (\xt^{p-1}, \yt^{q-1})$.  All of the interesting parallel transport occurs in the neck regions which we have deleted, and is described by `partial Dehn twists' which we explicitly describe later in a local model.

\subsection{The preliminary vanishing cycles}

Let $\vcpr{0}{l,m}$ denote the preliminary vanishing cycle in $\Sigma$ corresponding to the critical point $(\zeta^l\xt^+_\mathrm{crit}, \eta^m\yt^+_\mathrm{crit})$ and the preliminary vanishing path $\gamma_{l,m}$.  The goal of this subsection is to describe these cycles, by a combination of symmetry considerations and parallel transport computations.

Since $\wt_\eps$ has real coefficients, we can temporarily view it as a function $\R^2 \rightarrow \R$.  This function has a local minimum at $(\xt^+_\mathrm{crit}, \yt^+_\mathrm{crit})$, where it attains the value $c_\mathrm{crit} < 0$.  There are no critical values in the interval $(c_\mathrm{crit}, 0)$, so the level sets $\wt_\eps^{-1}(c)$ for $c$ in this range have a component which is a smooth loop encircling $(\xt^+_\mathrm{crit}, \yt^+_\mathrm{crit})$, and which shrinks down to this point as $c \downarrow c_\mathrm{crit}$.  As $c \uparrow 0$ this loop, which we'll denote by $\Lambda_c$, converges to the boundary of the region in the upper right quadrant of $\R^2$ that is bounded on the left by $\xt=0$, below by $\yt=0$, and above and to the right by $\fact=\eps$.  We'll denote this piecewise smooth limiting loop by $\Lambda_0$.

Now return from this purely real discussion to the full complex picture.  Symplectic parallel transport between the fibres of $\wt_\eps$ over a path $c(t)$ is described by the ODE
\begin{equation}
\label{ParTransODE}
\begin{pmatrix}\dot{\xt} \\ \dot{\yt}\end{pmatrix} = \frac{\dot{c}}{|\diff \wt_\eps|^2} \begin{pmatrix}\overline{\partial_{\xt}\wt_\eps} \\ \overline{\partial_{\yt}\wt_\eps}\end{pmatrix}.
\end{equation}
This obviously preserves the real part of the fibre when $c$ moves along the real axis, as it did in the previous paragraph, so we see that the loops $\Lambda_c$ are carried to one another by parallel transport.  In particular, $\Lambda_{-\delta}$ is exactly the preliminary vanishing cycle $\vcpr{0}{0,0}$.

Just as we viewed $\Sigma$ as a smoothing of $\wt_\eps^{-1}(0)$, we shall understand $\vcpr{0}{0,0} = \Lambda_{-\delta}$ as a smoothing of $\Lambda_0$.  In $\Sigma'$ it comprises: the real line segment joining the deleted ball about $0$ to the deleted ball about $\eps^{1/(p-1)}$ in the $\xt$-axis; the real line segment joining the deleted ball about $0$ to the deleted ball about $\eps^{1/(q-1)}$ in the $\yt$-axis; the positive real lift of the line segment joining the deleted balls about $(\eps, 0)$ and $(0, \eps)$ in $\{u+v=\eps\}$, under the covering map $(\xt, \yt) \mapsto (u, v)$ described above.  It enters three of the neck regions, namely those corresponding to $\vc{\yt\fact}{0}$, $\vc{\xt\fact}{0}$ and $\vc{\xt\yt}{}$, in each of which it is given by the positive real locus in $(\xt, \yt)$-coordinates.  This is indicated in \cref{figPrelimCycles}, where the deleted balls are indicated by the grey blobs and the three segments of $\vcpr{0}{0,0}$ are respectively the the horizontal dash-dotted line, the vertical dotted line, and the dotted diagonal arc.

\begin{figure}[ht]
\centering
\begin{tikzpicture}[blob/.style={circle, draw=gray, fill=gray, inner sep=0, minimum size=\blobsize, line width=0.5mm}, noblob/.style={inner sep=0, minimum size=0}, scale=1.25]
\def\blobsize{2mm}

\begin{scope}[xshift=-0.2cm, scale=0.9]
\draw (-2.5, -0.75) -- (1, 0.75) -- (8.5, 0.75) -- (5, -0.75) -- cycle;
\draw[name path global=plane1] (-1.5, -1.5) -- (2, 0) -- (2, 5.5) -- (-1.5, 4) -- cycle;
\draw[fill=white, opacity=0.8] (-1.5, -0.75) -- (5, -0.75) -- (8.5, 0.75) -- (2, 0.75) -- (2, 5.5) -- (-1.5, 4) -- cycle;
\draw[dotted] (-1.5, -0.75) -- (2, 0.75);
\end{scope}

\draw[line width=2pt, dash pattern=on 6pt off 2pt] (0, 0) to [bend right=10] (0.2, 2.5);
\draw[line width=2pt] (0, 0) to [bend left=15] (0, 3.5);
\draw[line width=2pt] (0, 0) to [bend left=20] (3.5, 0.2);

\draw[fill=white, fill opacity=0.8] plot [smooth, tension=1] coordinates {(2.5, 4.5) (1, 3.9) (0, 3.5) (1, 3) (0.2, 2.5) (1, 2) (0, 1.5) (1, 1) (1.5, 0) (2, 1) (2.5, 0.2) (3, 1) (3.5, 0.2) (4, 1) (4.5, 0) (5.2, 1) (6.5, 2)};

\draw plot [smooth, tension=1] coordinates {(3.5, 4.5) (3.2, 3.9) (4.5, 4.5)};
\draw plot [smooth, tension=1] coordinates {(5.2, 4.5) (4.2, 3.1) (6.5, 4)};
\draw plot [smooth, tension=1] coordinates {(6.5, 3.5) (5, 2.2) (6.5, 2.5)};

\begin{scope}[xshift=3.5cm, yshift=2cm, rotate=-20]
\draw (0, 0) to [bend left=30] (1, 0);
\draw (-0.1, 0.075) to [bend right=10] (0, 0) to [bend right=20] (1, 0) to [bend right=10] (1.1, 0.075);
\end{scope}

\begin{scope}[xshift=1.5cm, yshift=2.2cm, rotate=-20]
\draw[name path global=curve1] (0, 0) to [bend left=30] (1, 0);
\draw[name path global=curve2] (-0.1, 0.075) to [bend right=10] (0, 0) to [bend right=20] (1, 0) to [bend right=10] (1.1, 0.075);

\draw[name intersections={of=plane1 and curve1}] (intersection-1) node[noblob](pt1){};
\draw[name intersections={of=plane1 and curve2}] (intersection-1) node[noblob](pt2){};
\draw (pt1) -- (pt2);
\end{scope}

\begin{scope}[xshift=2.2cm, yshift=3.3cm, rotate=-20]
\draw (0, 0) to [bend left=30] (1, 0);
\draw (-0.1, 0.075) to [bend right=10] (0, 0) to [bend right=20] (1, 0) to [bend right=10] (1.1, 0.075);
\end{scope}

\draw[line width=2pt] (0, 3.5) .. controls (2.8, 2.8) .. (3.5, 0.2);
\draw[line width=2pt, dash pattern=on 6pt off 2pt] (0.2, 2.5) .. controls (1.7, 2.1) .. (1.5, 0);
\draw[line width=2pt, dash pattern=on 6pt off 2pt on 2pt off 2pt] (1.5, 0) -- (0, 0);
\draw[line width=2pt, dash pattern=on 2pt off 2pt] (0, 0) -- (0, 1.5) .. controls (1.3, 1.3) .. (1.5, 0);

\draw (0, 1.5) node[blob]{};
\draw (0.2, 2.5) node[blob]{};
\draw (0, 3.5) node[blob]{};

\draw (0, 0) node[blob]{};

\draw (1.5, 0) node[blob]{};
\draw (2.5, 0.2) node[blob]{};
\draw (3.5, 0.2) node[blob]{};
\draw (4.5, 0) node[blob]{};

\draw (0.9,1.5) node{\small$\vcpr{0}{0, 0}$};
\draw (2.05,1.35) node{\small$\vcpr{0}{0, 1}$};
\draw (1,3.5) node{\small$\vcpr{0}{2, 2}$};

\end{tikzpicture}
\caption{Schematic picture of some preliminary vanishing cycles in $\Sigma'$ for loop polynomials.\label{figPrelimCycles}}
\end{figure}
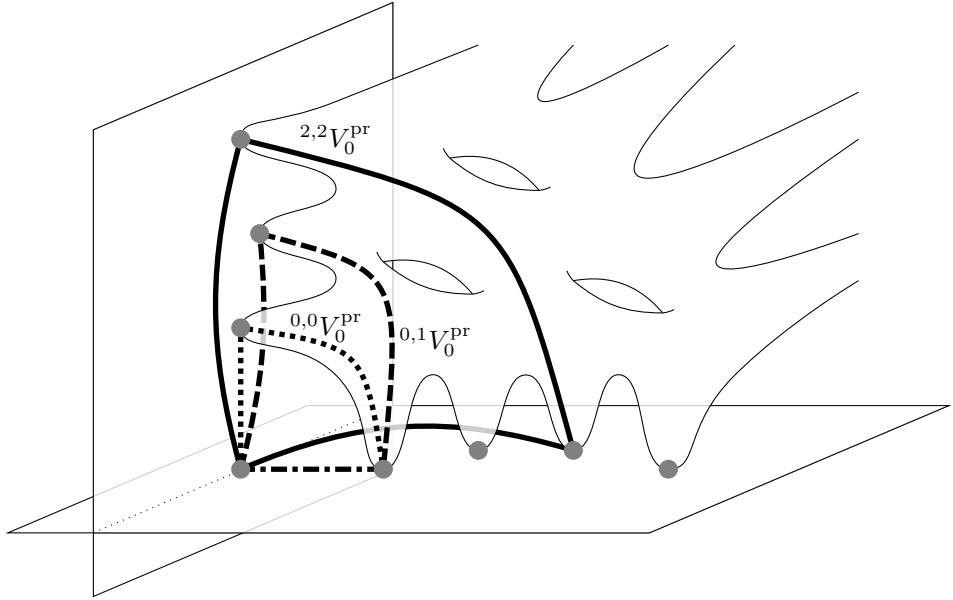

To compute the other $\vcpr{0}{l,m}$ we decompose the path $\gamma_{l,m}$ into its radial segment and its circular arc.  The map
\[
f_{l,m} : (\xt, \yt) \mapsto (\zeta^l \xt, \eta^m \yt)
\]
gives a symplectomorphism of $\C^2$ which $\wt_\eps$ intertwines with multiplication by $\zeta^l\eta^m$ on $\C$, so the curve $f_{l,m}(\vcpr{0}{0,0})$ is the vanishing cycle in the fibre over $-\zeta^l\eta^m\delta$ that corresponds to the critical point $(\zeta^l\xt^+_\mathrm{crit}, \eta^m\yt^+_\mathrm{crit})$ and the vanishing path given by the radial segment of $\gamma_{l,m}$.  This means that $\vcpr{0}{l,m}$ is obtained from $f_{l,m}(\vcpr{0}{0,0})$ by parallel transporting around the circular arc of $\gamma_{l,m}$.

We can therefore immediately describe the part of $\vcpr{0}{l,m}$ lying in $\Sigma'$, since it is obtained from the corresponding part of $\vcpr{0}{0,0}$ by applying $f_{l,m}$.  In full detail, it comprises: the radial line segment joining the deleted ball about $0$ to the deleted ball about $\zeta^l\eps^{1/(p-1)}$ in the $\xt$-axis; the real line segment joining the deleted ball about $0$ to the deleted ball about $\eta^m\eps^{1/(q-1)}$ in the $\yt$-axis; the lift to $\zeta^l\R_+ \times \eta^m\R_+ \subset \C^2$ of the line segment joining the deleted balls about $(\eps, 0)$ and $(0, \eps)$ in $\{u+v=\eps\}$, under the covering map $(\xt, \yt) \mapsto (u, v)$.  This is shown in \cref{figPrelimCycles}, where $\vcpr{0}{2,2}$ is drawn in solid black and $\vcpr{0}{0,1}$ is drawn dashed (the segment along which it overlaps with $\vcpr{0}{0,0}$ is shown dash-dotted).  The segments lying in the two coordiate axes should all really be straight, with the grey blobs lying on a circle about the origin, but we have deformed the picture in order to draw it in two dimensions.

To see what $\vcpr{0}{l,m}$ looks like in the three neck regions it meets, namely those corresponding to $\vc{\yt\fact}{l}$, $\vc{\xt\fact}{m}$ and $\vc{\xt\yt}{}$, we simply have to take the $(\zeta^l \R_+ \times \eta^m \R_+)$-locus in each of these necks over $-\zeta^l\eta^m\delta$ and parallel transport clockwise through angle $\theta_{l,m}$ around the circle of radius $\delta$; this is our next task.  Near the critical point $(0, 0)$, where $\xt$ and $\yt$ are both small, we may approximate $\wt_\eps$ by $-\eps \xt \yt$.  This corresponds to the $\vc{\xt\yt}{}$-neck region in $\Sigma$, and in this approximation the parallel transport equation \eqref{ParTransODE} simplifies to
\begin{equation}
\label{LocalParTrans}
\begin{pmatrix}\dot{\xt} \\ \dot{\yt}\end{pmatrix} = \frac{-\dot{c}}{\eps(|\xt|^2+|\yt|^2)} \begin{pmatrix}\overline{\yt} \\ \overline{\xt}\end{pmatrix}.
\end{equation}
We may also approximate the $(\zeta^l \R_+ \times \eta^m \R_+)$-locus in the $\vc{\xt\yt}{}$-neck over $-\zeta^l\eta^m\delta$ by the hyperbola
\[
(\xt, \yt)=\sqrt{\delta/\eps}(\zeta^le^s, \eta^me^{-s})
\]
parametrised by a small real variable $s$.  We want to parallel transport over the path $c(t) = -\delta e^{it}$ as $t$ decreases from $\theta_{l,m}$ to $0$, and we postulate a solution of the form $(\xt, \yt) = \sqrt{\delta/\eps}(e^{s+i\phi}, e^{-s+i(t-\phi)})$ where $\phi$ is a real function of $s$ and $t$.

Plugging this into \eqref{LocalParTrans} we obtain
\[
\begin{pmatrix}\dot{\phi}\xt \\ (1-\dot{\phi})\yt\end{pmatrix} = \frac{\xt\yt}{|\xt|^2+|\yt|^2} \begin{pmatrix}\overline{\yt} \\ \overline{\xt}\end{pmatrix},
\]
so after imposing the initial condition $\phi(s, \theta_{l,m}) = 2\pi l/(p-1)$ we get the unique solution
\begin{equation}
\label{phieqn}
\phi = \frac{2\pi l}{p-1} + \frac{e^{-2s}(t-\theta_{l,m})}{e^{2s}+e^{-2s}}.
\end{equation}
In particular, the value of $\phi$ at the end of the parallel transport ($t=0$), which we denote by $\phi_{l,m}$, is given by
\begin{equation}
\label{eqArgProfile}
\phi_{l,m}(s) \coloneqq \phi(s, 0) = \frac{2\pi}{e^{2s}+e^{-2s}} \left( \frac{e^{2s}l}{p-1} - \frac{e^{-2s}m}{q-1} \right).
\end{equation}
This is supposed to describe the argument of the $\xt$-component of $\vcpr{0}{l,m}$ (or minus the argument of the $\yt$-component) on the $\vc{\xt\yt}{}$-neck region of $\Sigma$, and note that it is consistent with the description we already have on $\Sigma'$: when $s$ becomes large this neck joins the $\xt$-axis, where we know that the $\xt$-component of $\vcpr{0}{l,m}$ has argument $2\pi l/(p-1)$; when $s$ becomes small the neck joins the $\yt$-axis, where we know that $\yt$-component of $\vcpr{0}{l,m}$ has argument $2\pi m/(q-1)$.

We can run analogous arguments on the other two necks that $\vcpr{0}{l,m}$ passes through.  To combine this information into a visualisable format, note that we can coordinatise the union of the $\xt$-axis part of $\Sigma'$ and the $\vc{\xt\yt}{}$- and $\vc{\yt\fact}{l}$-necks by $\xt$.  The $\xt$-projection of this region consists of the complex plane with a puncture at $0$, a puncture at $\zeta^l\xt^+_\mathrm{crit}$, and small balls about all other $\zeta^j\xt^+_\mathrm{crit}$ removed.  Small balls around the two punctures represent the two necks.  Strictly the punctures are extremely tiny deleted balls, but we will not make this distinction.

Away from the two neck regions in this picture, we are simply on the $\xt$-axis part of $\Sigma'$, so $\vcpr{0}{l,m}$ is given by the radial segment connecting them.  On the $\vc{\xt\yt}{}$-neck, near the puncture at $0$, the computation above shows that as we approach the pucture the argument of $\xt$ interpolates from $2\pi l /(p-1)$ to $-2\pi m/(q-1)$.  We can do the same on the $\vc{\yt\fact}{l}$ neck, near the puncture at $\zeta^l \xt^+_\mathrm{crit}$, but now the local coordinate is $\xt'$ where $\xt = \zeta^l\xt^+_\mathrm{crit} - \xt'$, and this time it is the argument of $\xt'$ which interpolates from $2\pi l/(p-1)$ to $-2\pi m/(q-1)$ as we approach the puncture.  The cases $(l,m)=(1, 0)$ and $(l,m)=(1,1)$ with $(p,q)=(4, 3)$ are shown in \cref{figlequalsL}.  We have drawn separate diagrams for the two choices of $(l,m)$ since the cycles overlap along their central segment and so would be difficult to distinguish if drawn on top of each other.
\begin{figure}[ht]
\centering
\begin{tikzpicture}[blob/.style={circle, draw=black, fill=black, inner sep=0, minimum size=\blobsize, line width=0.5mm}, critpt/.style={circle, draw=black, inner sep=0, minimum size=\ballsize}, noblob/.style={inner sep=0, minimum size=0}]
\def\ballsize{1.4cm}
\def\blobsize{0.03cm}
\def\r{2.5}

\draw (0,0) node[critpt, dotted]{};
\draw[opacity=0.2] (0,0) circle (\ballsize*0.25);
\draw (0,0) node[blob]{};
\foreach \a in {0, 4}
{
\draw ($cos(360*\a/6)*(\r, 0)+sin(360*\a/6)*(0, \r)$) node[critpt, dashed]{};
}
\foreach \a in {2}
{
\draw ($cos(360*\a/6)*(\r, 0)+sin(360*\a/6)*(0, \r)$) node(pt0)[critpt, dotted]{};
\draw[opacity=0.2] (pt0) circle (\ballsize*0.25);
\draw ($cos(360*\a/6)*(\r, 0)+sin(360*\a/6)*(0, \r)$) node(pt0)[blob]{};
}

\def\phi{(120*exp(2*\x)-180*exp(-2*\x))/(exp(2*\x)+exp(-2*\x))};
\draw[domain=-2.1:1.5, samples=100, smooth] plot ({0.2*exp(\x)*cos(\phi)}, {0.2*exp(\x)*sin(\phi)}) node(pt1)[noblob]{};
\begin{scope}[rotate=180, shift={(pt0)}]
\draw[domain=-2.1:1.5, samples=100, smooth] plot ({0.2*exp(\x)*cos(\phi)}, {0.2*exp(\x)*sin(\phi)}) node(pt2)[noblob]{};
\end{scope}
\draw[shorten <= -0.01cm, shorten >= -0.01cm] (pt1) -- (pt2);

\begin{scope}[xshift=-8cm]

\draw (0,0) node[critpt, dotted]{};
\draw[opacity=0.2] (0,0) circle (\ballsize*0.25);
\draw (0,0) node[blob]{};
\foreach \a in {0, 4}
{
\draw ($cos(360*\a/6)*(\r, 0)+sin(360*\a/6)*(0, \r)$) node[critpt, dashed]{};
}
\foreach \a in {2}
{
\draw ($cos(360*\a/6)*(\r, 0)+sin(360*\a/6)*(0, \r)$) node(pt0)[critpt, dotted]{};
\draw[opacity=0.2] (pt0) circle (\ballsize*0.25);
\draw ($cos(360*\a/6)*(\r, 0)+sin(360*\a/6)*(0, \r)$) node(pt0)[blob]{};
}

\def\phi{(120*exp(2*\x)-0*exp(-2*\x))/(exp(2*\x)+exp(-2*\x))};
\draw[domain=-2.1:1.5, samples=100, smooth] plot ({0.2*exp(\x)*cos(\phi)}, {0.2*exp(\x)*sin(\phi)}) node(pt1)[noblob]{};
\begin{scope}[rotate=180, shift={(pt0)}]
\draw[domain=-2.1:1.5, samples=100, smooth] plot ({0.2*exp(\x)*cos(\phi)}, {0.2*exp(\x)*sin(\phi)}) node(pt2)[noblob]{};
\end{scope}
\draw[shorten <= -0.01cm, shorten >= -0.01cm] (pt1) -- (pt2);

\end{scope}

\end{tikzpicture}
\caption{The $\xt$-projection of the preliminary vanishing cycles $\vcpr{0}{1,0}$ (left) and $\vcpr{0}{1,1}$ (right) in the $\xt$-axis part of $\Sigma'$ and the $\vc{\xt\yt}{}$- and $\vc{\yt\fact}{l}$-necks, with $(p,q)=(4,3)$.\label{figlequalsL}}
\end{figure}
The dashed circles represent the boundaries of the deleted balls, the dotted circles represent the boundaries of the neck regions, and the blobs represent the punctures.  The feint solid circles are the waist curves $\vc{\xt\yt}{}$ and $\vc{\yt\fact}{l}$.

There is a corresponding picture for the $\yt$-projection of the $\yt$-axis part of $\Sigma'$ and the $\vc{\xt\yt}{}$- and $\vc{\xt\fact}{m}$- necks.  The picture on $\{\fact=\eps\}$ part of $\Sigma'$ is essentially uninteresting since the $\vcpr{0}{l,m}$ are pairwise disjoint there.  This is because on that part the different $\vcpr{0}{l,m}$ are different lifts of the same segment in $\{u+v=\eps\}$.  Combining the pictures on these three parts of $\Sigma$ gives a complete description of all of the preliminary vanishing cycles.

\begin{rmk}
\label{rmkEasyIntersections}
There are some obvious points to note here, which are clear parallels of the structure of the generating set on the B-side.  First, the vanishing cycles $\vc{\xt\yt}{}$, $\vc{\yt\fact}{l}$ and $\vc{\xt\fact}{m}$ are all pairwise disjoint.  Second, each $\vcpr{0}{l,m}$ intersects $\vc{\xt\yt}{}$ exactly once, tranvsersely.  Third, $\vcpr{0}{l,m}$ and $\vc{\yt\fact}{L}$ intersect once, transversely, if $l=L$ and are disjoint otherwise (similarly for $\vc{\xt\fact}{M}$).  And finally, if $l\neq L$ and $m \neq M$ then $\vcpr{0}{l,m}$ and $\vcpr{0}{L,M}$ are disjoint except on the $\vc{\xt\yt}{}$-neck region, where \eqref{eqArgProfile} tells us that they intersect once, transversely, if $l > L$ and $m > M$ or vice versa, and are disjoint otherwise (as $|\xt|$ increases, the difference in their $\xt$-arguments varies monotonically from $2\pi(m-M)/(q-1)$ to $2\pi(L-l)/(p-1)$).
\end{rmk}

\subsection{Modifying the vanishing paths}
\label{ModifyingPaths}

As already noted, the preliminary vanishing paths (plus the vanishing paths connecting $-\delta$ to zero) do not form a distinguished basis of vanishing paths because they intersect and overlap each other.  In this subsection we describe how to remedy this, which also involves perturbing $\wt$ to separate the critical values, in such a way that the vanishing cycles are basically unaffected.

By plotting modulus and argument$+\pi$ we may view the preliminary paths $\gamma_{l,m}$ as right-angled paths in $\R^2$ from $(\delta, 0)$ to $(\delta, \theta_{l,m})$ to $(-c_\mathrm{crit}, \theta_{l,m})$.  We define modified paths $\gamma'_{l,m}$ using this picture to be the piecewise linear paths as follows:
\begin{itemize}
\item From $(\delta, 0)$ to $(\delta+\delta', \theta_{l,m})$ to $(-c_\mathrm{crit}, \theta_{l,m})$ for some small positive $\delta'$, if $\theta_{l,m} < 2\pi$.
\item From $(\delta, 0)$ to $(\delta+\delta', 2\pi+\lambda(\theta_{l,m}-4\pi))$ to $(\delta+2\delta', 2\pi+\lambda(\theta_{l,m}-4\pi))$ to $(\delta+3\delta', \theta_{l,m}-\theta'))$ to $(-c_\mathrm{crit}, \theta_{l,m}-\theta')$ for some small positive $\lambda$ and $\theta'$, if $\theta_{l,m} \geq 2\pi$.
\end{itemize}
In the second case we have moved the end-point of the path so we correspondingly perturb the fibration so that the critical point $(\zeta^l\xt^+_\mathrm{crit}, \eta^m\yt^+_\mathrm{crit})$ has its critical value $\zeta^l\eta^mc_\mathrm{crit}$ rotated by $e^{-i\theta'}$.  The paths are illustrated in the case $(p,q)=(4,6)$ in \cref{figGammap}.
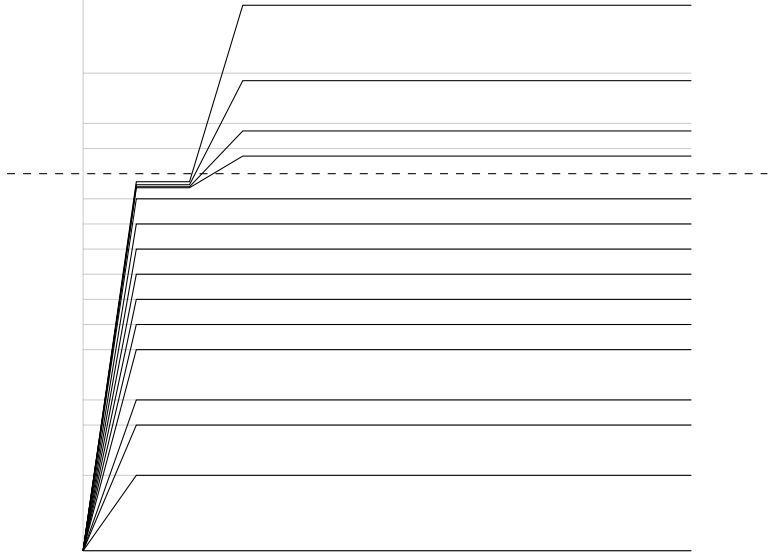
\begin{figure}[ht]
\centering
\begin{tikzpicture}
\def\ccrit{8}
\def\deltap{0.7}
\def\yscale{5}
\def\shifta{0.2}
\def\shiftb{0.1}

\draw[opacity=0.2] (0, 0) -- ($4/5*(0, \yscale)+2/3*(0, \yscale)$);
\foreach \l in {0, ..., 4}
{
\foreach \m in {0, ..., 2}
{
\draw[opacity=0.2] ($\l/5*(0, \yscale)+\m/3*(0, \yscale)$) --+ (\ccrit, 0);
}
}

\draw[dashed] (-1, \yscale) -- (\ccrit+1, \yscale);

\foreach \a in {0, 3, 5, 6, 8, 9, 10, 11, 12, 13, 14}
{
\draw (0, 0) -- ($\a/15*(0, \yscale)+(\deltap, 0)$) -- ($\a/15*(0, \yscale)+(\ccrit, 0)$);
}

\foreach \a in {16, 17, 19, 22}
{
\draw (0, 0) -- ($(0, \yscale)+\a/15*(0, \shifta)-2*(0, \shifta)+(\deltap, 0)$) -- ($(0, \yscale)+\a/15*(0, \shifta)-2*(0, \shifta)+2*(\deltap, 0)$) -- ($\a/15*(0, \yscale)-(0, \shiftb)+3*(\deltap, 0)$) -- ($\a/15*(0, \yscale)-(0, \shiftb)+(\ccrit, 0)$);
}

\end{tikzpicture}
\caption{The paths $\gamma'_{l,m}$ in modulus-(argument$+\pi$) space, when $(p,q)=(4,6)$.\label{figGammap}}
\end{figure}
The feint lines are the preliminary paths $\gamma_{l,m}$ and the dashed line is at height $2\pi$.

This construction has the following key properties:
\begin{itemize}
\item The clockwise ordering of the tangent directions $\dot{\gamma}'_{l,m}(0)$ is by decreasing value of $\theta_{l,m}$.
\item If $\theta_{l,m} = \theta_{L,M}$ then $\gamma'_{l,m} = \gamma'_{L,M}$.
\item If $\theta_{l,m} \neq \theta_{L,M}$ then $\gamma'_{l,m}$ and $\gamma'_{L,M}$ are disjoint unless $\theta_{l,m} > \theta_{L,M}+2\pi$ (or vice versa), in which case they intersect once, transversely, close to $-\zeta^L\eta^M\delta$ (respectively $-\zeta^l\eta^m\delta$).
\end{itemize}
The control on the position of the intersection point in the third property is the reason for the curious kink in the paths $\gamma'_{l,m}$ for $\theta_{l,m} \geq 2\pi$.  If we had instead taken these paths to be $(\delta, 0)$ to $(\delta+\delta', \theta_{l,m}-\theta')$ to $(-c_\mathrm{crit}, \theta_{l,m}-\theta')$ then the intersection between $\gamma'_{l,m}$ and $\gamma'_{L,M}$ when $\theta_{l,m} > \theta_{L,M}+2\pi$ woud have occurred on the sloping regions of both paths, and therefore been awkward to locate.

Our next task is to explain how to modify those $\gamma'_{l,m}$ for which $\theta_{l,m} > 2\pi$ in order to remove the transverse intersections just described.  The key observation is:

\begin{lem}
\label{lemCyclesDisjoint}
Suppose $\theta_{l,m} > \theta_{L,M} + 2\pi$, and let $z$ denote the intersection point of $\gamma'_{l,m}$ and $\gamma'_{L,M}$.  Inside the fibre $\Sigma_z = \wt_\eps^{-1}(z)$ there are vanishing cycles corresponding to the critical points $(\zeta^l\xt^+_\mathrm{crit}, \eta^m\yt^+_\mathrm{crit})$ and $(\zeta^L\xt^+_\mathrm{crit}, \eta^M\yt^+_\mathrm{crit})$ and the truncations of the vanishing paths $\gamma'_{l,m}$ and $\gamma'_{L,M}$.  Denoting these by $V_1$ and $V_2$ respectively, we have
\[
V_1 \cap V_2 = \emptyset.
\]
\end{lem}
\begin{proof}
First note that if $l \leq L$ then $\theta_{l,m}-\theta_{L,M}$ is at most $2\pi (q-2)/(q-1)$, so we must have $l > L$ and similarly $m > M$.  By applying $f_{L,M}^{-1}$ we may then assume without loss of generality that $L=M=0$ and $l, m > 0$.  The former means that $z$ is approximately $-\delta$, and that $V_2 \subset \Sigma_z$ is approximately $\vcpr{0}{0,0} \subset \Sigma$.  The curve $V_1$, meanwhile, is constructed in approximately the same way as $\vcpr{0}{l,m}$ but with the parallel transport around the circle of radius $\delta$ done from $\theta_{l,m}$ to $2\pi$, rather than to $0$.  For the rest of the argument we take these approximations to be exact.  Since the cycles $V_1$ and $V_2$ are compact, once we show that they are disjoint after our small approximation we automatically deduce that they were disjoint before (compact and disjoint implies separated by a positive distance).

Since $l$ and $m$ are both positive we see that $V_1$ and $V_2 = \vcpr{0}{0,0}$ are disjoint on $\Sigma' \subset \Sigma$, and that the only neck region that they both pass through is that corresponding to $\vc{\xt\yt}{}$.  This means that the only possible intersections occur in this neck, which we can coordinatise by projection to $\xt$.  In this projection we know that $\vcpr{0}{0,0}$ and $V_1$ are parametrised by
\[
\xt=\sqrt{\delta/\eps}e^s \text{\quad and \quad} \xt=\sqrt{\delta/\eps}e^{s+i\phi},
\]
respectively, where $\phi$ is given by setting $t=2\pi$ in \eqref{phieqn}.  It therefore suffices to show that this function $\phi$ never hits $2\pi \Z$.  To prove this, simply note that the function is monotonically increasing from $2\pi - 2\pi m/(q-1)$, which is strictly positive, to $2\pi l/(p-1)$, which is strictly less than $2\pi$.
\end{proof}

Now let $\gamma''_{l,m}$ denote the path obtained from $\gamma'_{l,m}$ by introducing a long thin finger which loops around the radial segment of $\gamma'_{L,M}$, for each $(L, M)$ with $\theta_{l,m} > \theta_{L,M} + 2\pi$.  \Cref{figFingers} illustrates $\gamma''_{2,4}$ in the case $(p,q)=(4,6)$.  The feint lines show the paths $\gamma'_{L,M}$ which we have had to loop around.
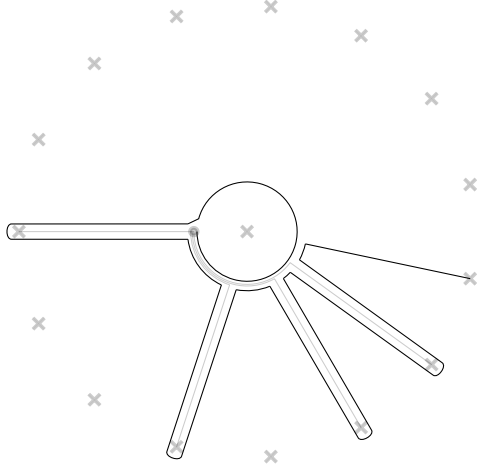
\begin{figure}[ht]
\centering
\begin{tikzpicture}[blob/.style={circle, draw=black, fill=black, inner sep=0, minimum size=\blobsize, line width=0.5mm}, critval/.style={cross out, draw=black, fill=black, inner sep=0, minimum size=\crosssize, line width=0.5mm}, noblob/.style={inner sep=0, minimum size=0}]
\def\blobsize{1mm}
\def\crosssize{1mm}
\def\masterDELTA{0.7}
\def\DELTA{\masterDELTA}
\def\r{3}
\def\fingerwidth{7.5}

\begin{scope}[opacity=0.2]
\draw (-\DELTA, 0) node[blob]{};
\draw (0, 0) node[critval]{};

\foreach \a in {1, ..., 15}
{
\draw ($-cos(360*\a/15)*(\r, 0)+sin(360*\a/15)*(0, \r)$) node[critval]{};
}

\draw (-\r, 0) -- (-\DELTA, 0);

\def\l{1}
\def\m{0}
\def\theta{360*(\l/3+\m/5)}
\def\DELTA{\masterDELTA*1.03}
\draw ($-cos(\theta)*(\r, 0)-sin(\theta)*(0, \r)$) -- ($-cos(\theta)*(\DELTA, 0)-sin(\theta)*(0, \DELTA)$);
\draw ($-cos(\theta)*(\DELTA, 0)-sin(\theta)*(0, \DELTA)$) arc(\theta+180:180:\DELTA);

\def\l{0}
\def\m{1}
\def\theta{360*(\l/3+\m/5)}
\def\DELTA{\masterDELTA*1.06}
\draw ($-cos(\theta)*(\r, 0)-sin(\theta)*(0, \r)$) -- ($-cos(\theta)*(\DELTA, 0)-sin(\theta)*(0, \DELTA)$);
\draw ($-cos(\theta)*(\DELTA, 0)-sin(\theta)*(0, \DELTA)$) arc(\theta+180:180:\DELTA);

\def\l{0}
\def\m{2}
\def\theta{360*(\l/3+\m/5)}
\def\DELTA{\masterDELTA*1}
\draw ($-cos(\theta)*(\r, 0)-sin(\theta)*(0, \r)$) -- ($-cos(\theta)*(\DELTA, 0)-sin(\theta)*(0, \DELTA)$);
\draw ($-cos(\theta)*(\DELTA, 0)-sin(\theta)*(0, \DELTA)$) arc(\theta+180:180:\DELTA);
\end{scope}

\def\DELTA{\masterDELTA*1.12}
\draw ($-cos(168)*(\r, 0)-sin(168)*(0, \r)$) -- ($-cos(168)*(\DELTA, 0)-sin(168)*(0, \DELTA)$);

\draw ($-cos(168)*(\DELTA, 0)-sin(168)*(0, \DELTA)$) arc(168+180:144+180+\fingerwidth:\DELTA);
\draw ($-cos(144-\fingerwidth)*(\DELTA, 0)-sin(144-\fingerwidth)*(0, \DELTA)$) arc(144+180-\fingerwidth:120+180+\fingerwidth:\DELTA);
\draw ($-cos(120-\fingerwidth)*(\DELTA, 0)-sin(120-\fingerwidth)*(0, \DELTA)$) arc(120+180-\fingerwidth:72+180+\fingerwidth:\DELTA);
\draw ($-cos(72-\fingerwidth)*(\DELTA, 0)-sin(72-\fingerwidth)*(0, \DELTA)$) arc(72+180-\fingerwidth:180+\fingerwidth:\DELTA);

\foreach \theta in {144, 120, 72, 0}
{
\draw ($-cos(\theta+\fingerwidth)*(\DELTA, 0)-sin(\theta+\fingerwidth)*(0, \DELTA)$) to ($-cos(\theta+\fingerwidth/4)*(\r+0.1, 0)-sin(\theta+\fingerwidth/4)*(0, \r+0.1)$) to [bend left=90] ($-cos(\theta-\fingerwidth/4)*(\r+0.1, 0)-sin(\theta-\fingerwidth/4)*(0, \r+0.1)$) to ($-cos(\theta-\fingerwidth)*(\DELTA, 0)-sin(\theta-\fingerwidth)*(0, \DELTA)$);
}

\draw ($-cos(-\fingerwidth)*(\DELTA, 0)-sin(-\fingerwidth)*(0, \DELTA)$) -- ($-cos(-\fingerwidth*2)*0.84*(\DELTA, 0)-sin(-\fingerwidth*2)*0.84*(0, \DELTA)$) arc(180-\fingerwidth*2:-180:\DELTA*0.84);

\end{tikzpicture}
\caption{The path $\gamma''_{2,4}$ when $(p,q)=(4,6)$.\label{figFingers}}
\end{figure}
In principle, each time we go around one of the fingers the `intermediate vanishing cycle' $V_1$ is changed by the monodromy around $\zeta^L\eta^Mc_\mathrm{crit}$, which is precisely the Dehn twist in $V_2$ (or, more accurately, the product of the Dehn twists in all cycles constructed in the same way as $V_2$ as $(L, M)$ ranges over all pairs with the same value of $\theta_{L,M}$), and by \cref{lemCyclesDisjoint} this has no effect.  We conclude that the vanishing cycles for the new paths $\gamma''_{l,m}$ coincide with those of the previous paths $\gamma'_{l,m}$, which in turn are small perturbations of those of the preliminary paths $\gamma_{l,m}$.  Note also that we can construct the new paths so as not to introduce any new intersections between them (for example, we can make sure the fingers for $\gamma''_{2, 3}$ go \emph{outside} the fingers for $\gamma''_{2,4}$ shown in \cref{figFingers}).

The upshot is that we now have vanishing paths $\gamma''_{l,m}$, plus the vanishing paths connecting $-\delta$ to $0$, which form a distinguished basis except for the fact that some of the paths coincide with each other.  This is straightforwardly fixed by making a small perturbation of the fibration to separate the critical values, and corresponding small perturbations of the paths.  The precise way in which this is done will affect the ordering of the paths, and hence the ordering of the vanishing cycles in $\mathcal{A}$, but this is irrelevant since the ambiguity is always between cycles which are disjoint and therefore orthogonal in the category.

We conclude:
\begin{prop}
\label{propAfromPrelim}
There exists a Morsification of $\wt$ and a distinguished basis of vanishing paths such that the corresponding vanishing cycles are arbitrarily small pertubations of the $\vcpr{0}{l,m}$, $\vc{\yt\fact}{l}$, $\vc{\xt\fact}{m}$ and $\vc{\xt\yt}{}$ as constructed above.  The $\vcpr{0}{l,m}$ are ordered by decreasing value of $\theta_{l,m}$, and by choosing the starting direction for our clockwise ordering to be $e^{i\theta}$, for $\theta$ a small positive angle, they occur before all of the other vanishing cycles.\hfill$\qed$
\end{prop}

\subsection{Isotoping the vanishing cycles and computing the morphisms}
\label{Isotopies}

Let us refer to the small perturbations of the preliminary vanishing cycles $\vc{0}{l,m}$ that appear in \cref{propAfromPrelim} as \emph{temporary} vanishing cycles.  In order to compute the category $\mathcal{A}$ we need to understand the intersection pattern of these temporary cycles.  Some pairs of these cycles were already transverse before perturbing, as described in \cref{rmkEasyIntersections}---in fact, all pairs except those of the form $\vcpr{0}{l,m}$, $\vcpr{0}{L,M}$ with $l=L$ or $m=M$---so their intersections are unaffected by the small perturbations.  For the non-transverse pairs of preliminary cycles, however, which actually overlap along segments, we cannot pin down the intersections of the corresponding temporary cycles without keeping more careful track of the perturbations, which is impractical.

In order to overcome this we shall modify these problematic temporary cycles, which are small perturbations of the $\vcpr{0}{l,m}$, by Hamilton isotopies to obtain \emph{final} vanishing cycles $\vc{0}{l,m}$ which we will use to compute $\mathcal{A}$.  This does not affect the quasi-equivalence type of the category.  These isotopies will be small in the absolute sense, and in particular will only affect intersections between pairs of cycles which were non-transverse before perturbing from preliminary to temporary, but will not be small compared with these perturbations.  Indeed, their very point is to undo any uncertainty in the intersection pattern which these perturbations introduced.

\begin{rmk}
Since each waist curve $\vc{\yt\fact}{l}$, $\vc{\xt\fact}{m}$, and $\vc{\xt\yt}{}$ was already transverse to all other cycles, the corresponding perturbed curve in \cref{propAfromPrelim} has the same intersection pattern.  We therefore do not notationally distinguish between the waist curves and their perturbations.
\end{rmk}

We only need describe the isotopies on the regions where the preliminary cycles were non-transverse.  This means that for each $\vcpr{0}{l,m}$ we may focus on neighbourhoods of its segments lying in the $\xt$-axis and $\yt$-axis regions of $\Sigma'$.  So fix an $(l,m)$ and consider the part of $\vcpr{0}{l,m}$ (strictly the temporary cycle obtained from this) lying in the $\xt$-axis part of $\Sigma'$ and the $\vc{\xt\yt}{}$- and $\vc{\yt\fact}{l}$-necks.  We view this in the $\xt$-projection, as in \cref{figlequalsL}.

We first isotope the $\xt$-axis segment, between the two necks, anticlockwise about $\xt=0$ by an amount proportional to $m$.  This of course requires corresponding small modifications at the boundaries of the neck regions to keep the curve continuous. To make the isotopy Hamiltonian, we then push the curve slightly clockwise just inside the $\vc{\xt\yt}{}$-neck.  The result is shown schematically in \cref{figHamPerts} for the $(l,m)=(1,0)$ and $(1,1)$ cycles with $(p,q)=(4,3)$.
\begin{figure}[ht]
\centering
\begin{tikzpicture}[blob/.style={circle, draw=black, fill=black, inner sep=0, minimum size=\blobsize, line width=0.5mm}, critpt/.style={circle, draw=black, inner sep=0, minimum size=\ballsize}, noblob/.style={inner sep=0, minimum size=0}]
\def\ballsize{1.4cm}
\def\blobsize{0.03cm}
\def\r{2.5}

\draw (0,0) node[critpt, dotted]{};
\draw (0,0) node[blob]{};
\foreach \a in {0, 4}
{
\draw ($cos(360*\a/6)*(\r, 0)+sin(360*\a/6)*(0, \r)$) node[critpt, dashed]{};
}
\foreach \a in {2}
{
\draw ($cos(360*\a/6)*(\r, 0)+sin(360*\a/6)*(0, \r)$) node(pt0)[critpt, dotted]{};
\draw ($cos(360*\a/6)*(\r, 0)+sin(360*\a/6)*(0, \r)$) node(pt0)[blob]{};
}

\def\phi{(120*exp(2*\x)-0*exp(-2*\x))/(exp(2*\x)+exp(-2*\x))};
\draw[domain=-2.1:1.5, samples=100, smooth] plot ({0.2*exp(\x)*cos(\phi)}, {0.2*exp(\x)*sin(\phi)}) node(pt1)[noblob]{};
\begin{scope}[rotate=180, shift={(pt0)}]
\draw[domain=-2.1:1.5, samples=100, smooth] plot ({0.2*exp(\x)*cos(\phi)}, {0.2*exp(\x)*sin(\phi)}) node(pt2)[noblob]{};
\end{scope}
\draw[shorten <= -0.01cm, shorten >= -0.01cm] (pt1) -- (pt2);

\def\phi{(130*exp(2*\x)-180*exp(-2*\x))/(exp(2*\x)+exp(-2*\x))};
\draw[domain=-2.1:1.5, samples=100, smooth] plot ({0.2*exp(\x)*cos(\phi)}, {0.2*exp(\x)*sin(\phi)}) node(pt1)[noblob]{};
\begin{scope}[rotate=180, shift={(pt0)}]
\def\phi{(110*exp(2*\x)-180*exp(-2*\x))/(exp(2*\x)+exp(-2*\x))};
\draw[domain=-2.1:1.5, samples=100, smooth] plot ({0.2*exp(\x)*cos(\phi)}, {0.2*exp(\x)*sin(\phi)}) node(pt2)[noblob]{};
\end{scope}
\draw[shorten <= -0.01cm, shorten >= -0.01cm] (pt1) to [bend left=5] (pt2);

\end{tikzpicture}
\caption{The $\xt$-projection of the final vanishing cycles $\vc{0}{1,0}$ and $\vc{0}{1,1}$ in the $\xt$-axis part of $\Sigma'$ and the $\vc{\xt\yt}{}$- and $\vc{\yt\fact}{1}$-necks, with $(p,q)=(4,3)$.\label{figHamPerts}}
\end{figure}
We then do a similar thing on the $\yt$-axis part of $\Sigma'$ and the $\vc{\xt\yt}{}$- and $\vc{\yt\fact}{l}$-necks.

The result is that the final cycles $\vc{0}{l,m}$ are all pairwise disjoint, except on the $\vc{\xt\yt}{}$-neck.  Inside this neck, the intersections between $\vc{0}{l,m}$ and $\vc{0}{L,M}$ remain as described in \cref{rmkEasyIntersections} when $l\neq L$ and $m\neq M$.  When $l=L$ and (without loss of generality) $m>M$ the effect is as follows.  Before perturbing and isotoping, the $\xt$-arguments of the curves on the $\vc{\xt\yt}{}$-neck are described by \eqref{eqArgProfile} and illustrated in the left-hand part of \cref{figArgProfiles}.  In particular, the curves converge as $|\xt|$ becomes large.
\begin{figure}[ht]
\centering
\begin{tikzpicture}[yscale=2]
\def\pert{0.2}
\def\ppert{\pert*tanh(20*\mq*(\x-1.4))}
\def\mpert{\pert*tanh(-20*\lp*(\x+1.4))}

\def\lp{1/2}
\def\mqa{1/3}
\def\mqb{2/3}

\draw[->] (0, -1) -- (0, 1) node[anchor=east]{\tiny$\arg \xt$};
\draw (0, \lp) node[anchor=south east]{\tiny$2\pi l/(p-1)$};
\draw[opacity=0.2] (-1, \lp) -- (12, \lp);
\draw (0, -\mqa) node[anchor=south east]{\tiny$-2\pi M/(q-1)$};
\draw[opacity=0.2] (-1, -\mqa) -- (12, -\mqa);
\draw (0, -\mqb) node[anchor=south east]{\tiny$-2\pi m/(q-1)$};
\draw[opacity=0.2] (-1, -\mqb) -- (12, -\mqb);

\begin{scope}[xshift=3cm]
\draw[->] (-2.4, 0) -- (2.4, 0) node[anchor=north]{\tiny$\log |\xt|$};

\def\mq{\mqa}
\draw[domain=-2:2, samples=100, smooth] plot ({\x}, {(exp(2*\x)*\lp-exp(-2*\x)*\mq)/(exp(2*\x)+exp(-2*\x))});

\def\mq{\mqb}
\draw[domain=-2:2, samples=100, smooth] plot ({\x}, {(exp(2*\x)*\lp-exp(-2*\x)*\mq)/(exp(2*\x)+exp(-2*\x))});
\end{scope}

\begin{scope}[xshift=9cm]
\draw[->] (-2.4, 0) -- (2.4, 0) node[anchor=north]{\tiny$\log |\xt|$};

\def\mq{\mqa}
\draw[domain=-2:-1, samples=100, smooth] plot ({\x}, {(exp(2*\x)*(\lp-\mq*\pert)-exp(-2*\x)*(\mq+\lp*\mpert))/(exp(2*\x)+exp(-2*\x))});
\draw[domain=-1.01:1.01, samples=100, smooth] plot ({\x}, {(exp(2*\x)*(\lp-\mq*\pert)-exp(-2*\x)*(\mq-\lp*\pert))/(exp(2*\x)+exp(-2*\x))});
\draw[domain=1:2, samples=100, smooth] plot ({\x}, {(exp(2*\x)*(\lp+\mq*\ppert)-exp(-2*\x)*(\mq-\lp*\pert))/(exp(2*\x)+exp(-2*\x))});

\def\mq{\mqb}
\draw[domain=-2:-1, samples=100, smooth] plot ({\x}, {(exp(2*\x)*(\lp-\mq*\pert)-exp(-2*\x)*(\mq+\lp*\mpert))/(exp(2*\x)+exp(-2*\x))});
\draw[domain=-1.01:1.01, samples=100, smooth] plot ({\x}, {(exp(2*\x)*(\lp-\mq*\pert)-exp(-2*\x)*(\mq-\lp*\pert))/(exp(2*\x)+exp(-2*\x))});
\draw[domain=1:2, samples=100, smooth] plot ({\x}, {(exp(2*\x)*(\lp+\mq*\ppert)-exp(-2*\x)*(\mq-\lp*\pert))/(exp(2*\x)+exp(-2*\x))});
\end{scope}

\end{tikzpicture}
\caption{The $\vc{\xt\yt}{}$-neck regions of the curves $\vc{0}{l,m}$ and $\vc{0}{l,M}$ before (left) and after (right) perturbing and isotoping.\label{figArgProfiles}}
\end{figure}
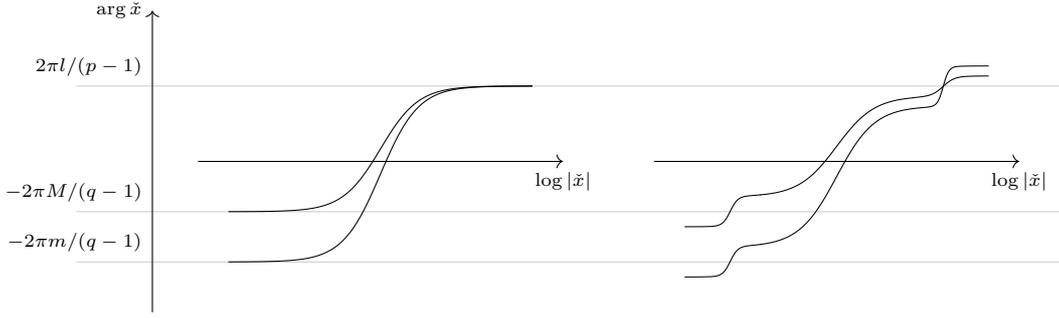
The isotoped curves are shown schematically in the right-hand part of the same diagram, and we see that now they intersect once, transversely, where $\vc{0}{l,m}$ has been pushed further anticlockwise than $\vc{0}{l,M}$.

Combining this with \cref{rmkEasyIntersections} and \cref{propAfromPrelim} (with the $\vc{0}{l,m}$ now being used in place of the $\vcpr{0}{l,m}$) we obtain a model for $\mathcal{A}$ with precisely the following basis of morphisms:
\begin{itemize}
\item An identity morphism for each object.
\item A morphism from $\vc{0}{l,m}$ to $\vc{0}{L,M}$ whenever $(l,m) \neq (L,M)$ but both $l \geq L$ and $m \geq M$.
\item A morphism from each $\vc{0}{l,m}$ to each of $\vc{\xt\yt}{}$, $\vc{\yt\fact}{l}$ and $\vc{\xt\fact}{m}$.
\end{itemize}
This is a chain-level description, but for any pair of objects the morphism complexes are either one- or zero-dimensional, so all differentials trivially vanish.  Additively the cohomology algebra therefore matches exactly with the quiver description of $\mathcal{B}$ in \cref{figLoopQuiver}, under the identification
\begin{equation}
\label{ObjectMatch}
\begin{aligned}
\vc{0}{l,m} &\leftrightarrow \obj{0}{i,j}
\\ \vc{\yt\fact}{l} &\leftrightarrow \obj{x}{i}[3]
\\ \vc{\xt\fact}{m} &\leftrightarrow \obj{y}{j}[3]
\\ \vc{\xt\yt}{} &\leftrightarrow \obj{\fac}{}[3]
\end{aligned}
\quad \text{with} \quad
\begin{aligned}
i+l&=p-1
\\ j+m&=q-1.
\end{aligned}
\end{equation}
To complete the proof of \cref{Thm1} in the loop case, we just need to check that the compositions agree, and that the vanishing cycles can be graded so as to place all morphisms in degree $0$.  These are the subjects of the next two subsections.

\begin{rmk}
The identification \eqref{ObjectMatch} is between the objects of $\mathcal{A} \subset \mathcal{F}(\Sigma)$ and $\mathcal{B} \subset \mathrm{mf}(\C^2, \Gamma_\w, \w)$.  In the ultimate equivalence $\mathcal{F}(\wt) \simeq \mathrm{mf}(\C^2, \Gamma_\w, \w)$ the vanishing cycles in \eqref{ObjectMatch} should be replaced by their images under the equivalence $\Tw \mathcal{A} \rightarrow \mathcal{F}(\wt)$, which are the corresponding Lefschetz thimbles.
\end{rmk}

\subsection{Composition}

Suppose $L_0$, $L_1$ and $L_2$ are three (final) vanishing cycles such that $L_0 < L_1 < L_2$ with respect to the ordering on the category $\mathcal{A}$ (we are calling them $L$ rather than $V$ to avoid conflict with our earlier notation for specific cycles).  We need to compute the composition
\begin{equation}
\label{eqFloerProduct}
HF^*(L_1, L_2) \otimes HF^*(L_0, L_1) \rightarrow HF^*(L_0, L_2),
\end{equation}
which is defined by counting pseudo-holomorphic triangles, and Seidel \cite[Section (13b)]{SeidelBook} shows that this can be done combinatorially by simply counting triangular regions bounded by the $L_i$.  The crucial point is that one can do without Hamiltonian perturbations or perturbations of the complex structure because the directedness of the category automatically rules out contributions from constant discs.  (It also rules out discs in which the ordering of the Lagrangians around the boundary does not match their ordering in the category.)

In order for this composition to have a chance of being non-zero (i.e.~in order for all three $HF^*$ groups to be non-zero) we must have $L_0 = \vc{0}{l,m}$ and $L_1 = \vc{0}{L,M}$ for some distinct $(l,m)$ and $(L,M)$ with $l \geq L$ and $m \geq M$.  We then have four cases, depending on whether $L_2$ is $\vc{\xt\yt}{}$, $\vc{\yt\fact}{L}$, $\vc{\xt\fact}{M}$, or of the form $\vc{0}{r,s}$ for some $(r, s) \neq (L, M)$ with $r \leq L$ and $s \leq M$.  We restrict our attention to these four cases from now on.

In each case there is a single obvious holomorphic triangle contributing to the product.  In the first and fourth cases the triangle lies in the $\vc{\xt\yt}{}$-neck region, as illustrated in \cref{figObviousTriangles},
\begin{figure}[ht]
\centering
\begin{tikzpicture}[yscale=2]
\def\pert{0.2}
\def\ppert{\pert*tanh(20*\mq*(\x-1.4))}
\def\mpert{\pert*tanh(-20*\lp*(\x+1.4))}

\def\lp{1/2}
\def\mqa{1/3}
\def\mqb{2/3}

\begin{scope}[xshift=-0.5cm]
\draw[->] (-0.5, 0) -- (1, 0) node[anchor=north]{\tiny$\log|\xt|$};
\draw[->] (0, -0.25) -- (0, 0.5) node[anchor=east]{\tiny$\arg \xt$};
\end{scope}

\begin{scope}[xshift=3.5cm]
\draw (0, -0.9) -- (0, 0.8);

\def\lp{1/2}
\def\mq{1/3}
\draw[domain=-2:-1, samples=100, smooth] plot ({\x}, {(exp(2*\x)*(\lp-\mq*\pert)-exp(-2*\x)*(\mq+\lp*\mpert))/(exp(2*\x)+exp(-2*\x))});
\draw[domain=-1.01:1.01, samples=100, smooth] plot ({\x}, {(exp(2*\x)*(\lp-\mq*\pert)-exp(-2*\x)*(\mq-\lp*\pert))/(exp(2*\x)+exp(-2*\x))});
\draw[domain=1:2, samples=100, smooth] plot ({\x}, {(exp(2*\x)*(\lp+\mq*\ppert)-exp(-2*\x)*(\mq-\lp*\pert))/(exp(2*\x)+exp(-2*\x))});

\def\lp{1/2}
\def\mq{2/3}
\draw[domain=-2:-1, samples=100, smooth] plot ({\x}, {(exp(2*\x)*(\lp-\mq*\pert)-exp(-2*\x)*(\mq+\lp*\mpert))/(exp(2*\x)+exp(-2*\x))});
\draw[domain=-1.01:1.01, samples=100, smooth] plot ({\x}, {(exp(2*\x)*(\lp-\mq*\pert)-exp(-2*\x)*(\mq-\lp*\pert))/(exp(2*\x)+exp(-2*\x))});
\draw[domain=1:2, samples=100, smooth] plot ({\x}, {(exp(2*\x)*(\lp+\mq*\ppert)-exp(-2*\x)*(\mq-\lp*\pert))/(exp(2*\x)+exp(-2*\x))});

\draw (1, 0.2) node{\small$L_0$};
\draw (0.8, 0.55) node{\small$L_1$};
\draw (-0.3, 0.3) node{\small$L_2$};


\end{scope}

\begin{scope}[xshift=8cm]
\def\lp{0}
\def\mq{1/3}
\draw[domain=-2:-1, samples=100, smooth] plot ({\x}, {(exp(2*\x)*(\lp-\mq*\pert)-exp(-2*\x)*(\mq+\lp*\mpert))/(exp(2*\x)+exp(-2*\x))});
\draw[domain=-1.01:1.01, samples=100, smooth] plot ({\x}, {(exp(2*\x)*(\lp-\mq*\pert)-exp(-2*\x)*(\mq-\lp*\pert))/(exp(2*\x)+exp(-2*\x))});
\draw[domain=1:2, samples=100, smooth] plot ({\x}, {(exp(2*\x)*(\lp+\mq*\ppert)-exp(-2*\x)*(\mq-\lp*\pert))/(exp(2*\x)+exp(-2*\x))});

\def\lp{1/2}
\def\mq{1/3}
\draw[domain=-2:-1, samples=100, smooth] plot ({\x}, {(exp(2*\x)*(\lp-\mq*\pert)-exp(-2*\x)*(\mq+\lp*\mpert))/(exp(2*\x)+exp(-2*\x))});
\draw[domain=-1.01:1.01, samples=100, smooth] plot ({\x}, {(exp(2*\x)*(\lp-\mq*\pert)-exp(-2*\x)*(\mq-\lp*\pert))/(exp(2*\x)+exp(-2*\x))});
\draw[domain=1:2, samples=100, smooth] plot ({\x}, {(exp(2*\x)*(\lp+\mq*\ppert)-exp(-2*\x)*(\mq-\lp*\pert))/(exp(2*\x)+exp(-2*\x))});

\def\lp{1/2}
\def\mq{2/3}
\draw[domain=-2:-1, samples=100, smooth] plot ({\x}, {(exp(2*\x)*(\lp-\mq*\pert)-exp(-2*\x)*(\mq+\lp*\mpert))/(exp(2*\x)+exp(-2*\x))});
\draw[domain=-1.01:1.01, samples=100, smooth] plot ({\x}, {(exp(2*\x)*(\lp-\mq*\pert)-exp(-2*\x)*(\mq-\lp*\pert))/(exp(2*\x)+exp(-2*\x))});
\draw[domain=1:2, samples=100, smooth] plot ({\x}, {(exp(2*\x)*(\lp+\mq*\ppert)-exp(-2*\x)*(\mq-\lp*\pert))/(exp(2*\x)+exp(-2*\x))});

\draw (1, 0.2) node{\small$L_0$};
\draw (0.8, 0.55) node{\small$L_1$};
\draw (0.6, -0.25) node{\small$L_2$};


\end{scope}

\end{tikzpicture}
\caption{The obvious triangles in the $\vc{\xt\yt}{}$-neck contributing to the product in the first (left) and fourth (right) cases.\label{figObviousTriangles}}
\end{figure}
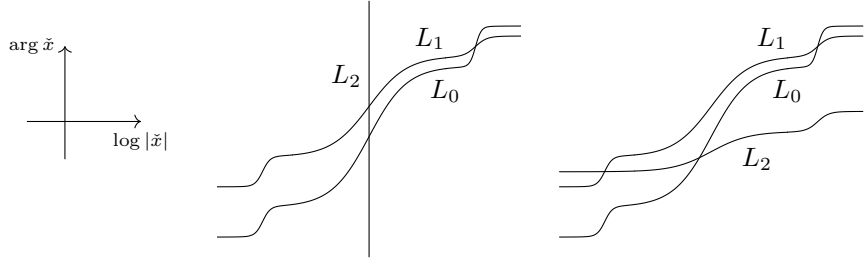
whilst in the second (respectively third) case it stretches between the $\vc{\xt\yt}{}$- and $\vc{\yt\fact}{l}$- (respectively $\vc{\xt\fact}{m}$-) neck regions in the $\xt$- (respectively $\yt$-) axis part of $\Sigma'$ as shown in \cref{figObviousTriangles2}.
\begin{figure}[ht]
\centering
\begin{tikzpicture}[blob/.style={circle, draw=black, fill=black, inner sep=0, minimum size=\blobsize, line width=0.5mm}, critpt/.style={circle, draw=black, inner sep=0, minimum size=\ballsize}, noblob/.style={inner sep=0, minimum size=0}]
\def\ballsize{1.4cm}
\def\blobsize{0.03cm}
\def\r{2.5}

\draw (0,0) node[critpt, dotted]{};
\draw (0,0) node[blob]{};
\foreach \a in {0, 4}
{
\draw ($cos(360*\a/6)*(\r, 0)+sin(360*\a/6)*(0, \r)$) node[critpt, dashed]{};
}
\foreach \a in {2}
{
\draw ($cos(360*\a/6)*(\r, 0)+sin(360*\a/6)*(0, \r)$) node(pt0)[critpt, dotted]{};
\def\ballsize{0.7cm}
\draw ($cos(360*\a/6)*(\r, 0)+sin(360*\a/6)*(0, \r)$) node(pt0)[critpt]{};
\draw ($cos(360*\a/6)*(\r, 0)+sin(360*\a/6)*(0, \r)$) node(pt0)[blob]{};
}

\def\phi{(120*exp(2*\x)-0*exp(-2*\x))/(exp(2*\x)+exp(-2*\x))};
\draw[domain=-2.1:1.5, samples=100, smooth] plot ({0.2*exp(\x)*cos(\phi)}, {0.2*exp(\x)*sin(\phi)}) node(pt1)[noblob]{};
\begin{scope}[rotate=180, shift={(pt0)}]
\draw[domain=-2.1:1.5, samples=100, smooth] plot ({0.2*exp(\x)*cos(\phi)}, {0.2*exp(\x)*sin(\phi)}) node(pt2)[noblob]{};
\end{scope}
\draw[shorten <= -0.01cm, shorten >= -0.01cm] (pt1) -- (pt2);

\def\phi{(130*exp(2*\x)-180*exp(-2*\x))/(exp(2*\x)+exp(-2*\x))};
\draw[domain=-2.1:1.5, samples=100, smooth] plot ({0.2*exp(\x)*cos(\phi)}, {0.2*exp(\x)*sin(\phi)}) node(pt1)[noblob]{};
\begin{scope}[rotate=180, shift={(pt0)}]
\def\phi{(110*exp(2*\x)-180*exp(-2*\x))/(exp(2*\x)+exp(-2*\x))};
\draw[domain=-2.1:1.5, samples=100, smooth] plot ({0.2*exp(\x)*cos(\phi)}, {0.2*exp(\x)*sin(\phi)}) node(pt2)[noblob]{};
\end{scope}
\draw[shorten <= -0.01cm, shorten >= -0.01cm] (pt1) to [bend left=5] (pt2);

\draw[->, >=stealth, line width=0.035cm] (1.5, 1.8) -- (-0.7, 1.05);

\end{tikzpicture}
\caption{The $\xt$-projection of the obvious triangle between $\vc{0}{1,0}$, $\vc{0}{1,1}$, and $\vc{\yt\fact}{1}$, when $(p,q)=(4,3)$.\label{figObviousTriangles2}}
\end{figure}
We claim that there are no other triangles, whence \eqref{eqFloerProduct} is the non-degenerate multiplication $e_{12} \otimes e_{01} \mapsto \pm e_{02}$, where $e_{ij}$ is the generator of $HF^*(L_i, L_j)$ corresponding to the unique intersection point of $L_i$ and $L_j$.  In fact there are two natural generators, differing by sign, and the $\pm$ in the multiplication depends on the specific generators chosen as well as the orientation on the moduli space of holomorphic triangles, but we shall argue shortly that all signs can be arranged to be positive.

To prove the claim, suppose $u$ is a non-constant holomorphic triangle with boundary on $L_\cup$, defined to be the union of the $L_i$.  By the open mapping theorem, after deleting $L_\cup$ the image of $u$ consists of a union of components of $\Sigma \setminus L_\cup$ whose closures in $\Sigma$ are compact.  Such components naturally correspond to generators of $\mathrm{H}_2(\Sigma/L_\cup; \Z) \cong \mathrm{H}_2(\Sigma, L_\cup; \Z)$, which the long exact sequence of the pair tells us is isomorphic to the kernel of the inclusion pushforward $\mathrm{H}_1(L_\cup; \Z) \rightarrow \mathrm{H}_1(\Sigma; \Z)$.  In all of our cases, the space $L_\cup$ is homeomorphic to three circles that touch pairwise, so its $\mathrm{H}_1$ has rank four.  Its image in $\mathrm{H}_1(\Sigma; \Z)$ meanwhile, contains the classes of $L_0$, $L_1$ and $L_2$, which are linearly independent since the vanishing cycles form a basis for $\mathrm{H}_1(\Sigma; \Z)$.  We conclude that $\mathrm{H}_2(\Sigma, L_\cup; \Z)$ has rank at most one, so there is at most one component of $\Sigma \setminus L_\cup$ that $u$ can enter.  We have already seen that there is at least one component, and counted the obvious triangle that it contributes, so we conclude that there are no other triangles.

To compute the signs we should equip each $L_i$ with an orientation and the non-trivial spin structure (this is the one that is induced by viewing $L_i$ as the boundary of a Lefschetz thimble in the total space of our Morsified fibration), and then calculate the induced orientation on the moduli space of holomorphic triangles.  As mentioned above, however, we can choose the generators of the morphism spaces so that all of the signs turn out to be positive.  We make these choices by induction on the \emph{length} of the morphism, defined to be the maximal length of a chain of non-identity morphisms whose composition is the given morphism (so, for example, the length of a generator of $HF^*(\vc{0}{l,m}, \vc{0}{L,M})$ is $l-L+m-M$).

First, choose arbitrary signs for the generators of length $1$.  Now modify these as follows.  Start at the bottom left-hand square in the quiver picture \cref{figLoopQuiver}---explicitly this corresponds to the square
\[
\begin{tikzcd}
\vc{0}{p-1,q-2} \arrow{r} & \vc{0}{p-2,q-2}
\\ \vc{0}{p-1,q-1} \arrow{u} \arrow{r} & \vc{0}{p-2,q-1} \arrow{u}
\end{tikzcd}
\]
If this commutes then do nothing, otherwise reverse the sign of the morphism along the top edge.  Then consider the next square to the right and do the same, and continue all the way along to the bottom right-hand square.  Now run the same procedure on the next row of squares up, and then the next, all the way to the top.  In this way we obtain sign choices for all generators of length $1$ such that the small squares commute.

For each morphism space of length $k>1$ we choose its generator by expressing the space as a composition of $k$ morphism spaces of length $1$ and taking the positive generator of each factor.  There may be several different ways of decomposing the space into length $1$ factors, but any two can be joined by a chain of moves where one commutes across a small square.  We have arranged it so that these moves have no effect, so there is no ambiguity in the overall procedure.  This proves that all signs can be taken to be $+$.

We conclude:

\begin{prop}
\label{ALoopUngraded}
There is a model for $\mathcal{A}$ which, under the identification \eqref{ObjectMatch}, is described by the quiver in \cref{figLoopQuiver} up to yet-to-be-determined gradings.\hfill$\qed$
\end{prop}

\subsection{Gradings and completing the proof}

Recall from \cite{SeidelGraded, SeidelBook} that to equip the Fukaya category of a symplectic manifold $X$ with a $\Z$-grading one must choose a homotopy class of trivialisation of the square $K_X^{-2}$ of the anticanonical bundle of $X$; this is possible if and only if $2c_1(X)$ vanishes in $\mathrm{H}^2(X; \Z)$, and in this case the set of choices forms a torsor for $\mathrm{H}^1(X; \Z)$.  We are interested in the Fukaya--Seidel category $\mathcal{F}(\wt)$ and the subcategory $\mathcal{A}$ of the compact Fukaya category of the smooth fibre, for which the relevant choices of $X$ are $\C^2$ and $\Sigma$ respectively.  The former has a unique grading, defined by the section $\sigma = (\partial_{\xt} \wedge \partial_{\yt})^{\otimes 2}$ of $K_{\C^2}^{-2}$, which induces a grading of the latter, and it is with respect to this induced grading that the quasi-equivalence $\Tw \mathcal{A} \rightarrow \mathcal{F}(\wt)$ is graded.

Trivialisations of $K_\Sigma^{-2}$ correspond naturally to line fields $\ell$ on $\Sigma$, i.e.~sections of the real projectivisation $\P_\R T\Sigma$ of the tangent bundle, and given a choice of $\ell$ the Lagrangian $L$ represented by an embedded curve $\gamma : S^1 \rightarrow \Sigma$ is gradable if and only if the sections $\gamma^* \ell$ and $\gamma^* TL$ of $\gamma^* \P_\R T\Sigma$ are homotopic.  In this case a \emph{grading} of $L$ is a homotopy class of homotopy between them.  At each point of $L$ we can measure the anticlockwise angle from $\ell$ to $TL$, and we denote this by $\pi \alpha$, where $\alpha$ is an element of $\R/\Z$.  The gradings of $L$ are then in bijection with lifts $\alpha^\#$ of this element to $\R$.  Given two graded Lagrangians $L_0$ and $L_1$, which intersect transversely at a point $x$, let their corresponding lifts at $x$ be $\alpha_0^\#$ and $\alpha_1^\#$ respectively.  By \cite[Example 11.20]{SeidelBook}, the grading of $x$ as a generator of the Floer complex $CF^*(L_0, L_1)$ is then given by
\begin{equation}
\label{eqDegree}
\lfloor \alpha_1^\# - \alpha_0^\# \rfloor + 1.
\end{equation}
From now on we will use $\ell$ to denote the specific (homotopy class of) line field corresponding to the grading on $\Sigma$ induced by the grading on $\C^2$.

To compute $\ell$ note that each Lefschetz thimble $\Delta$ in $\mathcal{F}(\wt)$ is gradable with respect to $\sigma$ ($\Delta$ is contractible so the grading obstruction trivially vanishes), and each choice of grading induces a grading of the corresponding vanishing cycle $V \subset \Sigma$ with respect to $\ell$.  In particular, all of the vanishing cycles $V$ are gradable with respect to $\ell$, and since they form a basis for $\mathrm{H}_1(\Sigma; \Z)$ this property---that $\ell$ has winding number zero around each $V$---determines $\ell$ uniquely.

\begin{rmk}
Recall that the ordering on $\mathcal{A}$ is determined by a choice of starting direction in the base $\C$, and strictly this choice enters into the construction of the bijection between gradings of a thimble $\Delta$ and of the corresponding vanishing cycle $V$.  This is unimportant for our present purposes, but we will see a manifestation of it in \cref{BPAModel}, where a change in this direction leads to a change in the grading of a vanishing cycle.
\end{rmk}

Using this characterisation, one can draw $\ell$ as shown in \cref{figLineField}: the left-hand diagram depicts a foliation of the line $\{u+v=\eps\}$ with the points $(\eps, 0)$ and $(0, \eps)$ deleted, and we lift its tangent distribution to give the line field on the branched cover comprising the $\{\fact=\eps\}$ part of $\Sigma'$ and the attached neck regions; the right-hand diagram depicts a foliation whose tangent distribution gives the line field on the $\xt$-axis part of $\Sigma'$ and the attached neck regions in the case $q=4$---it is clear how this generalises to other values of $q$ and that a similar picture can be drawn for the $\yt$-axis part.
\begin{figure}[ht]
\centering
\begin{tikzpicture}[scale=2, ball/.style={circle, dotted, draw=black, inner sep=0, minimum size=\ballsize}]
\def\ballsize{1cm}

\begin{scope}
\clip (-1.8, -1.2) rectangle (1.8, 1.2);
\draw (1,0) node[ball]{};
\draw (-1,0) node[ball]{};
\draw (-2, 0) -- (2, 0);
\draw (-1, 0) to [bend right=30, looseness=0.8] (1, 0);
\draw (-1, 0) to [bend right=60, looseness=1] (1, 0);
\draw (-1, 0) to [bend right=90, looseness=1.3] (1, 0);
\draw (-1, 0) to [bend right=120, looseness=2] (1, 0);
\draw (-1, 0) to [bend right=150, looseness=4.3] (1, 0);
\draw (-1, 0) to [bend right=-30, looseness=0.8] (1, 0);
\draw (-1, 0) to [bend right=-60, looseness=1] (1, 0);
\draw (-1, 0) to [bend right=-90, looseness=1.3] (1, 0);
\draw (-1, 0) to [bend right=-120, looseness=2] (1, 0);
\draw (-1, 0) to [bend right=-150, looseness=4.3] (1, 0);
\end{scope}

\begin{scope}[xshift=4cm]
\clip (-1.8, -1.2) rectangle (1.8, 1.2);
\draw (0,0) node[ball]{};
\foreach \theta in {0, 120, 240}
{
\begin{scope}[rotate=\theta]
\draw (1,0) node[ball]{};
\draw (0, 0) -- (3, 0);
\draw (0, 0) -- ($cos(60)*(3,0)+sin(60)*(0,3)$);
\draw (0, 0) .. controls (0.5, 0.2) and (0.6, 0.2) .. (1, 0);
\draw (0, 0) .. controls (0.5, 0.4) and (1.2, 0.5) .. (1, 0);
\draw (0, 0) .. controls (0.8, 0.8) and (1.5, 0.7) .. (1, 0);
\draw (0, 0) .. controls (1.2, 1.6) and (2.6, 0.8) .. (1, 0);
\draw (0, 0) .. controls (0.5, -0.2) and (0.6, -0.2) .. (1, 0);
\draw (0, 0) .. controls (0.5, -0.4) and (1.2, -0.5) .. (1, 0);
\draw (0, 0) .. controls (0.8, -0.8) and (1.5, -0.7) .. (1, 0);
\draw (0, 0) .. controls (1.2, -1.6) and (2.6, -0.8) .. (1, 0);
\end{scope}
}
\end{scope}

\end{tikzpicture}
\caption{Foliations defining the line field $\ell$ used to grade $\Sigma$.\label{figLineField}}
\end{figure}
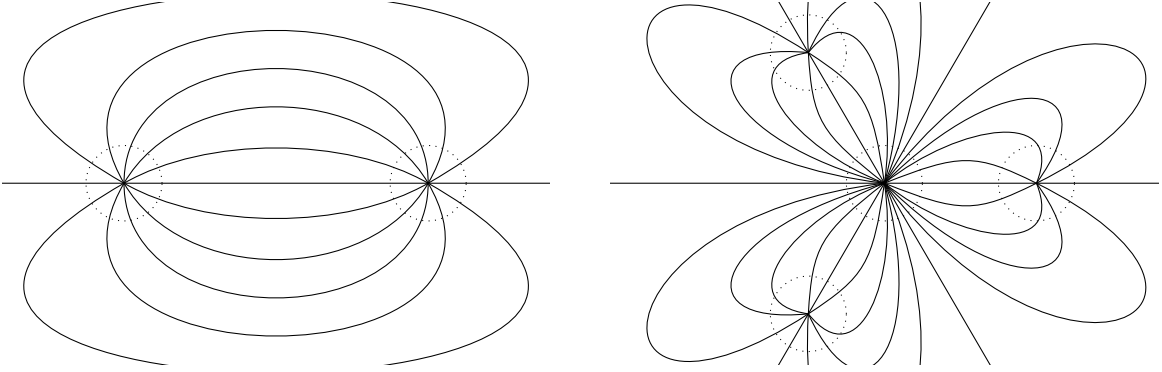
As usual, the dotted circles represent the boundaries of the neck regions.  Note that on each neck region the line field is longitudinal, so the different pictures glue together.

Each $\vc{0}{l,m}$ is approximately tangent to $\ell$ along its approximately straight segments in the three components of $\Sigma'$, and we choose to grade it so that the homotopy from $TL$ to $\ell$ is approximately constant on these regions.  This is consistent, in the sense that these homotopies patch together across the neck regions.  On each neck region, the lift $\alpha^\#$ is valued approximately between $0$ and $1/2$, and where two of these cycles intersect the one with the greater value of $\theta_{l,m}$ is `steeper' and hence has greater $\alpha^\#$.  We conclude that for distinct $(l,m)$ and $(L,M)$ with $l\geq L$ and $m\geq M$ the generator of $HF^*(\vc{0}{l,m}, \vc{0}{L,M})$ lies in degree $0$ (in the notation of \eqref{eqDegree} we have $1/2 > \alpha_0^\# > \alpha_1^\# > 0$).

Each of the other vanishing cycles is a waist curve on a neck region and as such is orthogonal to the line field.  We grade it so that the lift $\alpha^\#$ is $-1/2$.  This puts the generators of
\[
HF^*(\vc{0}{l,m}, \vc{\yt\fact}{l}) \text{, } HF^*(\vc{0}{l,m}, \vc{\xt\fact}{m}) \text{, and } HF^*(\vc{0}{l,m}, \vc{\xt\yt}{})
\]
all in degree $0$.  This means that the identification \eqref{ObjectMatch} matches up gradings, and we deduce:

\begin{thm}[\cref{Thm1}, loop polynomial case]
Under \eqref{ObjectMatch}, the $\Z$-graded $A_\infty$-category $\mathcal{A}$ is described by the quiver with relations in \cref{figLoopQuiver} and is formal.  In particular, by \cref{LoopEndAlgebra} it is quasi-equivalent to $\mathcal{B}$, and hence there is an induced quasi-equivalence
\[
\mathrm{mf}(\C^2, \Gamma_\w, \w) \simeq \mathcal{F}(\wt).
\]
\end{thm}
\begin{proof}
The cohomology-level version of the first statement follows from \cref{ALoopUngraded} plus the above grading computations.  Formality then follows immediately from directedness and the fact that the morphisms are concentrated in degree $0$ as in \cref{LoopEndAlgebra}.  This shows that $\mathcal{A}$ and $\mathcal{B}$ are quasi-equivalent, and the final statement then follows from the argument outlined in \cref{sscOutline}.
\end{proof}


\section{B-model for chain polynomials}
\label{BModelChain}

\subsection{The basic objects}
We now deal with the case of the chain polynomial $\w=x^py+y^q$.  This time the maximal grading group $L$ is the abelian group freely generated by $\vec{x}$, $\vec{y}$ and $\vec{c}$ modulo the relations 
\[
p\vec{x}+\vec{y}=q\vec{y}=\vec{c}.
\]
In contrast to the loop case, we have $L/\Z\vec{c}\cong \Z/(pq)$, generated by $\vec{x}$ but not by $\vec{y}=-p\vec{x}$. In keeping with our earlier notation let $S$ be the $L$-graded algebra $\C[x,y]$, with $x$ and $y$ in degrees $\vec{x}$ and $\vec{y}$ respectively, and let $R=S/(\w)$.  Let $\fac$ now denote $x^{p}+y^{q-1}$ so that $\w=y\fac$.

The stack $[\w^{-1}(0)/\Gamma_\w]$ has two components, whose structure sheaves correspond to the matrix factorisations
\[
\obj{y}{}^\bullet = ( \cdots \rightarrow S(-\vec{c}) \xrightarrow{\fac} S(-\vec{y}) \xrightarrow{y} S \rightarrow \cdots ),
\]
and
\[
\obj{\fac}{}^\bullet = ( \cdots \rightarrow S(-\vec{c}) \xrightarrow{y} S(-\vec{c}+\vec{y}) \xrightarrow{\fac} S \rightarrow \cdots ).
\]
We will need the shifts
\[
\obj{y}{j} = \obj{y}{}((j+1-q)\vec{y}) \quad \text{for } j=1, \dots, q-1.
\]
Note that $\obj{\fac}{}[1] \cong \obj{y}{}(\vec{y})$.

The unique singular point of the stack is still the origin, and the objects we need that are supported at this point are the $\obj{0}{i,j}$ defined by
\begin{center}
	\begin{tikzcd}[row sep=7ex, column sep=10ex]
		S(\vec{y}) \arrow{r}{y^j} \arrow{dr}[near start, outer sep=-2pt]{-x^i} \ar[d, phantom, description, "\cdots\hskip7ex\bigoplus\phantom{\hskip7ex\cdots}"]
		&  S((j+1)\vec{y}) \arrow{r}{y^{q-j}} \arrow{dr}[near start]{x^i} \ar[d, phantom, description, "\bigoplus"]
		& S(\vec{c}+\vec{y}) \ar[d, phantom, description, "\phantom{\cdots\hskip7ex}\bigoplus\hskip7ex\cdots"]
		\\ S(-\vec{c}+i\vec{x}+(j+1)\vec{y}) \arrow{ur}[near start, outer sep=-1pt]{x^{p-i}y} \arrow{r}{y^{q-j}}
		& S(i\vec{x}+\vec{y}) \arrow{ur}[near start, outer sep=-1pt]{-x^{p-i}y} \arrow{r}{y^j}
		& S(i\vec{x}+(j+1)\vec{y})
	\end{tikzcd}
\end{center}
for $i=1, \dots, p-1$ and $j=1, \dots, q-1$, obtained by stabilising $R(i\vec{x}+(j+1)\vec{y})/(x^i,y^j)$.


\subsection{Morphisms between the $\obj{y}{}$'s and $\obj{\fac}{}$}

For all $l$ in $L$, and all integers $m$, we have $\Hom^{2m}(\obj{y}{},\obj{y}{}(l)) \cong (R/(y, \fac))_{m\vec{c}+l}$ and $\Hom^{2m-1}(\obj{y}{},\obj{y}{}(l))=0$.  The analogue of \cref{GradingIdeal} and \cref{GradingZero}, proved by similar arguments, is now:

\begin{lem}
\label{ChainGrading}
Suppose $a$ and $b$ are integers, with $a \leq p-1$, and $s$ is an element of $S$ (or $R$) which is homogeneous modulo $\vec{c}$, of degree $a\vec{x}+b\vec{y} \mod \vec{c}$.  Then:
\begin{enumerate}[(i)]
\item\label{chgr1} $s$ is divisible by $x^a$.
\item\label{chgr2} If also $b \leq q-1$ then $s$ lies in the ideal $(x^ay^b, x^{p+a})$.
\item\label{chgr3} If $a=b=0$ then the non-constant terms of $s$ lie in $(x^{pq}, x^py, y^q)$.\hfill$\qed$
\end{enumerate}
\end{lem}

Applying this to the above computation we obtain:

\begin{lem}
The objects $\obj{y}{1}, \dots \obj{y}{q-1}$ are exceptional and pairwise orthogonal.\hfill$\qed$
\end{lem}

Using the fact that $\obj{\fac}{}[1] \cong \obj{y}{}(\vec{y})$, we also get:

\begin{lem}
The object $\obj{\fac}{}$ is exceptional and is orthogonal to the $\obj{y}{j}$. \hfill$\qed$
\end{lem}


\subsection{Morphisms between $\obj{y}{}$'s and $\obj{\fac}{}$ and $\obj{0}{}$'s}

For all $l$ and all $(i,j)$, $\Hom^\bullet(\obj{y}{}(l), \obj{0}{i,j})$ is given by the cohomology of the complex
\[
\cdots \rightarrow (R/(x^i, y^j))_{i\vec{x}+(j+1)\vec{y}-l} \xrightarrow{y} (R/(x^i, y^j))_{i\vec{x}+(j+2)\vec{y}-l} \xrightarrow{\fac} (R/(x^i, y^j))_{\vec{c}+i\vec{x}+(j+1)\vec{y}-l} \rightarrow \cdots.
\]
By \cref{ChainGrading}(\ref{chgr1}) we see that for all $J$
\[
\Hom^\bullet(\obj{y}{J}, \obj{0}{i,j}) = \Hom^\bullet(\obj{\fac}{}, \obj{0}{i,j}) = 0.
\]

Morphisms in the other directions are computed by the complex
\begin{center}
	\begin{tikzcd}[row sep=7ex, column sep=10ex]
		(R/(y))_{-\vec{c}-\vec{y}+l} \arrow{r}{y^{q-j}}\arrow{dr}[near start, outer sep=-2pt]{-x^{p-i}y} \ar[d, phantom, description, "\cdots\hskip7ex\bigoplus\phantom{\hskip7ex\cdots}"]
		& (R/(y))_{-(j+1)\vec{y}+l} \arrow{r}{y^j} \arrow{dr}[near start]{x^{p-i}y} \ar[d, phantom, description, "\bigoplus"]
		& (R/(y))_{-\vec{y}+l} \ar[d, phantom, description, "\phantom{\cdots\hskip7ex}\bigoplus\hskip7ex\cdots"]
		\\ (R/(y))_{-i\vec{x}-(j+1)\vec{y}+l} \arrow{ur}[near start, outer sep=-1pt]{x^i} \arrow{r}{y^j}
		& (R/(y))_{-i\vec{x}-\vec{y}+l} \arrow{ur}[near start, outer sep=-1pt]{-x^i} \arrow{r}{y^{q-j}}
		& (R/(y))_{\vec{c}-i\vec{x}-(j+1)\vec{y}+l}
	\end{tikzcd}
\end{center}
The only non-vanishing differentials are $x^i$, so we get
\begin{gather*}
\Hom^{2m}(\obj{0}{i,j}, \obj{y}{J}) \cong (R/(x^i, y))_{(m-2)\vec{c}+J\vec{y}},
\\ \Hom^{2m+1}(\obj{0}{i,j}, \obj{y}{J}) \cong (R/(x^i, y))_{(m-1)\vec{c}+(J-j)\vec{y}},
\\ \Hom^{2m}(\obj{0}{i,j}, \obj{\fac}{}) \cong (R/(x^i, y))_{(m-1)\vec{c}-j\vec{y}},
\\ \Hom^{2m+1}(\obj{0}{i,j}, \obj{\fac}{}) \cong (R/(x^i, y))_{(m-1)\vec{c}},
\end{gather*}
and hence:

\begin{lem}
	In $\mathrm{HMF}(\C^2,\Gamma_\w,\w)$ there are no morphisms from $\obj{y}{J}$ or $\obj{\fac}{}$ to $\obj{0}{i,j}$. The morphism spaces in the other direction are spanned by
\[
(1, 0) \in \Hom^3(\obj{0}{i,j}, \obj{y}{j})
\]
and
\[
(0, 1) \in \Hom^3(\obj{0}{i,j}, \obj{\fac}{})
\]
in the above complexes.
\end{lem}
\begin{proof}
The even degree morphisms all vanish by \cref{ChainGrading}(\ref{chgr2}), and if $j\neq J$ then the same holds for $\Hom^{2m+1}(\obj{0}{i,j}, \obj{y}{J})$ (if $J<j$ then rewrite the grading as $(q+J-j)\vec{y} \mod \vec{c}$).  \Cref{ChainGrading}(\ref{chgr3}) tells us that the only surviving odd morphisms are the constants.
\end{proof}


\subsection{Morphisms between the $\obj{0}{}$'s}

The complex computing $\Hom^\bullet(\obj{0}{i,j}, \obj{0}{I,J})$ is
\begin{center}
	\begin{tikzcd}[row sep=7ex, column sep=10ex]
		(R/(x^I, y^J))_{-\vec{c}+I\vec{x}+J\vec{y}} \arrow{r}{y^{q-j}}\arrow{dr}[near start, outer sep=-2pt]{-x^{p-i}y} \ar[d, phantom, description, "\cdots\hskip7ex\bigoplus\phantom{\hskip7ex\cdots}"]
		& (R/(x^I, y^J))_{I\vec{x}+(J-j)\vec{y}} \arrow{r}{y^j} \arrow{dr}[near start]{x^{p-i}y} \ar[d, phantom, description, "\bigoplus"]
		& (R/(x^I, y^J))_{I\vec{x}+J\vec{y}} \ar[d, phantom, description, "\phantom{\cdots\hskip7ex}\bigoplus\hskip7ex\cdots"]
		\\ (R/(x^I, y^J))_{(I-i)\vec{x}+(J-j)\vec{y}} \arrow{ur}[near start, outer sep=-1pt]{x^i} \arrow{r}{y^j}
		& (R/(x^I, y^J))_{(I-i)\vec{x}+J\vec{y}} \arrow{ur}[near start, outer sep=-1pt]{-x^i} \arrow{r}{y^{q-j}}
		& (R/(x^I, y^J))_{\vec{c}+(I-i)\vec{x}+(J-j)\vec{y}}
	\end{tikzcd}
\end{center}
The top row vanishes by \cref{ChainGrading}(\ref{chgr1}), and the same is true of the bottom row if $I < i$  (after adding $p\vec{x}+\vec{y} \mod \vec{c}$ to the gradings), so assume that $I\geq i$.  The complex becomes
\[
\cdots \rightarrow (R/(x^I, y^J))_{(I-i)\vec{x}+(J-j)\vec{y}} \xrightarrow{y^j} (R/(x^I, y^J))_{(I-i)\vec{x}+J\vec{y}} \xrightarrow{y^{q-j}} (R/(x^I, y^J))_{\vec{c}+(I-i)\vec{x}+(J-j)\vec{y}} \rightarrow \cdots
\]
and the odd position terms vanish by \cref{ChainGrading}(\ref{chgr2}), so $\Hom^{2m+1}(\obj{0}{i,j}, \obj{0}{I,J})=0$ and
\[
\Hom^{2m}(\obj{0}{i,j}, \obj{0}{I,J}) \cong (R/(x^I, y^J))_{m\vec{c}+(I-i)\vec{x}+(J-j)\vec{y}}.
\]
If $J<j$ then this is zero by \cref{ChainGrading}(\ref{chgr2}) (after adding $q\vec{y} \mod \vec{c}$ to the grading), so assume $J \geq j$.  \Cref{ChainGrading}(\ref{chgr2}) tells us that any element is divisible by $x^{I-i}y^{J-j}$ modulo $(x^I, y^J)$, and then \cref{ChainGrading}(\ref{chgr3}) tells us that only constant multiples survive.  We conclude:

\begin{lem}
For all $(i,j)$ and $(I,J)$ we have that
\begin{flalign*}
&& \Hom^\bullet (\obj{0}{i,j}, \obj{0}{I,J}) &\cong \begin{cases} \C \cdot x^{I-i}y^{J-j} & \text{if } I\geq i \text{, } J \geq j \text{ and } \bullet=0 \\ 0 & \text{otherwise.} \end{cases} & \qed
\end{flalign*}
\end{lem}


\subsection{The total endomorphism algebra of the basic objects}

It is easy to compute the compositions between the morphisms and obtain the following description of the full subcategory $\mathcal{B}$ of $\mathrm{mf}(\C^2, \Gamma_\w, \w)$ on the objects $\obj{y}{j}[3]$, $\obj{\fac}{}[3]$, $\obj{0}{i,j}$:

\begin{thm}
	\label{ChainEndAlgebra}
	The cohomology category $\mathrm{H}(\mathcal{B})$ is the path algebra of the quiver-with-relations described in \cref{figChainQuiver}.  Any $\Z$-graded $A_\infty$-structure on this algebra---and hence in particular that induced from the dg-structure on $\mathrm{mf}(\C^2, \Gamma_\w, \w)$---is formal.\hfill$\qed$	
\begin{figure}[ht]
\centering
\begin{tikzpicture}[blob/.style={circle, draw=black, fill=black, inner sep=0, minimum size=\blobsize}, arrow/.style={->, shorten >=6pt, shorten <=6pt}]
\def\blobsize{1.2mm}

\draw (0, 0) node[blob]{};
\draw (1, 0) node[blob]{};
\draw (2, 0) node[blob]{};
\draw (3.25, 0) node{\ $\cdots$};
\draw (4.5, 0) node[blob]{};
\draw (6, 0) node[blob]{};

\draw[arrow] (0, 0) -- (1, 0);
\draw[arrow] (1, 0) -- (2, 0);
\draw[arrow] (2, 0) -- (3, 0);
\draw[arrow] (3.5, 0) -- (4.5, 0);
\draw[arrow] (4.5, 0) -- (6, 0);

\draw[arrow] (0, 0) -- (0, 1);
\draw[arrow] (1, 0) -- (1, 1);
\draw[arrow] (2, 0) -- (2, 1);
\draw[arrow] (4.5, 0) -- (4.5, 1);

\begin{scope}[yshift=1cm]
\draw (0, 0) node[blob]{};
\draw (1, 0) node[blob]{};
\draw (2, 0) node[blob]{};
\draw (3.25, 0) node{\ $\cdots$};
\draw (4.5, 0) node[blob]{};
\draw (6, 0) node[blob]{};

\draw[arrow] (0, 0) -- (1, 0);
\draw[arrow] (1, 0) -- (2, 0);
\draw[arrow] (2, 0) -- (3, 0);
\draw[arrow] (3.5, 0) -- (4.5, 0);
\draw[arrow] (4.5, 0) -- (6, 0);

\draw[arrow] (0, 0) -- (0, 1);
\draw[arrow] (1, 0) -- (1, 1);
\draw[arrow] (2, 0) -- (2, 1);
\draw[arrow] (4.5, 0) -- (4.5, 1);
\end{scope}

\begin{scope}[yshift=3.5cm]
\draw (0, 0) node[blob]{};
\draw (1, 0) node[blob]{};
\draw (2, 0) node[blob]{};
\draw (3.25, 0) node{\ $\cdots$};
\draw (4.5, 0) node[blob]{};
\draw (6, 0) node[blob]{};

\draw[arrow] (0, 0) -- (1, 0);
\draw[arrow] (1, 0) -- (2, 0);
\draw[arrow] (2, 0) -- (3, 0);
\draw[arrow] (3.5, 0) -- (4.5, 0);
\draw[arrow] (4.5, 0) -- (6, 0);
\end{scope}

\begin{scope}[yshift=5cm]
\draw (6, 0) node[blob]{};
\end{scope}

\draw (0, 2.37) node{$\vdots$};
\draw (1, 2.37) node{$\vdots$};
\draw (2, 2.37) node{$\vdots$};
\draw (4.5, 2.37) node{$\vdots$};
\draw (6, 2.37) node{$\vdots$};

\begin{scope}[yshift=2.5cm]
\draw[arrow] (0, 0) -- (0, 1);
\draw[arrow] (1, 0) -- (1, 1);
\draw[arrow] (2, 0) -- (2, 1);
\draw[arrow] (4.5, 0) -- (4.5, 1);
\end{scope}

\draw (3.25, 2.37) node{\ $\iddots$};
\draw[arrow] (4.5, 3.5) -- (6, 5);

\draw[line width=1mm, opacity=0.2] (-0.5, -0.5) rectangle (5, 4);
\draw[line width=1mm, opacity=0.2] (5.5, -0.5) rectangle (6.5, 4);

\draw (-1.2, 2) node{$\obj{0}{i,j}$};
\draw (7.25, 2) node{$\obj{y}{j}[3]$};
\draw (6.6, 5) node{$\obj{\fac}{}[3]$};

\draw[->, dashed, rounded corners] (4.72, 0.15) -- (4.72, 0.78) -- (5.85, 0.78);
\draw[->, dashed, rounded corners] (4.72, 2.65) -- (4.72, 3.28) -- (5.85, 3.28);

\draw (10.6, 2.25) node{\parbox{115pt}{\small \ \textbf{Relations:} \begin{enumerate}[(i)]\item Squares commute\item Dashed compositions vanish\end{enumerate}}};

\end{tikzpicture}
\caption{The quiver describing the category $\mathcal{B}$ for chain polynomials.\label{figChainQuiver}}
\end{figure}
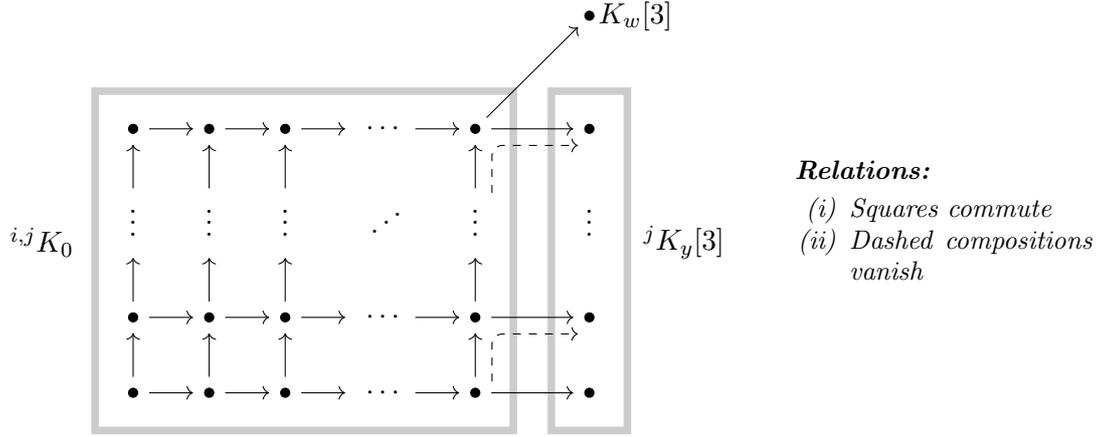
	
\end{thm}


\subsection{Generation}

The final thing we need to check is:

\begin{lem}
The objects in $\mathcal{B}$ split-generate $\mathrm{HMF}(\C^2, \Gamma_\w, \w)$.
\end{lem}
\begin{proof}
Let $V=\{\obj{y}{j}, \obj{\fac}{},\obj{0}{i,j}\}$. As in the loop case, it suffices to prove that the category $\langle V\rangle$ contains all of the $L/\Z\vec{c}$-grading shifts of $R/(x,y)$.  Again following the loop case, we easily have that $R(i\vec{x}+(j+1)\vec{y})/(x, y)$ lies in $\langle V\rangle$ for any $1\leq i \leq p-1$ and $1\leq j\leq q-1$.

By combining $\obj{\fac}{} \cong \obj{y}{}(\vec{y})[-1]$, the $\obj{y}{j}$, and all of their $[\cdot]$-shifts, we see that $\langle V \rangle$ contains $R(l)/(y)$ for all $l$ in $\Z\vec{y}+\Z\vec{c}$ (the $\Z\vec{c}$ is redundant here but we include it for clarity).  Consequently, for each integer $j$ we have that $\langle V \rangle$ contains the cokernel of
\[
R((j+1)\vec{y}-\vec{c})/(y) \xrightarrow{x^p} R(j\vec{y})/(y),
\]
which is $R(j\vec{y})/(x^p, y)$.  Peeling off one-dimensional pieces $R(i\vec{x}+j\vec{y})/(x,y)$ for $i=-1, \dots, -(p-1)$, by taking cones, we're left with $R(j\vec{y})/(x,y)$.  If $j$ lies in $1, \dots, q-1$ then (after applying the trivial operation $(p\vec{x}+\vec{y})[-2]$) each of these pieces is in $\langle V \rangle$ by the previous paragraph.  The conclusion is that $R(j\vec{y})/(x,y)$ lies in $\langle V \rangle$ for all such $j$.

We have therefore constructed $R(a\vec{x}+b\vec{y})/(x,y)$ for $0 \leq a \leq p-1$ and $0 \leq b \leq q-1$ except for $(a,b)=(0,0)$ and $(a,b)=(1,1), \dots, (1, p-1)$.  To obtain the latter, consider the extension
\[
0\rightarrow R/(w)\xrightarrow{x^i} R(i\vec{x})/(w)\rightarrow R(i\vec{x})/(x^i,y^{q-1})\rightarrow 0
\]
for $i=1, \dots, p-1$.  The outer terms lie in $\langle V \rangle$ (they are $\obj{w}{}$ and $\obj{0}{i, q-1}[-2]$), so we deduce that $R(i\vec{x})/(w)$ also lies in $\langle V\rangle$.  Again using the fact that $\obj{w}{}\cong\obj{y}{}(\vec{y})[-1]$, we get that $R(i\vec{x}+\vec{y})/(y)$ is in $\langle V \rangle$ for $i=0, \dots, p-1$ (the $i=0$ case comes from $\obj{w}{}[1]$ itself, not from the preceding argument).  From these we see that
\[
R(i\vec{x}+\vec{y})/(x,y)\cong\Cone\big(R((i-1)\vec{x}+y)/(y)\xrightarrow{x}R(i\vec{x}+\vec{y})/(y)\big)
\]
lies in $\langle V\rangle$ for $i=1, \dots, p-1$.

All that is left to show now is that we have $R/(x,y)$ in $\langle V\rangle$, and this closely follows the loop case: we can realise this module as the cokernel of
\[
R(-\vec{x})/(x^{pq-1},y)\xrightarrow{x}R/(x^{pq},y),
\]
and the domain can be built of the shifts of $R/(x,y)$ that we already have.  The codomain, meanwhile, is given by
\[
\Cone\big(R(-(q-1)\vec{c})/(y)\xrightarrow{x^{pq}}R/(y)\big).\qedhere
\]
\end{proof}

\begin{rmk}
\label{rmkGenVsSplitChain}
The $R(l)/(x,y)$ still only \emph{split-}generate the category (which we saw for loop polynomials in \cref{rmkGenVsSplit}), since the above proof shows that they are annihilated by the homomorphism
\[
K_0(\mathrm{mf}(\C^2, \Gamma_\w, \w)) \rightarrow \Z/2
\]
which sends the basis elements $\obj{0}{i,j}$ to $0$ but $\obj{y}{j}$ and $\obj{\fac}{}$ to $1$.
\end{rmk}

As in the loop case, we deduce:

\begin{thm}[\cref{Thm2}, chain polynomial case]
	The object
	\[
	\mathcal{E} \coloneqq \bigg( \bigoplus_{\substack{i=1, \dots, p-1\\j=1, \dots, q-1}} \obj{0}{i,j} \bigg) \oplus \bigg( \bigoplus_{j=1}^{q-1} \obj{y}{j}[3] \bigg) \oplus \obj{w}{}[3]
	\]
	is a tilting object for $\mathrm{mf}(\C^2, \Gamma_\w, \w)$.\hfill$\qed$
\end{thm}

This was proved by Futaki--Ueda \cite[Section 4]{FutakiUedaD} in the case $q=2$.


\section{A-model for chain polynomials}
\label{AModelChain}

\subsection{The setup}

Just as for the B-model, our basic strategy for understanding the A-model will closely follow the loop polynomial case.  This time the Berglund--H\"ubsch transpose is $\wt = \xt^p + \xt\yt^q$, and our starting point is once more the resonant Morsification $\wt_\eps = \wt - \eps \xt\yt$ for small positive real $\eps$.  We denote $\xt^{p-1} + \yt^q$ by $\fact$.  The critical points now fall into three types:
\begin{enumerate}[(i)]
\item $\xt=0$, $\yt^{q-1}=\eps$
\item $\xt=\yt=0$
\item\label{crit3Chain} $\yt^{q-1} = \frac{\eps}{q}$, $\xt^{p-1} = \frac{(q-1)\eps \yt}{pq}$.
\end{enumerate}
The first two types have critical value zero, whilst the third type has critical value
\[
-\xt\yt\eps(p-1)(q-1)/pq,
\]
on the ray through $-\xt\yt$.  These critical points are indeed all Morse.

There is a unique positive real solution to (\ref{crit3Chain}) which we denote by $(\xt^+_\mathrm{crit}, \yt^+_\mathrm{crit})$, and again we call the corresponding (negative real) critical value $c_\mathrm{crit}$.  Still letting $\zeta$ and $\eta$ denote the roots of unity
\[
\zeta = e^{2\pi i/(p-1)} \quad \text{and} \quad \eta = e^{2\pi i/(q-1)},
\]
but now also letting $\mu = e^{2\pi i/(p-1)(q-1)}$, the type (\ref{crit3Chain}) critical points are
\[
\{(\zeta^l\mu^m\xt^+_\mathrm{crit}, \eta^m\yt^+_\mathrm{crit}) : 0 \leq l \leq p-2 \text{, } 0 \leq m \leq q-2\},
\]
with critical values $\mu^{(q-1)l+pm}c_\mathrm{crit}$.

Taking regular fibre $\Sigma = \wt_\eps^{-1}(-\delta)$ with $0 < \delta \ll \eps$, we again choose the straight line segment from $-\delta$ to $0$ as the vanishing path for the critical points over zero, and denote the corresponding vanishing cycles by $\vc{\xt\fact}{m}$ and $\vc{\xt\yt}{}$.  We also choose the same preliminary vanishing paths $\gamma_{l,m}$ as before, but with $\theta_{l,m}$ now given by
\[
\theta_{l,m} = 2\pi \left(\frac{l}{p-1}+\frac{pm}{(p-1)(q-1)}\right),
\]
and write $\vcpr{0}{l,m}$ for the preliminary vanishing cycles.

\subsection{The vanishing cycles}

The central fibre $\wt_\eps^{-1}(0)$, shown in \cref{figZeroFibreChain}, now has only two components, namely the line $\{\xt=0\}$ and the smooth curve $\{\fact = \eps \yt\}$.  The $q$ nodes are smoothed to thin necks in $\Sigma$, whose complement we again refer to as $\Sigma'$, and we trivialise the fibration $\wt_\eps$ on this complement over the disc of radius $\delta$.
\begin{figure}[ht]
\centering
\begin{tikzpicture}[blob/.style={cross out, draw=black, fill=black, inner sep=0, minimum size=\blobsize, line width=0.5mm}, noblob/.style={inner sep=0, minimum size=0}]
\def\blobsize{1.5mm}

\begin{scope}[xshift=-0.2cm]
\draw[fill=white, opacity=0.8, name path global=plane1] (-1.5, -2) -- (2, -0.5) -- (2, 5.5) -- (-1.5, 4) -- cycle;
\end{scope}

\draw[fill=white, fill opacity=0.8] plot [smooth, tension=0.7] coordinates {(3.5, 4.5) (3, 4.2) (1, 3.9) (0, 3.5) (1, 3) (0.2, 2.5) (1, 2) (0, 1.5) (1.25, 0.75) (0, 0) (1, -0.2) (2, 0) (3.5, 0.5) (6.5, 1)};

\draw plot [smooth, tension=1] coordinates {(4.5, 4.5) (4, 3.8) (5.5, 4.5)};
\draw plot [smooth, tension=1] coordinates {(6.4, 4.5) (4.5, 2.8) (6.5, 3.2)};
\draw plot [smooth, tension=1] coordinates {(6.5, 2.5) (5, 1.6) (6.5, 1.7)};

\begin{scope}[xshift=3.4cm, yshift=2.2cm, rotate=-20]
\draw (0, 0) to [bend left=30] (1, 0);
\draw (-0.1, 0.075) to [bend right=10] (0, 0) to [bend right=20] (1, 0) to [bend right=10] (1.1, 0.075);
\end{scope}

\begin{scope}[xshift=1.6cm, yshift=3.5cm, rotate=-20]
\draw[name path global=curve1] (0, 0) to [bend left=30] (1, 0);
\draw[name path global=curve2] (-0.1, 0.075) to [bend right=10] (0, 0) to [bend right=20] (1, 0) to [bend right=10] (1.1, 0.075);

\draw[name intersections={of=plane1 and curve1}] (intersection-1) node[noblob](pt1){};
\draw[name intersections={of=plane1 and curve2}] (intersection-1) node[noblob](pt2){};
\draw (pt1) -- (pt2);
\end{scope}

\begin{scope}[xshift=2cm, yshift=1.2cm, rotate=-20]
\draw (0, 0) to [bend left=30] (1, 0);
\draw (-0.1, 0.075) to [bend right=10] (0, 0) to [bend right=20] (1, 0) to [bend right=10] (1.1, 0.075);
\end{scope}

\draw (-0.3, 0.7) node[anchor=east]{$\xt=0$};
\draw (2.1, 2.2) node{$\fact=\eps\yt$};

\draw (0, 1.5) node[blob]{};
\draw (0.2, 2.5) node[blob]{};
\draw (0, 3.5) node[blob]{};

\draw (0, 0) node[blob]{};

\draw [decorate,decoration={brace,amplitude=7pt}] (-0.05, 1.3) -- (-0.05, 3.7);
\draw (-0.28, 2.45) node[anchor=east]{\small$q-1$};

\end{tikzpicture}
\caption{The fibre $\wt_\eps^{-1}(0)$ for chain polynomials.\label{figZeroFibreChain}}
\end{figure}
This time we compute
\begin{gather*}
\# \text{ punctures of } \Sigma = \gcd(p-1, q)+1
\\ g(\Sigma) = \frac{1}{2}\left(pq-p+1-\gcd(p-1,q)\right).
\end{gather*}

Just as in the loop case, the preliminary cycle $\vcpr{0}{0,0}$ is given by the loop in the positive quadrant of the real part of $\Sigma$.  On $\Sigma'$ the other preliminary cycles are given by the action of $(\zeta^l\mu^m, \eta^m)$. In particular, they are pairwise disjoint on the $\{\fact = \eps \yt\}$ part of $\Sigma'$ (since $\xt$ and $\yt$ are both nowhere-zero here).  The only intersections on the $\{\xt=0\}$ part occur when the $m$-values coincide, and in this case the cycles overlap (at least in the limit $\delta \downarrow 0$) exactly as before.

On the $\vc{\xt\yt}{}$-neck region, the argument of the $\yt$-component of $\vcpr{0}{l,m}$ interpolates from
\[
-2\pi \left( \frac{l}{p-1} + \frac{m}{(p-1)(q-1)} \right) \quad \text{to} \quad \frac{2\pi m}{q-1}
\]
as $|\yt|$ increases, whilst on the $\vc{\xt\fact}{m}$-neck the argument of $\yt - \eta^m \yt^+_\mathrm{crit}$ interpolates back the other way as its argument decreases.  This is completely analogous to the picture in \cref{figlequalsL}.

We modify the preliminary paths, and correspondingly perturb the fibration, exactly as in \cref{ModifyingPaths}.  The chain polynomial version of \cref{lemCyclesDisjoint} is:

\begin{lem}
Suppose $\theta_{l,m} > \theta_{L,M} + 2\pi$, and let $z = \gamma'_{l,m} \cap \gamma'_{L,M}$.  Inside $\Sigma_z = \wt_\eps^{-1}(z)$ we have vanishing cycles $V_1$ and $V_2$ corresponding to the critical points $(\zeta^l\mu^m\xt^+_\mathrm{crit}, \eta^m\yt^+_\mathrm{crit})$ and $(\zeta^L\mu^M\xt^+_\mathrm{crit}, \eta^M\yt^+_\mathrm{crit})$ and the truncations of $\gamma'_{l,m}$ and $\gamma'_{L,M}$.  These cycles are disjoint.
\end{lem}
\begin{proof}
We must have $l \geq L$ and $m > M$, so we can apply $f_{L,M}^{-1}$ to get $(L,M) = (0,0)$ with $m > 0$.  The latter ensures that $V_1$ and $V_2$ are disjoint on $\Sigma' \subset \Sigma \approx \Sigma_z$, and that their only possible intersection is in the $\vc{\xt\yt}{}$-neck region.  On this region the argument of $\yt$ is approximately $0$ for $V_2$, and interpolates between
\[
2\pi \left( 1 - \frac{l}{p-1} - \frac{m}{(p-1)(q-1)} \right) \quad \text{and} \quad \frac{2\pi m}{q-1}
\]
for $V_1$, so they are disjoint there too.
\end{proof}

This allows us to introduce fingers to the vanishing paths $\gamma'_{l,m}$, as before, without affecting the vanishing cycles.  We then make Hamiltonian isotopies as in \cref{Isotopies} (but now only in the $\yt$-axis part of $\Sigma'$ and the $\vc{\xt\yt}{}$- and $\vc{\xt\fact}{m}$-necks) to obtain the final vanishing cycles.  This gives a model for $\mathcal{A}$ with the following basis of morphisms:
\begin{itemize}
\item An identity morphism for each object.
\item A morphism from $\vc{0}{l,m}$ to $\vc{0}{L,M}$ whenever $(l,m) \neq (L,M)$ but both $l \geq L$ and $m \geq M$.
\item A morphism from each $\vc{0}{l,m}$ to $\vc{\xt\yt}{}$ and to $\vc{\xt\fact}{m}$.
\end{itemize}
As in the loop case the differentials on morphism complexes trivially vanish so we are left to check compositions and gradings.

\subsection{Composition and gradings}

Once more we have one obvious triangle contributing to each non-trivial product, and by the same homology computation as for loop polynomials there can be no others.  We can also run the same inductive argument to ensure that all of the signs in the compositions are positive.

To grade the category we must again take the unique homotopy class of line field $\ell$ on $\Sigma$ whose winding number along each vanishing cycle $V$ is zero, and then pick a homotopy class of homotopy from $\ell|_V$ to $TV$.  By homotoping $\ell$ we may assume it points longitudinally in each neck region, orthogonal to the waist curves, and then up to homotopy it must look like the right-hand diagram in \cref{figLineField} in the union of the neck regions and the $\yt$-axis part of $\Sigma'$.  We can then define the gradings in the same way as in the loop case, and see that all morphisms then lie in degree $0$.

The conclusion is:

\begin{thm}[\cref{Thm1}, chain polynomial case]
Under the correspondence
\[
\begin{aligned}
\vc{0}{l,m} &\leftrightarrow \obj{0}{i,j}
\\ \vc{\xt\fact}{m} &\leftrightarrow \obj{y}{j}[3]
\\ \vc{\xt\yt}{} &\leftrightarrow \obj{\fac}{}[3]
\end{aligned}
\quad \text{with} \quad
\begin{aligned}
i+l&=p-1
\\ j+m&=q-1
\end{aligned}
\]
the $\Z$-graded $A_\infty$-category $\mathcal{A}$ is described by the quiver with relations in \cref{figChainQuiver} and is formal, so there is a quasi-equivalence
\begin{flalign*}
&& \mathrm{mf}(\C^2, \Gamma_\w, \w) &\simeq \mathcal{F}(\wt). & \qed
\end{flalign*}
\end{thm}

This was also proved by Futaki--Ueda for $q=2$, as a special case of \cite[Theorem 1.2]{FutakiUedaD}.  They state the result at the level of derived categories, i.e.~after passing to cohomology, but as we have seen it is trivial to upgrade from this to the full $A_\infty$ result.


\section{Brieskorn--Pham polynomials}
\label{BrieskornPham}

\subsection{B-model}

Now $\w$ is given by $x^p + y^q$, and the maximal grading group $L$ is generated by $\vec{x}$, $\vec{y}$ and $\vec{c}$ modulo
\[
p\vec{x} = q\vec{y} = \vec{c},
\]
so is simply $\Z/p \oplus \Z/q$, generated by $\vec{x} = (1, 0)$ and $\vec{y} = (0, 1)$.  Let $S = k[x,y]$, graded by $L$ in the obvious way, and let $R = S/(\w)$.

The stack $[\w^{-1}(0)/\Gamma_\w]$ has only one component this time, and the objects that we need are the matrix factorisations $\obj{0}{i,j}$ given by
\begin{center}
	\begin{tikzcd}[row sep=7ex, column sep=10ex]
		S \arrow{r}{y^j} \arrow{dr}[near start, outer sep=-2pt]{-x^i} \ar[d, phantom, description, "\cdots\hskip7ex\bigoplus\phantom{\hskip7ex\cdots}"]
		&  S(j\vec{y}) \arrow{r}{y^{q-j}} \arrow{dr}[near start]{x^i} \ar[d, phantom, description, "\bigoplus"]
		& S(\vec{c}) \ar[d, phantom, description, "\phantom{\cdots\hskip7ex}\bigoplus\hskip7ex\cdots"]
		\\ S(-\vec{c}+i\vec{x}+j\vec{y}) \arrow{ur}[near start, outer sep=-2pt]{x^{p-i}} \arrow{r}{y^{q-j}}
		& S(i\vec{x}) \arrow{ur}[near start, outer sep=-2pt]{-x^{p-i}} \arrow{r}{y^j}
		& S(i\vec{x}+j\vec{y})
	\end{tikzcd}
\end{center}
for $i=1, \dots, p-1$ and $j=1, \dots, q-1$, stabilising $R(i\vec{x}+j\vec{y})/(x^i,y^j)$.

For any $(i,j)$ and $(I,J)$ the morphism space $\Hom^\bullet(\obj{0}{i,j}, \obj{0}{I,J})$ is computed by the complex
\begin{center}
	\begin{tikzcd}[row sep=7ex, column sep=10ex]
		(R/(x^I, y^J))_{-\vec{c}+I\vec{x}+J\vec{y}} \arrow{r}{y^{q-j}}\arrow{dr}[near start, outer sep=-2pt]{-x^{p-i}} \ar[d, phantom, description, "\cdots\hskip7ex\bigoplus\phantom{\hskip7ex\cdots}"]
		& (R/(x^I, y^J))_{I\vec{x}+(J-j)\vec{y}} \arrow{r}{y^j} \arrow{dr}[near start]{x^{p-i}} \ar[d, phantom, description, "\bigoplus"]
		& (R/(x^I, y^J))_{I\vec{x}+J\vec{y}} \ar[d, phantom, description, "\phantom{\cdots\hskip7ex}\bigoplus\hskip7ex\cdots"]
		\\ (R/(x^I, y^J))_{(I-i)\vec{x}+(J-j)\vec{y}} \arrow{ur}[near start, outer sep=-1pt]{x^i} \arrow{r}{y^j}
		& (R/(x^I, y^J))_{(I-i)\vec{x}+J\vec{y}} \arrow{ur}[near start, outer sep=-1pt]{-x^i} \arrow{r}{y^{q-j}}
		& (R/(x^I, y^J))_{\vec{c}+(I-i)\vec{x}+(J-j)\vec{y}}
	\end{tikzcd}
\end{center}
By considering gradings modulo $\vec{x}$ and modulo $\vec{y}$, one sees that the top row and the odd position terms in the bottom row vanish, and the remaining terms vanish if $I<i$ or $J<j$.  We therefore assume that $I\geq i$ and $J\geq j$, and read off that $\Hom^{2m+1}(\obj{0}{i,j}, \obj{0}{I,J})=0$ and
\[
\Hom^{2m}(\obj{0}{i,j}, \obj{0}{I,J}) \cong (R/(x^I, y^J))_{(I-i)\vec{x}+(J-j)\vec{y}}.
\]
Arguing as in the loop and chain cases, this is spanned by $x^{I-i}y^{J-j}$.

The full $A_\infty$-subcategory of $\mathrm{mf}(\C^2, \Gamma_\w, \w)$ on the objects $\obj{0}{i,j}$ is therefore described by the quiver with relations in \cref{figBrieskornPhamQuiver}, and is formal as before.
\begin{figure}[ht]
\centering
\begin{tikzpicture}[blob/.style={circle, draw=black, fill=black, inner sep=0, minimum size=\blobsize}, arrow/.style={->, shorten >=6pt, shorten <=6pt}]
\def\blobsize{1.2mm}

\draw (0, 0) node[blob]{};
\draw (1, 0) node[blob]{};
\draw (2, 0) node[blob]{};
\draw (3.25, 0) node{\ $\cdots$};
\draw (4.5, 0) node[blob]{};

\draw[arrow] (0, 0) -- (1, 0);
\draw[arrow] (1, 0) -- (2, 0);
\draw[arrow] (2, 0) -- (3, 0);
\draw[arrow] (3.5, 0) -- (4.5, 0);

\draw[arrow] (0, 0) -- (0, 1);
\draw[arrow] (1, 0) -- (1, 1);
\draw[arrow] (2, 0) -- (2, 1);
\draw[arrow] (4.5, 0) -- (4.5, 1);

\begin{scope}[yshift=1cm]
\draw (0, 0) node[blob]{};
\draw (1, 0) node[blob]{};
\draw (2, 0) node[blob]{};
\draw (3.25, 0) node{\ $\cdots$};
\draw (4.5, 0) node[blob]{};

\draw[arrow] (0, 0) -- (1, 0);
\draw[arrow] (1, 0) -- (2, 0);
\draw[arrow] (2, 0) -- (3, 0);
\draw[arrow] (3.5, 0) -- (4.5, 0);

\draw[arrow] (0, 0) -- (0, 1);
\draw[arrow] (1, 0) -- (1, 1);
\draw[arrow] (2, 0) -- (2, 1);
\draw[arrow] (4.5, 0) -- (4.5, 1);
\end{scope}

\begin{scope}[yshift=3.5cm]
\draw (0, 0) node[blob]{};
\draw (1, 0) node[blob]{};
\draw (2, 0) node[blob]{};
\draw (3.25, 0) node{\ $\cdots$};
\draw (4.5, 0) node[blob]{};

\draw[arrow] (0, 0) -- (1, 0);
\draw[arrow] (1, 0) -- (2, 0);
\draw[arrow] (2, 0) -- (3, 0);
\draw[arrow] (3.5, 0) -- (4.5, 0);
\end{scope}

\draw (0, 2.37) node{$\vdots$};
\draw (1, 2.37) node{$\vdots$};
\draw (2, 2.37) node{$\vdots$};
\draw (4.5, 2.37) node{$\vdots$};

\begin{scope}[yshift=2.5cm]
\draw[arrow] (0, 0) -- (0, 1);
\draw[arrow] (1, 0) -- (1, 1);
\draw[arrow] (2, 0) -- (2, 1);
\draw[arrow] (4.5, 0) -- (4.5, 1);
\end{scope}

\draw (3.25, 2.37) node{\ $\iddots$};

\draw[line width=1mm, opacity=0.2] (-0.5, -0.5) rectangle (5, 4);

\draw (-1.2, 2) node{$\obj{0}{i,j}$};

\draw (8, 2) node{\parbox{115pt}{\small \ \textbf{Relations:} \begin{enumerate}[(i)]\item Squares commute\end{enumerate}}};

\end{tikzpicture}
\caption{The quiver describing the category $\mathcal{B}$ for Brieskorn--Pham polynomials.\label{figBrieskornPhamQuiver}}
\end{figure}
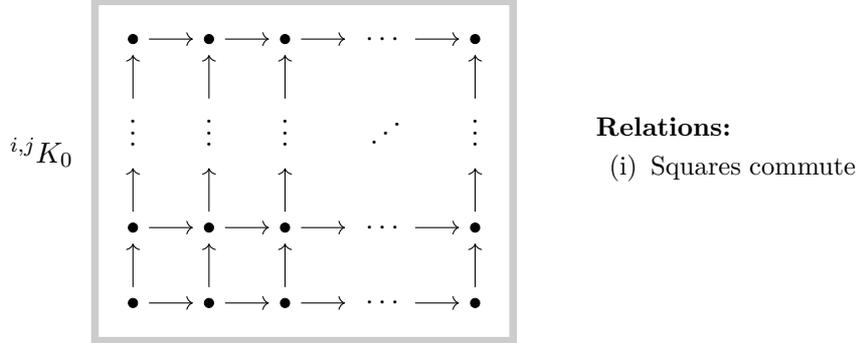
This is the tensor product of the $A_{p-1}$ and $A_{q-1}$ quivers, which describe the one-variable graded matrix factorisations of $x^p$ and $y^q$ respectively.

To prove these objects generate we just need to check that we can build all $L/\Z\vec{c}$-shifts of $R/(x,y)$ from them.  One easily constructs $R(a\vec{x}+b\vec{y})/(x,y)$ for $a=1, \dots, p-1$, $b=1, \dots, q-1$ by taking cones on these generators as in the previous cases.  To construct the remaining shifts, note that the modules $R(i\vec{x}+j\vec{y})/(x^i,y^j)$ and $R(\vec{c})/(x^{p-i},y^{q-j})$ are isomorphic in the singularity category, as they give rise to equivalent matrix factorisations.  Taking $i=p-1$ and $j=q-1, q-2, \dots, 1$ in turn, we can inductively build the $a=0$ shifts from $R(\vec{c})/(x^{p-i}, q^{q-j})$.  Reversing the roles of $x$ and $y$ gives the remaining shifts.  In contrast to \cref{rmkGenVsSplit} and \cref{rmkGenVsSplitChain}, the $R(l)/(x,y)$ now generate the category, rather than just split-generate.

We conclude the following well-known result, which goes back to at least \cite[Theorem 6]{FutakiUedaBrieskornPhamProceedings}, \cite[Theorem 1.2]{FutakiUedaBrieskornPhamJournal}:

\begin{thm}[\cref{Thm2}, Brieskorn--Pham polynomial case]
The object
\[
\mathcal{E} \coloneqq \bigoplus_{\substack{i=1, \dots, p-1\\j=1, \dots, q-1}} \obj{0}{i,j}
\]
is a tilting object for $\mathrm{mf}(\C^2, \Gamma_\w, \w)$.\hfill$\qed$
\end{thm}


\subsection{A-model}
\label{BPAModel}

We consider the resonant Morsification $\wt_\eps \coloneqq \xt^p + \yt^q - \eps \xt\yt$ of the Berglund--H\"ubsch transpose $\wt = \xt^p + \yt^q$.  The critical points are:
\begin{enumerate}[(i)]
\item $\xt=\yt=0$
\item\label{crit2BP} $\xt^{p-1} = \frac{\eps \yt}{p}$, $\yt^{q-1} = \frac{\eps \xt}{q}$.
\end{enumerate}
These are Morse, with critical values $0$ and $-\xt\yt\eps(pq-p-q)/pq$ respectively.  The equations (\ref{crit2BP}) reduce to
\[
\xt^{(p-1)(q-1)-1} = \frac{\eps^q}{p^{q-1}q} \quad \text{and} \quad \yt = \frac{p\xt^{p-1}}{\eps},
\]
so there is a unique positive real solution $(\xt^+_\mathrm{crit}, \yt^+_\mathrm{crit})$ whose critical value we denote by $c_\mathrm{crit}$ as before.  All other critical points differ by the action of $(pq-p-q)$th roots of unity with weights $(q-1, 1)$, or equivalently $(1, p-1)$, on $(\xt,\yt)$.  We parametrise these critical points, and the associated vanishing paths and cycles, by
\[
(l,m) \in \left(\{0, \dots, p-2\} \times \{0, \dots, q-2\}\right) \setminus \{(p-2, q-2)\}
\]
as
\[
(\mu^{(q-1)l+m} \xt^+_\mathrm{crit}, \mu^{l+(p-1)m} \yt^+_\mathrm{crit}),
\]
where $\mu = e^{2\pi i/(pq-p-q)}$.

The fibre $\wt_\eps^{-1}(0)$ is shown in \cref{figZeroFibreBP}.
\begin{figure}[ht]
\centering
\begin{tikzpicture}[blob/.style={cross out, draw=black, fill=black, inner sep=0, minimum size=\blobsize, line width=0.5mm}, noblob/.style={inner sep=0, minimum size=0}]
\def\blobsize{1.5mm}

\begin{scope}[yshift=2cm]
\draw plot [smooth, tension=0.6] coordinates {(5, 2.5) (4, 2) (1.3, 1.5) (0.2, 1.2) (-0.2, 0.7) (0,0) (0.6, 0.5) (1.5,0) (0.6, -0.5) (0, 0) (-0.2, -0.7) (0.2, -1.2) (1.3, -1.2) (4, -0.5) (6.5, -0.5)};
\end{scope}

\draw plot [smooth, tension=1] coordinates {(5.8, 4.5) (4.6, 3.3) (6.5, 3.9)};
\draw plot [smooth, tension=1] coordinates {(6.5, 3.2) (5, 2.4) (6.5, 2.2)};

\begin{scope}[xshift=1.7cm, yshift=3.2cm, rotate=-20]
\draw (0, 0) to [bend left=30] (1, 0);
\draw (-0.1, 0.075) to [bend right=10] (0, 0) to [bend right=20] (1, 0) to [bend right=10] (1.1, 0.075);
\end{scope}

\begin{scope}[xshift=3.5cm, yshift=3.1cm, rotate=-20]
\draw (0, 0) to [bend left=30] (1, 0);
\draw (-0.1, 0.075) to [bend right=10] (0, 0) to [bend right=20] (1, 0) to [bend right=10] (1.1, 0.075);
\end{scope}

\begin{scope}[xshift=2.8cm, yshift=2.4cm, rotate=-20]
\draw (0, 0) to [bend left=30] (1, 0);
\draw (-0.1, 0.075) to [bend right=10] (0, 0) to [bend right=20] (1, 0) to [bend right=10] (1.1, 0.075);
\end{scope}

\draw (2.3, 1.7) node{$\fact=\eps\xt\yt$};

\draw (0, 2cm) node[blob]{};

\end{tikzpicture}
\caption{The fibre $\wt_\eps^{-1}(0)$ for Brieskorn--Pham polynomials.\label{figZeroFibreBP}}
\end{figure}
This time it is irreducible.  At infinity the defining equation looks like $\xt^p+\yt^q=0$ so the smooth fibre $\Sigma = \wt_\eps^{-1}(-\delta)$ satisfies
\begin{gather*}
\# \text{ punctures of }\Sigma = \gcd(p,q)
\\ g(\Sigma) = \frac{1}{2}\left((p-1)(q-1)-\gcd(p,q)+1\right).
\end{gather*}

We divide $\Sigma$ into $\Sigma'$ and a single neck region, and trivialise the fibration on $\Sigma'$ over a small disc.  We define preliminary vanishing paths and cycles $\vc{\xt\yt}{}$ and $\vcpr{0}{l,m}$ as usual, taking
\[
\theta_{l,m} = \frac{2\pi(ql + pm)}{pq-p-q}.
\]
Note that by our bounds on $l$ and $m$ this lies in $[0, 4\pi)$.  The cycle $\vc{\xt\yt}{}$ is the waist curve on the neck, whilst $\vcpr{0}{0,0}$ lives in the positive quadrant of the real part of $\Sigma$.  The other $\vcpr{0}{l,m}$ are obtained from $\vcpr{0}{0,0}$ by the action of roots of unity on $\Sigma'$ and by a local parallel transport computation on the neck.  In particular, all intersections between the vanishing cycles occur on the neck.  We modify the vanishing paths (and correspondingly perturb the fibration), introducing fingers to remove their intersections, in the familiar way.

There is now no need to isotope the cycles further, since they are already all transverse.  In particular, on the $\vc{\xt\yt}{}$-neck the argument of $\xt$ along $\vcpr{0}{l,m}$ interpolates from
\[
-2\pi\frac{l+(p-1)m}{pq-p-q} \quad \text{to} \quad 2\pi\frac{(q-1)l+m}{pq-p-q}
\]
as its modulus increases.  The intersection pattern is thus described by the morphisms in the quiver \cref{figBrieskornPhamQuiver2}, in the sense that the number of intersections between two curves is the dimension of the corresponding morphism space; the $l$ and $m$ indices decrease from bottom left to top right.
\begin{figure}[ht]
\centering
\begin{tikzpicture}[blob/.style={circle, draw=black, fill=black, inner sep=0, minimum size=\blobsize}, arrow/.style={->, shorten >=6pt, shorten <=6pt}]
\def\blobsize{1.2mm}

\draw (1, 0) node[blob]{};
\draw (2, 0) node[blob]{};
\draw (3.25, 0) node{\ $\cdots$};
\draw (4.5, 0) node[blob]{};

\draw[arrow] (1, 0) -- (2, 0);
\draw[arrow] (2, 0) -- (3, 0);
\draw[arrow] (3.5, 0) -- (4.5, 0);

\draw[arrow] (1, 0) -- (1, 1);
\draw[arrow] (2, 0) -- (2, 1);
\draw[arrow] (4.5, 0) -- (4.5, 1);

\begin{scope}[yshift=1cm]
\draw (0, 0) node[blob]{};
\draw (1, 0) node[blob]{};
\draw (2, 0) node[blob]{};
\draw (3.25, 0) node{\ $\cdots$};
\draw (4.5, 0) node[blob]{};

\draw[arrow] (0, 0) -- (1, 0);
\draw[arrow] (1, 0) -- (2, 0);
\draw[arrow] (2, 0) -- (3, 0);
\draw[arrow] (3.5, 0) -- (4.5, 0);

\draw[arrow] (0, 0) -- (0, 1);
\draw[arrow] (1, 0) -- (1, 1);
\draw[arrow] (2, 0) -- (2, 1);
\draw[arrow] (4.5, 0) -- (4.5, 1);
\end{scope}

\begin{scope}[yshift=3.5cm]
\draw (0, 0) node[blob]{};
\draw (1, 0) node[blob]{};
\draw (2, 0) node[blob]{};
\draw (3.25, 0) node{\ $\cdots$};
\draw (4.5, 0) node[blob]{};

\draw[arrow] (0, 0) -- (1, 0);
\draw[arrow] (1, 0) -- (2, 0);
\draw[arrow] (2, 0) -- (3, 0);
\draw[arrow] (3.5, 0) -- (4.5, 0);
\end{scope}

\draw (0, 2.37) node{$\vdots$};
\draw (1, 2.37) node{$\vdots$};
\draw (2, 2.37) node{$\vdots$};
\draw (4.5, 2.37) node{$\vdots$};

\begin{scope}[yshift=2.5cm]
\draw[arrow] (0, 0) -- (0, 1);
\draw[arrow] (1, 0) -- (1, 1);
\draw[arrow] (2, 0) -- (2, 1);
\draw[arrow] (4.5, 0) -- (4.5, 1);
\end{scope}

\draw (3.25, 2.37) node{\ $\iddots$};

\draw[line width=1mm, opacity=0.2] (-0.5, -0.5) rectangle (5, 4);

\draw (-1.2, 2) node{$\vc{0}{l,m}$};

\draw (8, 2) node{\parbox{115pt}{\small \ \textbf{Relations:} \begin{enumerate}[(i)]\item Squares commute\end{enumerate}}};

\draw[arrow] (4.5, 3.5) -- (5.5, 4.5);
\draw (5.5, 4.5) node[blob]{};
\draw (6.1, 4.5) node{$\vc{\xt\yt}{}$};

\end{tikzpicture}
\caption{The quiver describing the intersection pattern.\label{figBrieskornPhamQuiver2}}
\end{figure}
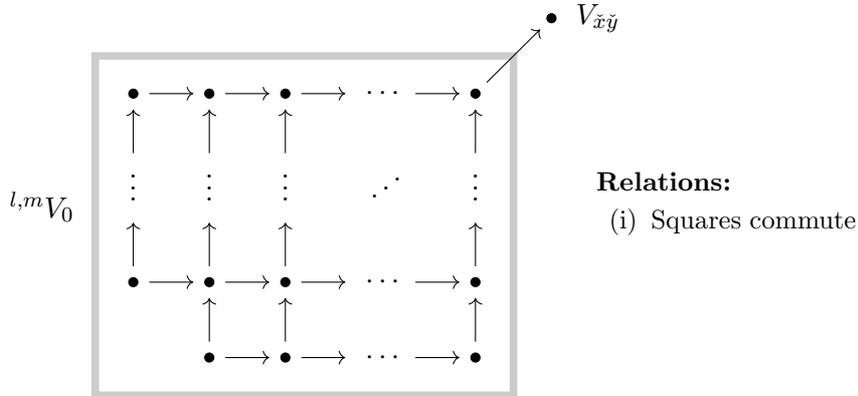
This is not quite the pattern we want, but this can be rectified as follows.  Recall that the ordering of the cycles is determined by the clockwise ordering of the directions of their vanishing paths as they emanate from the reference base point $-\delta$.  We have so far been starting the ordering from the direction $e^{i\theta}$ for $0 < \theta \ll 2\pi$, but we now change this to $e^{-i\theta}$.  This has the effect of moving $\vc{\xt\yt}{}$ from last to first in the ordering, and hence modifying the quiver from \cref{figBrieskornPhamQuiver2} to \cref{figBrieskornPhamQuiver}.

\begin{rmk}
Alternatively, one can leave the starting direction as $e^{i\theta}$ and instead replace the indexing set
\[
\left(\{0, \dots, p-2\} \times \{0, \dots, q-2\}\right) \setminus \{(p-2, q-2)\},
\]
over which $(l,m)$ ranges, by
\[
\left(\{0, \dots, p-2\} \times \{0, \dots, q-2\}\right) \setminus \{(0, 0)\}.
\]
This moves the top right vertex inside the rectangle in \cref{figBrieskornPhamQuiver2} to the bottom left.  Now $\theta_{l,m}$ lies in $(0, 4\pi]$, rather than $[0, 4\pi)$, so the prescription given at the start of \cref{ModifyingPaths} has to be modified so that $\gamma'_{l,m}$ is described in modulus-(argument$+\pi$) space by the piecewise linear path:
\begin{itemize}
\item From $(\delta, 0)$ to $(\delta+\delta', \theta_{l,m})$ to $(-c_\mathrm{crit}, \theta_{l,m})$ for some small positive $\delta'$, if $\theta_{l,m} \leq 2\pi$.
\item From $(\delta, 0)$ to $(\delta+\delta', 2\pi+\lambda\theta_{l,m})$ to $(\delta+2\delta', 2\pi+\lambda\theta_{l,m})$ to $(\delta+3\delta', \theta_{l,m}-\theta'))$ to $(-c_\mathrm{crit}, \theta_{l,m}-\theta')$ for some small positive $\lambda$ and $\theta'$, if $\theta_{l,m} > 2\pi$.
\end{itemize}
Note that the inequalities $< 2\pi$ and $\geq 2\pi$ have become $\leq 2\pi$ and $> 2\pi$, whilst the $\lambda(\theta_{l,m}-4\pi)$ terms have become $\lambda\theta_{l,m}$, so that the short horizontal segments in \cref{figGammap} are pushed slightly above the dashed $2\pi$ line.
\end{rmk}

Compositions are non-degenerate by the standard argument, and we can arrange all signs to be positive.  To fix gradings we take the unique homotopy class of line field $\ell$ on $\Sigma$ with respect to which all vanishing cycles are gradable.  We may assume $\ell$ is longitudinal on the neck, and equip the $\vc{0}{l,m}$ with the standard gradings (we choose the lift $\alpha^\#$ to be approximately between $0$ and $1/2$).  We previously gave $\vc{\xt\yt}{}$ the grading with $\alpha^\# = -1/2$, but now that we have changed the ordering we should choose $\alpha^\# = 1/2$ to put all morphisms in degree $0$.

We arrive at the following result Futaki--Ueda \cite[Theorem 5]{FutakiUedaBrieskornPhamProceedings}, \cite[Theorem 1.3]{FutakiUedaBrieskornPhamJournal}:

\begin{thm}[\cref{Thm1}, Brieskorn--Pham polynomial case]
Under the correspondence
\[
\begin{aligned}
\vc{0}{l,m} &\leftrightarrow \obj{0}{i,j}
\\ \vc{\xt\yt}{} &\leftrightarrow \obj{0}{1,1}
\end{aligned}
\quad \text{with} \quad
\begin{aligned}
i+l&=p-1
\\ j+m&=q-1
\end{aligned}
\]
the $\Z$-graded $A_\infty$-category $\mathcal{A}$ is described by \cref{figBrieskornPhamQuiver} and is formal, so there is a quasi-equivalence
\begin{flalign*}
&& \mathrm{mf}(\C^2, \Gamma_\w, \w) &\simeq \mathcal{F}(\wt). & \qed
\end{flalign*}
\end{thm}

\bibliography{LoopPolyBibliography}
\bibliographystyle{utcapsor2}

\end{document}